\documentclass[12pt,reqno]{amsart}
\usepackage{amsmath,amssymb,amsfonts,amscd,latexsym,amsthm,mathrsfs,verbatim,textcomp}
\usepackage[colorlinks,linkcolor=black,citecolor=black]{hyperref}
\usepackage{cite}
\usepackage{hypbmsec}
\usepackage{enumerate}
\usepackage{bm}
\usepackage{tikz}
\usetikzlibrary{matrix,arrows,decorations.pathmorphing}
\usepackage[margin=1in]{geometry}
\newtheorem{lemma}{Lemma}
\newtheorem{theorem}{Theorem}

\newtheorem{prop}{Proposition}
\theoremstyle{remark}
\newtheorem{remark}{\bf Remark}

\renewcommand{\Im}{\operatorname{Im}}
\renewcommand{\Re}{\operatorname{Re}}

\newcommand{\Z}{\mathbb{Z}}

\newcommand{\N}{\mathbb{N}}

\usepackage{etoolbox}
\patchcmd{\section}{\scshape}{\bfseries}{}{}
\makeatletter
\renewcommand{\@secnumfont}{\bfseries}
\makeatother
\numberwithin{equation}{section}
\numberwithin{lemma}{section}
\numberwithin{theorem}{section}
\numberwithin{prop}{section}
\numberwithin{remark}{section}
\begin{document}

\title{The twisted second moment of modular half integral weight $L$--functions}

\author{Alexander Dunn}
\address{School of Mathematics, Georgia Institute of Technology, Atlanta, USA}
\email{adunn61@gatech.edu}

\author{Alexandru Zaharescu}
\address{Department of Mathematics, University of Illinois, 1409 West Green
Street, Urbana, IL 61801, USA and Simon Stoilow Institute of Mathematics of the Romanian Academy, P.O. Box 1-764, RO-014700 Bucharest, Romania}
\email{zaharesc@illinois.edu}
\subjclass[2010]{Primary 11F66, 11L05, 11L07. Secondary 11F72, 11L15.}
\keywords{Kloosterman and Sali\'{e} sums, half-integer weight automorphic forms, $L$-functions, twisted second moment, shifted convolution, local spacing statistics}

\maketitle

\begin{abstract}
Given a half-integral weight holomorphic Kohnen newform $f$ on $\Gamma_0(4)$, we prove an asymptotic formula for large primes $p$ with power saving error term for 
\begin{equation*}
\sideset{}{^*} \sum_{\chi \hspace{-0.15cm} \pmod{p}}  | L(1/2,f,\chi) |^2.
\end{equation*}
Our result is unconditional, it does not rely on the Ramanujan--Petersson conjecture for the form $f$. This gives a very sharp Lindel\"{o}f on average result for Dirichlet series attached to Hecke eigenforms  without an Euler product. The Lindel\"{o}f hypothesis for such series was originally conjectured by Hoffstein. There are two main inputs. The first is a careful spectral analysis of a highly unbalanced shifted convolution problem involving the Fourier coefficients of half-integral weight forms. The second input is a bound for sums of products of Sali\'{e} sums in the Polya--Vinogradov range. Half--integrality is fully exploited to establish such an estimate. We use the closed form evaluation of the Sali\'{e} sum to relate our problem to the sequence $\alpha n^2 \pmod{1}$.  Our treatment of this sequence is inspired by work of Rudnick--Sarnak and the second author on the local spacings of $\alpha n^2$ modulo one. 

\end{abstract}
\tableofcontents
\section{Introduction and statement of results}
Moments of $L$-functions play a central role in analytic number theory. Classical examples include the fourth moment of Riemann zeta
\begin{equation*}
\int_{0}^T |\zeta(1/2+it) |^4 dt=T P_4(\log T)+O_{\varepsilon} \big(T^{\frac{2}{3}+\varepsilon} \big),
\end{equation*}
for a certain polynomial $P_4$ (see \cite{IM,M1,Zav}), and the cuspidal analogue due to Good \cite{Go}
\begin{equation*}
\int_{0}^T |L(1/2+it,f) |^2 dt=T P_1(\log T)+O_{\varepsilon} \big(T^{\frac{2}{3}+\varepsilon} \big),
\end{equation*}
for a certain polynomial $P_1$ depending on $f$. 

The complexity of a moment computation for a family $\mathcal{F}$ of $L$-functions is measured by the quotient $r=\log \mathcal{C}/\log |\mathcal{F}|$, where $\mathcal{C}$ is the analytic conductor of each function in the family. The edge of current technology where one can hope to obtain an asymptotic with power saving error term is $r=4$. Results in the case $r=4$ can be found in a host of works, including Iwaniec--Sarnak \cite{IS}, Kowalski--Michel--VanderKam \cite{KMV} and Blomer \cite{B1}. 

From an adelic point of view, it is natural to replace the archimedean twist $|\det |^{it}$ with a non-archimedean twist by a Dirichlet character $\chi$. Let $p>2$ be prime, $\psi(p):=p-2$ denote the number of primitive characters modulo $p$ and 
\begin{equation*}
L(s,\chi):=\sum_{n=1}^{\infty} \frac{\chi(n)}{n^s}, \quad \Re s>1,
\end{equation*}
be the usual Dirichlet $L$--function. Young's breakthrough paper \cite{Y2} in 2011 proved for any $\varepsilon>0$ that 
\begin{equation*}
\sideset{}{^*} \sum_{\substack{ \chi \hspace{-0.15cm} \pmod{p} }} |L(1/2,\chi)|^4=\psi(p) P_4(\log p)+O_{\varepsilon} (p^{1-\frac{1}{80}(1-2 \theta)+\varepsilon}),
\end{equation*}
where $\psi(p):=p-2$ is the number of primitive Dirichlet characters modulo $p$,
$P_4$ is a degree four polynomial and $\theta=\frac{7}{64}$ is the best known exponent toward the Ramanujan--Petersson conjecture (due to Kim and Sarnak \cite{Ki}) for Maass forms. The fourth moment of Dirichlet $L$-functions (for a general modulus $q \not \equiv 2 \pmod{4}$) is a special case of the more general moment 
\begin{equation} \label{mixedsecond}
\sideset{}{^*} \sum_{\chi \hspace{-0.15cm} \pmod{q}} L(1/2,f \otimes \chi) \overline{L(1/2,g \otimes \chi)},
\end{equation}
where $f,g$ are two fixed integral weight Hecke eigenforms (either holomorphic, Maass or Eisenstein) and could be either cuspidal or non--cuspidal. Here, $\{\lambda_f(n)\}_{n \geq 1}$ denotes the  
system of Hecke eigenvalues attached to $f$ and 
\begin{equation} \label{Linteg}
L(s,f \otimes \chi):=\sum_{n=1}^{\infty} \frac{\lambda_f(n) \chi(n)}{n^s}, \quad \Re s>1.
\end{equation}
Note that \eqref{Linteg} has an Euler product when the weight of the form is integral. Striking progress has been made on the moment \eqref{mixedsecond} in a sequence of works due to Blomer--Fouvry--Kowalski--Michel--Mili\'{c}evi\'{c}--Sawin \cite{BM1,BFKMM,KMS,BFKMMS}. An asymptotic for \eqref{mixedsecond} (in the case $f=g$) with power saving error term appears in \cite[Theorem~1.17]{BFKMMS}. The same family of twisted $L$-functions had also been previously studied in various contexts. One can see Chinta \cite{Chin}, Duke--Friendlander--Iwaniec \cite{DFI2}, Gao--Khan--Ricotta \cite{GKR}, Stefanicki \cite{Stef} and Hoffstein--Lee \cite{HLe}. 

We would also like to highlight the recent 2022 breakthrough work of Li \cite{XLi} that proves an asymptotic for the twisted second moment over the family of primitive quadratic Dirichlet characters (with fractional logarithmic power saving error term). This improved a result of Soundararajan and Young \cite{SouYou} that was conditional on the Generalised Riemann Hypothesis.

In this work we focus on the half integral weight analogue of \eqref{mixedsecond} when $f=g$. To enable subsequent discussion and introduce our results, we require some notation. More details are provided below in Section \ref{holoaut}. For $j \in \mathbb{N}$, let $k:=\frac{1}{2}+2j$ be an odd half integer. Suppose $f: \mathbb{H} \rightarrow \mathbb{C}$ is holomorphic, vanishes at all three cusps of $\Gamma_0(4)$, and satisfies
\begin{equation*}
f(\gamma \tau)=\nu_{\theta}(\gamma) (c \tau+d)^k f(\tau) \quad \text{for all} \quad \gamma \in \Gamma_0(4),
\end{equation*}
where $\nu_{\theta}$ is the standard theta multiplier on $\Gamma_0(4)$. Let $\mathcal{S}_k(4)$ denote this space of cusp forms.

Let the Fourier expansion of $f$ at $\infty$ be given by
\begin{equation} \label{fourier}
f(\tau):=\sum_{n=1}^{\infty} b(n) e(n \tau)=\sum_{n=1}^{\infty} a(n) n^{\frac{k-1}{2}} e(n \tau).
\end{equation}
For a prime $p>2$ and primitive character $\chi$ modulo $p$, define the twisted form 
\begin{equation*}
f_{\chi}(\tau):=\sum_{n=1}^{\infty} \chi(n) a(n) n^{\frac{k-1}{2}} e(n \tau),
\end{equation*} 
of level $4p^2$. The twisted $L$--function is given by the Dirichlet series
\begin{equation} \label{dirich}
L(s,f,\chi):=\sum_{n=1}^{\infty} \frac{a(n) \chi(n)}{n^s}, \quad \Re s>1.
\end{equation}
We have used a slightly different notation here to distinguish with the integral weight case discussed previously. Taking the Mellin transform of \eqref{dirich} one obtains a completed $L$--function of degree $2$ that has both meromorphic continuation to all of $\mathbb{C}$ (in fact holomorphic, because $f$ is cuspidal) and a functional equation, but is without an Euler product. The coefficients $a(n)$ are no longer multiplicative, except at squares. 

For odd primes $q$, the Hecke operators $\mathcal{T}_{q^2}$ defined on $\mathcal{S}_k(4)$ (with $k=\frac{1}{2}+2j$) are given by 
\begin{equation*}
\mathcal{T}_{q^2} f(\tau):=\sum_{n \geq 1} \Big( b(q^2n)+\Big(\frac{n}{q} \Big) q^{k-\frac{3}{2}} b(n)+q^{2k-2} b \Big(\frac{n}{q^2} \Big) \Big) e(n \tau).
\end{equation*}
Here we have used the convention that $b(x)=0$ unless $x \in \mathbb{Z}$. We call a half-integral weight cusp form a Hecke cusp form if $\mathcal{T}_{q^2} f=\lambda(q) f$ for all $q>2$. One of the main tools for understanding half integral weight forms and their coefficients is the Shimura lift \cite{Shim1}. Following Kohnen--Zagier \cite{Koh}, we focus on \emph{Kohnen's plus subspace}. The behaviour of these forms under the Shimura lift is well understood. The Kohnen plus space $S^{+}_k(4)$ (when $k=\frac{1}{2}+2j$) is the subspace of $\mathcal{S}_k(4)$ consisting of forms whose Fourier coefficients satisfy 
\begin{equation} \label{kohnen}
b(n)=0 \quad \text{unless} \quad n \equiv 0,1 \pmod{4}.
\end{equation} 
This space has a basis consisting of simultaneous eigenfunctions of the $\mathcal{T}_{q^2}$ for odd $q$. As $k \rightarrow \infty$, asymptotically one third of half integral weight cusp forms lie in Kohnen's plus space by dimension considerations. Given a Hecke cusp form $f \in S^{+}_k(4)$, one can normalise it so that its coefficients are totally real algebraic numbers \cite{Stev}. Let $d$ be a fundamental discriminant and 
\begin{equation} \label{quadchard}
\psi_d(\bullet):=\Big( \frac{d}{\bullet} \Big).
\end{equation}

There is no Euler product representation for \eqref{dirich}, so one does not expect a Riemann hypothesis to hold. There are examples of Dirichlet series that do not have an Euler product that fail to be subconvex at the center point. Such an example is given in \cite{CG}. Let
  \begin{equation*}
  D(s):=\sum_{n=1}^{\infty} \frac{\tau(n) \cos \big( \frac{2 \pi n}{q} \big)  }{n^s},
  \end{equation*}
where $\tau(n)$ is the divisor function and $q$ is a prime. This series has conductor $q^2$ and $D(1/2)$ gets as large as $\sqrt{q} \log q$ (convexity) as $q \rightarrow \infty$ through primes. This counterexample would suggest that the Euler product is crucial for subconvexity. However, in the case of automorphic $L$--functions attached to forms of integral weight, the Euler product is induced by the property that the attached form is a simultaneous eigenfunction for the Hecke operators. Jeffrey Hoffstein informally conjectured at Oberwolfach in 2011 that such a property was crucial in implying a Lindel\"{o}f hypothesis. Kiral provided the first progress towards a possible Lindel\"{o}f hypothesis. In particular, Kiral proved \cite{K} that for primitive $\chi$ mod $p$ we have 
\begin{equation}  \label{kiral}
L \Big(\frac{1}{2},f, \chi \Big) \ll_{f,\varepsilon} p^{\frac{3}{8}+\frac{\theta}{4}+\varepsilon},
\end{equation}
where $\theta=7/64$ is the Kim-Sarnak bound. Interestingly, any subconvex exponent $3/8+\theta/4<1$ would be sufficient
to obtain a power saving in our Theorem \ref{secondmoment} below (cf. Remark \ref{subconvexexponent} and the argument above it).
Kiral's result also holds for more general moduli. For reference, the conductor here is $\asymp_{k} p^2$, so the exponent $3/8$ suggests a bound of Burgess quality. In this work we compute the ``barrier" moment for this class of $L$--functions with power saving error term. That is, the moment that gives Lindel\"{o}f on average, but still yields the convexity bound for each individual $L$--value. Computing higher moments in this family certainly warrants further investigation to such an end.  
A mollified and/or amplified variant of the asymptotic second moment result in Theorem \ref{secondmoment} (in addition to the first moment) would lead to a 
positive proportion of non-vanishing at center point and subconvexity results. We leave this to the interested reader.
One can also see \cite{HoKo} for applications of subconvex bounds in the level aspect for double Dirichlet series.
Blomer \cite{Blo11} proved subconvex bounds of such series in the $t$-aspect on the critical line.

This family of $L$--functions attached to half integral weight Kohnen newforms has also been studied in other contexts. In 2020, Lester and Radziwi{\l\l} under the Generalised Riemann Hypothesis proved that half-integral weight holomorphic Hecke forms in Kohnen's plus space satisfy Quantum Unique Ergodicity (QUE) \cite{LR}.

We use the spectral theory of automorphic forms and a delicate analysis of the distribution of $\alpha n^2$ modulo one to prove the following moment result.
\begin{theorem} \label{secondmoment}
 Let $\varepsilon>0$, $j \in \mathbb{N}$, and  $f$ be a holomorphic cuspidal newform of weight $k:=\frac{1}{2}+2 j$ on $\Gamma_0(4)$ such that
\begin{itemize}
\item $f$ lies in Kohnen's plus space,
\item $f$ is a simultaneous Hecke eigenform for all $\mathcal{T}_{q^2}$ with $q>2$ prime, 
\item $f$ is normalised so that its Fourier coefficients are totally real algebraic numbers.
\end{itemize}
As $p \rightarrow \infty$ through  primes $p \equiv 1 \pmod{4}$, we have
\begin{equation} \label{bigtheorem}
\sideset{}{^*} \sum_{\chi \hspace{-0.15cm} \pmod p} | L  (1/2,f,\chi ) |^2=c_1(f) \psi(p) \log(p)+c_2(f) \psi(p)+O_{f,\varepsilon}(p^{1-\frac{1}{600}+\varepsilon}),
\end{equation}
where $\psi(p):=p-2$ is the number of primitive Dirichlet characters modulo $p$, and $c_1(f), c_2(f) \in \mathbb{R}$ are constants depending only on $f$.
Furthermore, we have that $c_1(f)>0$. 
\end{theorem}

\begin{remark}
The constants $c_1(f)$ and $c_2(f)$ are given in \eqref{c1comp} and \eqref{c2comp} respectively.
\end{remark}

\begin{remark}
Other cases of Theorem \ref{secondmoment} i.e. when $k=\frac{3}{2}+2j$ and/or $p \equiv 3 \pmod{4}$ can be established by a mild adaption of the methods in this paper.
It is technically convenient to restrict attention to the case $k=\frac{1}{2}+2j$ and $p \equiv 1 \pmod{4}$.
\end{remark}

\begin{remark}
We emphasize that the purpose of this paper was to break the moral ``convexity barrier" by establishing a power saving error term in Theorem \ref{secondmoment}.
Optimality of the power saving is not pursued in this paper.
\end{remark}

\begin{remark}
There are other interesting potential variants of Theorem \ref{secondmoment}. The methods of the paper should be easily adapted 
to prove a moment with summand of the form $L (1/2,f,\chi ) \overline{L(1/2,g,\chi)}$ with $f,g$ both
Kohnen newforms and orthogonal to one another.
The main term should have magnitude $\psi(p)$ (with no $\log p$) in this case. Another variant is a moment with summand 
$|L(1/2,f,\chi)|^2$ where $f$ is a non-cuspidal metaplectic Eisenstein series. This appears to be more involved because the Fourier coefficients
of half-integer weight Eisenstein series are essentially quadratic Dirichlet $L$-functions.  Unlike the case of Young \cite{Y2}, the convolution structure of the divisor 
function can't be used in this case. Another interesting variant is a moment with summand $L (1/2,f,\chi ) \overline{L(1/2,g,\chi)}$ where
$f$ is a Kohnen newform and $g$ is its Shimura correspondent.
\end{remark}

Theorem \ref{secondmoment} depends on the following bound for a short sum of products of Sali\'{e} sums. The result and its proof are of independent interest, because its origins
are a bilinear form in Sali\'{e} sums.
\begin{theorem} \label{critrangethm}
Let $p$ be a prime with $p \equiv 1 \pmod 4$. Suppose $\varepsilon>0$, $p^{\frac{1}{2}-\frac{1}{10}} \leq N \leq p^{\frac{1}{2}+\frac{1}{10}}$, $1 \leq M \leq p/2$ and $c \in \mathbb{F}^{\times}_p$. Then we have
\begin{align} \label{critrangestat}
& \sum_{N \leq n_1, n_2 \leq 2N} \Big | \sum_{M \leq m \leq 2M} S(m, c n_1,p) \overline{S(m, c n_2,p)} \Big |  \nonumber \\
& \ll_{\varepsilon} p^{\varepsilon} ( MN^2 p^{1-\frac{1}{27}}+MN p^{\frac{3}{2}-\frac{1}{27}} 
+N^2 p^{\frac{3}{2}-\frac{1}{27}}+Np^{2-\frac{1}{27}}+N^{\frac{1}{2}} p^{2+\frac{23}{108}}+p^{\frac{5}{2}-\frac{1}{27} }),
\end{align}
where $S(m,n,p)$ denotes the usual unnormalised Sali\'{e} sum (cf. \eqref{salieval}),
and the implied constant depends only on $\varepsilon$.
\end{theorem}
Theorem \ref{critrangethm} gives a non-trivial power saving over the trivial bound in the Polya--Vinogradov range $M,N \sim p^{1/2+o(1)}$ (cf. \eqref{vordis2}).
In a subsequent (to the writing of this paper) joint work, both authors joint with Kerr and Shparlinski \cite{DKSZ} improved Theorem \ref{critrangethm} using 
an alternative argument that exploited the geometry of numbers and additive combinatorics. One can also see \cite{KSSZ} for further improvements,
as well as generalisation to higher order Sali\'{e} sums.
An arithmetic application of bilinear forms in Sali\'{e} sums  to the equidistribution modulo $1$ of roots to the quadratic congurence $x^{2} \equiv p \pmod{q}$ with $q$ a large prime and $p$ varying over primes $p \leq q$ is also given in \cite{DKSZ}.  Average versions (over $q$) of these applications are given in \cite{SSZ2}.
An average version (over the modulus $q$) of Theorem \ref{secondmoment}
with power saving error term was proved by the second author joint with 
Shkredov and Shparlinski \cite{SSZ}.  

\section{High level sketch} \label{highlevel}
We work with normalised forms whose Fourier coefficients are totally real algebraic numbers to emulate the integral weight setting as much as possible. The natural starting point is an approximate functional equation for the product of $L$--functions
\begin{equation*}
L(s,f,\chi) \overline{L(s,f,\chi) }=L(s,f,\chi) L(\overline{s},f,\overline{\chi}).
\end{equation*}
After summing the approximate functional equation over all primitive $\chi$ modulo $p$ using orthogonality and extracting the main terms, one obtains expressions roughly of the form 
\begin{equation} \label{keyexp1}
\text{(A)} \quad \frac{1}{p} \sum_{\substack{mn \leq p^2 \\ m \neq n } } a(m) \Big( \frac{m}{p} \Big) a(n)  \Big(\frac{n}{p} \Big) \quad \text{and} \quad \text{(B)} \quad \sum_{\substack{mn \leq p^2 \\ m \equiv n \hspace{-0.15cm} \pmod{p} \\ m \neq n } } a(m) a(n),
\end{equation}
where $a(m)$ denotes the Fourier coefficients of the holomorphic half-integer weight cusp form $f$.  The twisted terms in (A) appear because the theta multiplier causes the second term in the approximate functional equation for $L(s,f,\chi) L(\overline{s},f, \overline{\chi})$ to contain Gauss sums attached to character $\chi \big(\frac{\bullet}{p} \big)$.

If one knew the Ramanujan--Petersson conjecture for the Fourier coefficients $a(n)$ of $f$ (cf. \eqref{ram}), then applying this bound pointwise to \eqref{keyexp1} would yield the ``trivial" bound of $O(p^{1+\varepsilon})$. We will beat this bound by a power savings in $p$, without recourse to the Ramanujan--Petersson conjecture. The Ramanujan--Petersson conjecture for the Fourier coefficients of $f$  is tantamount to the Lindel\"{o}f hypothesis
for the $L$-function attached to quadratic twists of the Shimura correspondent of $f$ by the Kohnen-Zagier formula \cite{KZ} (cf.\eqref{KZ}).

We reserve the discussion here for (B). The terms in (A) can then be effectively handled using results implicit in the work of Kiral \cite{K}. There are two well known ways to interpret this double summation. One point of view is to cast it as a shifted convolution problem involving Fourier coefficients of the half-integral weight form. This is useful when the size of the variables is not too far apart. Another option is to consider it a sum over  the Fourier coefficients of a half-integral weight cusp form in arithmetic progressions. This has utility when one variable is significantly larger than the other. 

To be precise, we restrict the variables to $n \asymp N$ and $m \asymp M$ where $N \geq M$ by symmetry and $NM=p^2$. On one hand we can apply Voronoi summation in the inner sum of 
\begin{equation*}
\sum_{m \asymp M} a(m) \sum_{\substack{n \asymp N \\ n \equiv m \pmod{p}}} a(n),
\end{equation*}
obtaining an expression roughly of the form 
\begin{equation} \label{vordis}
\frac{N}{p^2} \sum_{m \asymp M} a(m) \sum_{n \asymp p^2/N} a(n) S(m,n,p),
\end{equation}
where $S(m,n,p)$ denotes the usual unnormalised Sali\'{e} sum (cf. \eqref{salieval}). Using Rankin--Selberg bounds and the evaluation of the Sali\'e sum, we obtain a bound of $M p^{\frac{1}{2}}$, which is admissible if  
$M \leq p^{\frac{1}{2}-\delta}$ (or equivalently $N \geq p^{\frac{3}{2}+\delta}$) for some fixed $\delta>0$.

On the other hand, we can interpret the problem as an averaged shifted convolution sum 
\begin{equation*}
\sum_{r \asymp N/p} \sum_{\substack{n \asymp N, m \asymp M  \\  n-m=rp}} a(m) a(n).
\end{equation*}
We detect the equality using additive characters and apply Jutila circle method to set up the problem. One of the key steps is to perform Voronoi summation in both the $m$ and $n$ summations. The collision of the two theta multipliers (evaluated at opposite sign) in this process has the net effect of twisting by the quadratic character $\chi:=\big(\frac{4 \ell_1 \ell_2}{\bullet} \big)$, and essentially returns us to a weight zero setting.  This feature can also be seen another way. The function
$V_{\ell_1,\ell_2}(\tau):=f(\ell_1 \tau) \overline{f(\ell_2 \tau)} y^k$ has nebentypus $\chi$ on $\Gamma_0(4 \ell_1 \ell_2)$. 
The standard approach to the shifted convolution problem $\ell_1 n-\ell_2 m =h$ 
 would be to obtain the spectral decomposition of
 $\langle V_{\ell_1,\ell_2}, \mathcal{P}_{h}(\cdot,s)\rangle$, where $\mathcal{P}_{h}(\tau,s)$ is 
an appropriate Poincar\'{e }series. One can see \cite{K}.

We get an expression roughly of the form 
\begin{equation*}
\frac{M^2 (\ell_1 \ell_2)}{C^3} \hspace{0.1cm} \sum_{\substack{b \\ |b| \asymp \mathcal{K}}} \hspace{0.15cm} \sum_{\substack{\ell_1 n-\ell_2 m=b \\ \ell_1 n \asymp C^2 N/M^2 \\ \ell_2 m \asymp C^2/M}} a(m) a(n) \sum_{\ell_1 \ell_2 \mid c} \frac{K(b,h,c,\chi)}{c} \Phi \Big(4 \pi \frac{\sqrt{|b| h}}{c} \Big),
\end{equation*}
where $C:=N^{1000}$, $\Phi$ is some smooth weight function and the $K(b,h,c,\chi)$ are the usual weight zero Kloosterman sums twisted by a quadratic character.  Note that the size of $C$ has no bearing on the eventual bounds that are obtained. The Kuznetsov formula can then be applied to the summation over $c$ to decompose it into the contributions from the holomorphic, Maass and Eisenstein spectrums. Here we are able to use the analysis of Blomer and Mili\'{c}evi\'{c} \cite{BM1}. A crucial input in this analysis is a flexible version of the large sieve for Maass forms due to Blomer and Mili\'{c}evi\'{c} \cite[Theorem~13]{BM1} that allows for extra divisibility conditions. This idea leads to a bound roughly of the shape of $Np^{-\frac{1}{2}}$, rather than one of the quality $Np^{\theta-\frac{1}{2}}$, where $\theta$ is the best known exponent toward the Ramanujan--Petersson conjecture for weight zero Maass forms. A precise version of this bound is stated in Proposition \ref{spectralest}. We also develop the analogous flexible large sieve bounds for coefficients of Eisenstein series attached to even Dirichlet characters, generalising the one for trivial nebentypus due to Blomer--Harcos--Michel in \cite{BHM}.  The computational technology for Eisenstein series developed by Kiral--Young \cite{KY1} and Young \cite{Y1} is useful for this. 

Analogous to \cite{BM1}, it remains to close the small gap where $M=p^{\frac{1}{2}+o(1)}$ and $N=p^{\frac{3}{2}+o(1)}$, referred to as the \emph{critical range}. We define the sets,
\begin{equation*}
\mathcal{N}_{0}(f):=\{n \in \mathbb{N}:  0 \leq |a(n)| \leq 1 \},
\end{equation*}
and for all $r \geq 1$,
\begin{equation*}
\mathcal{N}_{r}(f):=\{n \in \mathbb{N}:  2^{r-1}<|a(n)| \leq 2^r \}.
\end{equation*}
We break \eqref{vordis} into $O(\log^2 p)$ subsums 
\begin{equation} \label{vordis2}
\frac{N}{p^2} \sum_{m \asymp M} a(m) \sum_{\substack{n \asymp p^2/N \\ n \in \mathcal{N}_r(f) }} a(n) S(m,n,p).
\end{equation}
Observe that $M=p^{\frac{1}{2}+o(1)}$ and $p^2/N=p^{\frac{1}{2}+o(1)}$, and so we are left estimating a bilinear form involving Sali\'{e} sums in the Polya--Vinogradov range. 
Power saving bounds for bilinear forms in Kloosterman sums and generalized Kloosterman sums in the Polya--Vinogradov range have been 
stunningly proved by Kowalski--Michel--Sawin \cite{KMS,KMS2} using deep algebro-geometric techniques. 
We emphasize that the techniques of \cite{KMS,KMS2} do not apply to the case of Sali\'{e} sums (the monodromy group of the Sali\'{e} sums is too small to make the arguments work). The elementary nature of Sali\'{e} sums require a completely different approach of diophantine nature.
Applying Cauchy--Schwarz to the $m$ sum in \eqref{vordis2} and using Rankin--Selberg bounds we arrive at
\begin{equation} \label{CRdis}
\frac{NM^{\frac{1}{2}}}{p^2} \Big(\sum_{\substack{n_1,n_2 \asymp p^2/N \\ n_1, n_2 \in \mathcal{N}_r(f) }} |a(n_1) a(n_2) | \Big | \sum_{m \asymp M} S(m,n_1,p) \overline{S(m,n_2,p)} \Big | \Big)^{\frac{1}{2}}.
\end{equation}
At this stage it is tempting to invoke the Ramanujan--Petersson conjecture to handle the Fourier coefficients using a sup norm in \eqref{CRdis}. Instead, we estimate \eqref{CRdis} in two different ways depending on the size of $r$. For $r$ large, we estimate trivially. The main feature here is that the Rankin--Selberg bound \eqref{L2} implies that 
\begin{equation*}
| \mathcal{N}_r(f) \cap [0,X]  | \ll_{\varepsilon} \frac{X^{1+\varepsilon}}{2^{2r}}.
\end{equation*}
This guarantees that \eqref{vordis2} is
\begin{equation*}
\ll \frac{M p^{\frac{1}{2}}}{2^r},
\end{equation*}
which saves over the trivial bound as soon as $r$ is large enough. 

For $r$ small, we take the sup norm in \eqref{CRdis}, extend the summation on $n$ by positivity, and use the closed form evaluation of Sali\'{e} sums in terms of Weyl sums to obtain 
\begin{equation} \label{maindis}
\frac{2^r NM^{\frac{1}{2}}}{p^{3/2}} \Big(\sum_{n_1,n_2 \asymp p^2/N}  \Big | \sum_{M \leq m \leq 2 M} \sum_{\substack{u,v \hspace{-0.2cm} \pmod{p} \\ u^2 \equiv mn_1 \hspace{-0.1cm} \pmod{p} \\ v^2 \equiv mn_2 \hspace{-0.1cm} \pmod{p}}} e \Big( \frac{2(u+v)}{p}  \Big)  \Big | \Big)^{\frac{1}{2}}.
\end{equation}
Recall that $M=p^{\frac{1}{2}+o(1)}$ and $p^2/N=p^{\frac{1}{2}+o(1)}$ here. The strategy now is to obtain cancellation in the short $m$ summation by utilising the short average over $n_1$ and $n_2$.  For simplicity, we consider a restricted version of the sum in \eqref{maindis} whose variables $m,n_1$ and $n_2$ satisfy
\begin{equation*}
\Big(\frac{n_1}{p} \Big)=\Big(\frac{n_2}{p} \Big)=\Big(\frac{m}{p} \Big)=1.
\end{equation*}
The other case is analogous. For $\ell \in \mathbb{F}^{\times}_p$, define
\begin{equation*}
A_{\ell}:=\sum_{M \leq m \leq 2 M} \hspace{0.15cm} \sum_{t^2 \equiv m \hspace{-0.2cm} \pmod{p}} e \Big( \frac{2 t \ell}{p} \Big), 
\end{equation*}
and 
\begin{align*} 
\mathbb{S}_{\ell}:= \Big \{(u,v)  \in (\mathbb{F}_p^{\times})^2 &:  (u^2, v^2) \pmod{p} \in   [p^2/N,p^2/N] \times [p^2/N,p^2/N] \\
 & \text{ and } \quad u + v \equiv \ell \hspace{-0.2cm} \pmod{p} \Big \}.
\end{align*}
The triangle inequality asserts that the restricted version of the bracketed sum in \eqref{maindis} is
\begin{equation} \label{ineqdis}
\leq \sum_{\ell \hspace{-0.2cm} \pmod{p}}  | A_{\ell} | |\mathbb{S}_{\ell}|.
\end{equation}
We focus on the non-trivial case when $\ell \in \mathbb{F}^{\times}_p$.
The exponential sums $A_{\ell}$ are too short to complete, so we focus on $\mathbb{S}_{\ell}$, whose elements still capture the averaging over $n_1$ and $n_2$. Recall that $(u,v) \in \mathbb{S}_{\ell}$ are solutions to the linear equation
\begin{equation} \label{linear} 
u+v \equiv \ell  \pmod{p},
\end{equation}
whose squares lie in a short interval. Algebraically manipulating \eqref{linear} we see that $(u,v) \in \mathbb{S}_{\ell}$ must satisfy the polynomial congruence
\begin{equation} \label{congdis}
\overline{\ell}^2 (u^2-v^2)^2+\ell^2 \equiv 2(u^2+v^2) \pmod{p}.
\end{equation} 
We set
\begin{equation*}
\alpha_{\ell}:=\frac{\overline{\ell}^2}{p} \in \mathbb{Q} / \mathbb{Z} \quad \text{and} \quad \beta_{\ell}:=\frac{\ell^2}{p} \in \mathbb{Q} / \mathbb{Z}.
\end{equation*}
Thus \eqref{congdis} implies 
\begin{equation} \label{diophantinedis}
\| \alpha_{\ell} (u^2-v^2)^2 +\beta_{\ell} \| \leq \frac{8p}{N},
\end{equation}
where $\| \bullet \|$ denotes the distance to the closest integer. Therefore pairs $(u,v) \in \mathbb{S}_{\ell}$ produce elements of the sequence $\{\alpha_{\ell} n^2\}_{0 \leq n \leq N}$ modulo $1$ and lie in a cluster around $-\beta_{\ell}$.  For a given $n$ and $\ell$, there is at most one pair $(u,v) \in \mathbb{S}_{\ell}$ that corresponds to them. 

The local spacing distribution of the sequence $\alpha n^2$ for $\alpha$ irrational has been extensively studied in the literature. A classical result of Rudnick and Sarnak states that \cite{RS} for all integers $d \geq 2$ and almost all real $\alpha$, the pair correlation of sequence $\alpha n^d$ mod 1 is Poissonian. This is in contrast with the case $d=1$,
where it is well known that for all $\alpha $ and all $N$, the gaps between  consecutive elements of $\alpha n$ mod 1, $1 \leq n \le N$ can take at most three values. 

Returning to the case $d= 2$, Rudnick, Sarnak, and the second author \cite{RSZ}, \cite{ZA2} show that for sufficiently well approximable numbers $\alpha$, the $m$-level correlations and consecutive spacing are Poissonian along subsequences. For $\alpha=\sqrt{2}$, these types of conjectures are supported numerically \cite{GGI} because of their close connection to the distribution between neighbouring levels of a generic integrable quantum system. It is also shown in \cite{RSZ} that when $b/p \in \mathbb{Q}$, the sequence $bn^2/p$ also has a local Poissonian distribution when the number of points sampled are in certain ranges (in terms of $p$). In shorter ranges, they are able to show such a phenomenon dramatically fails for some $b$. Moreover, some of the clusters of these points in these sequences are so dense that they are capable of making the 5-level (and all higher level) correlations diverge. One key aspect is that, although our $\alpha$'s are rational, our intuition comes from the case $\alpha$ irrational in \cite{RSZ}. Thus, we will not work
with the numbers $\alpha_{\ell}$ themselves to analyse the cluster of points in \eqref{diophantinedis}, but instead consider various convergents to their respective continued fractions. We pay close attention to the size of the denominators of these convergents.

We fix $\boldsymbol{\delta}:=(\delta_2,\delta_3,\delta_4, \delta_5) \in (0,1)^4$ such that $\delta_2<\delta_3<\delta_4$ (in reality, we will need more parameters at our disposal). This vector is chosen appropriately in the course of the proof of Theorem \ref{critrangethm}. To control the size of $\mathbb{S}_{\ell}$, it is necessary to have control over the discrepancy of the sequence $\alpha_{\ell} n^2$. The natural strategy is to use the Erd\"{o}s--Turan theorem in conjunction with Weyl's inequality for exponential sums whose argument is a quadratic polynomial. For this to work, the continued fraction expansion of $h \alpha_{\ell}$ for all $h \in [1,p^{\delta_5}]$ must have a convergent with denominator trapped in $[p^{\delta_2},p^{\delta_3}]$ say. This leads us to essentially partition the summation variable in \eqref{ineqdis} into three sets
\begin{equation*}
\ell \in \mathbb{F}^{\times}_p:=\mathcal{H}_1 \cup \mathcal{H}_2 \cup \mathcal{H}_3,
\end{equation*} 
which are described below. The subset $\mathcal{H}_1$ contains exactly those $\ell$ described and so the sum over such $\ell$ in \eqref{ineqdis} can be handled. 

The next subset of $\ell$ we consider are those such that there exists a $h_{\ell}  \in [1,p^{\delta_5}]$ (it may depend on $\ell$) such that $h_{\ell} \alpha_{\ell}$ has no convergent with denominator in the larger interval $[p^{\delta_2},p^{\delta_4}]$. There is no toggle to control the size of $\mathbb{S}_{\ell}$ here, but $h_{\ell} \alpha_{\ell}$ has two consecutive convergents whose denominators have a large gap. This is somewhat a rare event, and an argument with standard inequalities from continued fractions indeed forces the size of $\mathcal{H}_2$ to be small. 

This leaves the third and final set $\mathcal{H}_3$ to consider. For $\ell \in \mathcal{H}_3$, there exists a $h_{\ell} \in [1,p^{\delta_5}]$ such that $h_{\ell} \alpha_{\ell}$ has no convergent with denominator in $[p^{\delta_2},p^{\delta_3}]$, but guaranteed to have one, say $a^{\star}_{\ell}/b^{\star}_{\ell}$, with $b^{\star}_{\ell} \in (p^{\delta_3},p^{\delta_4}]$. Here we must study the sizes of $\mathcal{H}_3$ and $\mathbb{S}_{\ell}$ (which is equivalent to analysing $\alpha_{\ell} n^2$) simultaneously. We make this precise now. For each $\ell$, denote 
\begin{equation*}
\mathbb{V}_{\ell}:=\Big \{0 \leq n \leq p^2/N: \| \alpha_{\ell} n^2+\beta_{\ell}  \| \leq \frac{8 p}{N}   \Big \},
\end{equation*}
and for each $p^{\delta_3} \leq U \leq p^{\delta_4}$ and $0 \leq V \leq p^2/N$ we define
\begin{equation*}
\mathcal{E}(U,V):=\{\ell \in \mathcal{H}_3: b^{*}_{\ell} \in [U,2U] \quad \text{and} \quad |\mathbb{V}_{\ell}| \in [V,2V]  \}.
\end{equation*}
Thus
\begin{equation} \label{refinedintro}
\sum_{\ell \in \mathcal{H}_3} |A_{\ell}| |\mathbb{S}_{\ell} | \ll M \log^2 p \max_{p^{\delta_3} \leq U \leq p^{\delta_4}} \max_{1 \ll V \leq p^2/N} V \cdot |\mathcal{E}(U,V)|.
\end{equation}  
We can assume $V$ is moderately large by trivial considerations. Thus we need to bound $V \cdot |\mathcal{E}(U,V)|$.  For each $\ell \in \mathcal{E}(U,V)$, we now construct an algebraic set $\mathfrak{C}_{\ell} \subseteq \mathbb{F}^3_p$ with restricted variables. Arrange the numbers $n_{\ell,j} \in \mathbb{V}_{\ell}$ with order
 \begin{equation} \label{stringintro}
0 \leq n_{\ell,1}<n_{\ell,2}<\cdots<n_{\ell,|\mathbb{V}_{\ell}|} \leq p^2/N.
 \end{equation}
 The average consecutive gap between these numbers is 
 \begin{equation*}
 \frac{p^2}{N |\mathbb{V}_{\ell}|} \asymp \frac{p^2}{NV}.
 \end{equation*}
 More than $|\mathbb{V}_{\ell}|/2$ consecutive gaps are less than or equal to $2p^2/(N |\mathbb{V}_{\ell}|)$. By the pigeonhole principle there exists an integer $1 \leq d_{\ell} \leq 2p^2/N |\mathbb{V}_{\ell}|$ that is repeated as a consecutive gap greater than or equal to $|\mathbb{V}_{\ell}|^2 N/4p^2$ times.  Thus we consider
\begin{multline} \label{syscong1intro}
 \mathfrak{C}_{\ell}:= \big \{ (n,A,B) \in [1,p^2/N] \times [-8p^2/N,8p^2/N]^2 : \overline{\ell}^2 n^2+\ell^2 \equiv A \pmod{p} \\
    \quad \text{and} \quad \overline{\ell}^2 (n+d_{\ell})^2+\ell^2  \equiv B \pmod{p} \big \},
\end{multline}
and 
\begin{equation*}
\mathfrak{U}(U,V):= \bigcup_{\ell \in \mathcal{E}(U,V)} \{\ell \} \times \mathfrak{C}_{\ell} \subseteq \mathbb{F}^4_p,
\end{equation*}
and perform an overall count of points in this last set in order to obtain a contradiction, unless the
bound in the statement of Theorem \ref{critrangethm} holds. On the one hand, we know that the
size of this set should be large by construction. On the other hand, the nature of the set
forces it to be thin enough. This eventually leads to the proof of Theorem \ref{critrangethm}.

The anatomy of the paper is as follows. Section \ref{holohalf} gives background material on holomorphic half--integer weight modular forms. Section \ref{core} contains the main argument to establish Theorem \ref{secondmoment}. The main terms are extracted using Kronecker's first limit formula, and auxiliary bounds for sums of Fourier coefficients that are needed are also listed there.  Propositions \ref{critandupper} and \ref{spectralest} are the main inputs used in the proof of Theorem \ref{secondmoment}. Proposition \ref{spectralest} uses spectral techniques. The relevant automorphic preliminaries are contained in Section \ref{autprelim2} and Proposition \ref{spectralest} is proved in Section \ref{spectralsec}. Proposition \ref{critandupper} contains the critical range bounds and takes Theorem \ref{critrangethm} as input. Theorem \ref{critrangethm} is proved in Section \ref{critsec}.

\section{Conventions}
All implied constants in proofs are allowed to depend on $\varepsilon>0$ (possibly different in each instance), $f \in \mathcal{S}_k(4)$, and 
the fixed smooth functions introduced in various partitions of unity.
The square root $\sqrt{}$ denotes the principal branch of the square root.

\section{Acknowledgement}
We thank Matthew Young and Djordje Mili\'{c}evi\'{c} for their comments on the manuscript and Valentin Blomer and E.~Mehmet Kiral for helpful correspondences. 
We also thank the referee for their meticulous comments on the manuscript.

\section{Automorphic preliminaries I (half integral weight)}  \label{holohalf}
\subsection{Holomorphic cusp forms and $L$-functions} \label{holoaut}
For $\tau:=x+iy \in \mathbb{H}$, define $q:=e^{2 \pi i \tau}$. Let 
\begin{equation*}
\theta(\tau):=\sum_{n=-\infty}^{\infty} q^{n^2},
\end{equation*}
and 
\begin{equation*}
\eta(\tau):=q^\frac1{24}\prod_{n=1}^\infty(1-q^n)
\end{equation*}
be the fundamental theta functions. Define the theta multiplier $\nu_{\theta}$ on $\Gamma_0(4)$ by  
\begin{equation*}
\theta(\gamma \tau)=\nu_{\theta}(\gamma) \sqrt{c \tau+d} \hspace{0.1cm} \theta(\tau), \quad \text{for} \quad \gamma=\begin{pmatrix}
a & b \\
c & d
\end{pmatrix} \in \Gamma_0(4),
\end{equation*}
and the eta multiplier $\nu_{\eta}$ on $\text{SL}_2(\mathbb{Z})$ given by
\begin{equation*}
\eta(\gamma \tau)=\nu_{\eta}(\gamma) \sqrt{c \tau+d} \hspace{0.1cm} \eta(\tau), \quad \text{for} \quad \gamma=\begin{pmatrix}
a & b \\
c & d
\end{pmatrix} \in \text{SL}_2(\mathbb{Z}).
\end{equation*}

The theta multiplier is given by the formula
\begin{equation} \label{thetamultiplier}
\nu_{\theta}(\gamma)=\varepsilon^{-1}_d \Big( \frac{c}{d} \Big),
\end{equation}
where $\big( \frac{\bullet}{\bullet}  \big)$ denotes the Kronecker symbol and 
\begin{equation} \label{epsdefn}
\varepsilon_d:=\begin{cases}
1 & \text{ if } d \equiv 1 \pmod{4} \\
i & \text{ if } d \equiv 3 \pmod{4}.
\end{cases}
\end{equation}
For $j \geq 1$ and $k:=\frac{1}{2}+2j$, let $\mathcal{S}_k (4)$ denote the space of holomorphic cusp forms of weight $k$, level $4$ and trivial nebentypus. If $f \in \mathcal{S}_k (4)$, then $f: \mathbb{H} \rightarrow \mathbb{C}$ is holomorphic, vanishes on all three cusps of $\Gamma_0(4)$, and satisfies
\begin{equation*}
f(\gamma \tau)=\nu_{\theta}(\gamma) (c \tau+d)^k f(\tau) \quad \text{for all} \quad \gamma \in \Gamma_0(4).
\end{equation*}
For $f,g \in \mathcal{S}_k(4)$, define the Petersson inner product
\begin{equation*}
\langle f, g \rangle:=\int_{\Gamma_0(4) \backslash \mathbb{H}} y^k f(\tau) \overline{g(\tau)} d \mu(\tau), \quad d \mu(\tau):=\frac{dx dy}{y^2}.
\end{equation*} 
Recall that $\mathcal{S}_k(4)$ becomes a Hilbert space with the inner product defined above. 

Let the Fourier expansion of $f$ at $\infty$ be given by
\begin{equation} 
f(\tau)=\sum_{n=1}^{\infty} b(n) e(n \tau)=\sum_{n=1}^{\infty} a(n) n^{\frac{k-1}{2}} e(n \tau).
\end{equation}
It follows from \cite[pg~786]{DI} that for $X \geq 1$ we have 
\begin{equation} \label{L2}
\sum_{n \leq X} |a(n)|^2 \ll_{f} X \log X.
\end{equation}
A Wilton type bound also follows from \cite[pg~786]{DI}, 
\begin{equation} \label{wilton}
\sum_{n \leq X} a(n) e(n \alpha) \ll_{f} X^{\frac{1}{2}} \log^2 X,
\end{equation}
where the implied constant is uniform with respect to $\alpha \in \mathbb{R}$.

For odd primes $q$, the Hecke operators $\mathcal{T}_{q^2}$ defined on $\mathcal{S}_k(4)$ (with $k=\frac{1}{2}+2j$) are defined by  
\begin{equation*}
\mathcal{T}_{q^2} f(\tau):=\sum_{n \geq 1} \Big( b(q^2n)+\Big(\frac{n}{q} \Big) q^{k-\frac{3}{2}} b(n)+q^{2k-2} b \Big(\frac{n}{q^2} \Big) \Big) e(n \tau).
\end{equation*}
Here we have used the convention that $b(x)=0$ unless $x \in \mathbb{Z}$. We call a half-integral weight cusp form a Hecke cusp form if $\mathcal{T}_{q^2} f=\lambda(q) f$ for all $q>2$.  The Kohnen plus space $\mathcal{S}_k^{+}(4)$ denotes the subspace of $\mathcal{S}_{k}(4)$ consisting of cusp forms $f$ whose Fourier coefficients satisfy 
\begin{equation*}
b(n)=0 \quad \text{unless} \quad n \equiv 0,1 \pmod{4}.
\end{equation*} 
The plus space has a basis of simultaneous eigenfunctions of the $\mathcal{T}_{q^2}$ for $q$ odd. For $f \in \mathcal{S}^{+}_k(4)$, there exists a Hecke cusp form $g \in \mathcal{S}_{2k-1}(1)$  (via the Shimura lift), where $\mathcal{S}_{2k-1}(1)$ denotes the space of integer weight cusp forms of weight $2k-1$ and trivial nebentypus.
This correspondence has the property that $H_q g=\lambda(q) g$, where $H_q$ denotes the usual integer weight Hecke operator on $\mathcal{S}_{2k-1}(1)$. By the strong multiplicity one theorem this determines $f$ up to scalar multiplication. Write the Fourier expansion of $g$ as 
\begin{equation*}
g(\tau)=\sum_{n=1}^{\infty} c(n) e(n \tau),
\end{equation*}
and normalise such that $c(1)=1$. We can normalise $f$ so that its coefficients lie in the field generated over $\mathbb{Q}$ by the coefficients $c(n)$ by \cite[Proposition~2.3.1]{Stev}, and hence are totally real algebraic numbers. We have the coefficient relation
\begin{equation} \label{heckecoeff}
b(d \delta^2)=b(d) \sum_{e \mid \delta} \mu(e) e^{k-\frac{3}{2}} \psi_d(e) c \Big( \frac{\delta}{e} \Big),
\end{equation}
where $\psi_d$ was defined in \eqref{quadchard}. Recalling that both $f$ and $g$ have been normalised, the Kohnen--Zagier formula \cite[Theorem~1]{KZ} asserts  
\begin{equation} \label{KZ}
\frac{b(d)^2}{\frac{1}{6} \langle f,f \rangle}=\frac{\big(k-\frac{3}{2} \big)!}{\pi^{k-\frac{1}{2}}} d^{k-1} \frac{L(g \otimes \psi_d,1/2)}{\langle g,g \rangle}.
\end{equation}
Combining \eqref{heckecoeff} and \eqref{KZ} we see that the Lindel\"{o}f hypothesis for all quadratic twists of the Shimura lift of $f$ implies the Ramanujan--Petersson conjecture
\begin{equation} \label{bbound}
b(n) \ll_{f,\varepsilon} n^{\frac{k-1}{2}+\varepsilon} \quad \text{for all} \quad n \in \mathbb{N}.
\end{equation} 
Recalling the normalisation in \eqref{fourier}, we obtain
\begin{equation} \label{ram}
a(n) \ll_{f,\varepsilon} n^{\varepsilon}.
\end{equation}
It is well known that (cf \cite[(1.1)]{I2})
\begin{equation}  \label{iwaniec}
a(n) \ll_{f,\varepsilon} n^{\frac{1}{4}+\varepsilon} \quad \text{for all} \quad n \in \mathbb{N}.
\end{equation}
There has been considerable progress toward \eqref{ram}. Iwaniec \cite{I2} proved that
\begin{equation} \label{314}
a(n) \ll_{f,\varepsilon} n^{\frac{3}{14}+\varepsilon},
\end{equation}
for all squarefree $n$.
Conrey and Iwaniec \cite{CI} improved \eqref{314} to a Weyl-type subconvex bound
\begin{equation} \label{CIbound}
a(n) \ll_{f,\varepsilon} n^{\frac{1}{6}+\varepsilon},
\end{equation}
for all squarefree $n$.
Conrey and Iwaniec achieved this bound by estimating a moment involving the $L$-values $L(1/2,g \otimes \psi_n)^3$ summed
over all primitive cusp forms $g$ of level dividing $n$ and fixed integral weight.
If $f$ a Hecke cusp form then \eqref{CIbound} can be extended to all $n \in \mathbb{N}$ via \eqref{heckecoeff}.  

Define the operator on $\mathcal{S}_k(4)$ given by
 \begin{align*}
(W_4 f)(\tau)&:=(2i \tau)^{-k} f \Big(\frac{-1}{4 \tau} \Big).
\end{align*}
Note that $W_4$ is an involution. Since $W_4$ commutes with each $\mathcal{T}_{p^2}$, a Kohnen newform 
is also an eigenfunction of $W_4$ with eigenvalue $\varepsilon(f)=\pm 1$ by the strong multiplicity one theorem
for the plus space. 

Let $Q \in \mathbb{N}$ and $\chi$ be a Dirichlet character of modulus $Q$ and conductor $Q^{\star}$. Then if 
\begin{equation*}
f(\tau)=\sum_{n=1}^{\infty} b(n) q^n \in \mathcal{S}_k (4),
\end{equation*}
define the $\chi$-twist by 
\begin{equation*}
f_{\chi}(\tau):=\sum_{n=1}^{\infty} b(n) \chi(n) q^n \in \mathcal{S}_k (4{Q^{\star}}^2,\chi^2 ).
\end{equation*}
 The $L$-function of the twist $f_{\chi}$ (recall the normalisation in \eqref{fourier}) is defined by
\begin{equation*}
L(s,f,\chi):=\sum_{n=1}^{\infty} \frac{a(n) \chi(n)}{n^s}, \quad \Re s>1.
\end{equation*} 
After taking the Mellin transform as on \cite[pg.~694]{K}, the completed $L$-function of $f_{\chi}$ is given by
\begin{equation} \label{completed}
L^{*}(s,f,\chi)=2^{-1} \pi^{-\frac{k}{2}-s} (4 Q^2)^{\frac{s}{2}} \Gamma \Big(\frac{s+\frac{k-1}{2}}{2} \Big) \Gamma \Big(\frac{s+\frac{k+1}{2}}{2} \Big) L(s,f,\chi).
\end{equation} 

We now give some details regarding the functional equation of $L^{*}(s,f,\chi)$ in the case $(Q,4)=1$ following \cite{K}. Observe that $f_{\chi}$ can be realised as an average over additive twists
\begin{equation} \label{twist}
f_{\chi}(\tau)=\frac{1}{\mathcal{G}_{\overline{\chi}}(1;Q)} \sum_{u \hspace{-0.2cm} \pmod{Q} } \overline{\chi}(u) f \Big( \tau+\frac{u}{Q} \Big).
\end{equation}
Here for $c,n \in \mathbb{N}$ with $Q \mid c$, we denote the Gauss sum
\begin{equation} \label{gaussdef}
\mathcal{G}_{\chi}(n;c):=\sideset{}{^*} \sum_{d=1}^c \chi(d) e \Big( \frac{nd}{c} \Big),
\end{equation}
where $*$ denotes that the summation is over all $d$ modulo $c$ such that $(d,c)=1$. Using \eqref{twist}, we can rewrite \eqref{completed} as
\begin{equation} \label{newcomplete}
L^{*}(s,f,\chi)=\frac{1}{\mathcal{G}_{\overline{\chi}}(1;Q)} \sum_{u \hspace{-0.2cm} \pmod{Q}} \overline{\chi}(u) L^{*} \Big(s,f,\frac{u}{Q} \Big),
\end{equation}
where 
\begin{equation*}
L^{*} \Big(s,f, \frac{u}{Q} \Big):=(4 Q^2)^{\frac{s}{2}} \int_{0}^{\infty} f \Big( iy+\frac{u}{Q} \Big) y^{s+\frac{k-1}{2}} \frac{dy}{y}.
\end{equation*}
Let 
\begin{equation*}
\widetilde{\psi}_Q(\bullet):=\Big( \frac{\small{\bullet}}{Q} \Big).
\end{equation*}
Suppose $u$ and $v$ are any integers satisfying $4uv \equiv -1 \pmod{Q}$. A computation using the matrix identity   
\begin{equation*} \label{matrix}
\begin{pmatrix}
1 & u/Q \\
0 & 1
\end{pmatrix} 
\begin{pmatrix}
0 & -1/Q \\
4Q & 0
\end{pmatrix}=\begin{pmatrix}
(4uv+1)/Q & u \\
4v & Q
\end{pmatrix}
\begin{pmatrix}
0 & -1 \\
4 & 0
\end{pmatrix}
\begin{pmatrix}
1 & v/Q \\
0 & 1
\end{pmatrix},
\end{equation*} 
and the fact that 
\begin{equation*}
W_4 f=\varepsilon(f) f,
\end{equation*}
gives the relation
\begin{equation} \label{funcadd}
L^{*} \Big(s,f,\frac{u}{Q} \Big)=\varepsilon(f) \varepsilon_Q^{-2k} \Big( \frac{v}{Q} \Big) L^{*} \Big(1-s,f,\frac{v}{Q} \Big).
\end{equation}
Observe that \eqref{newcomplete} and \eqref{funcadd} imply the functional equation 
\begin{equation} \label{functionaleq}
L^*(s,f,\chi)=\varepsilon^{*}(f,\chi) L^{*}(1-s,f,\chi, \chi \widetilde{\psi}_Q),
\end{equation}
where $\varepsilon^*(f,\chi)$ is a quantity of absolute value one 
\begin{equation*}
\varepsilon^*(f,\chi):=\varepsilon(f) \varepsilon_Q^{-2k} \chi(-4),
\end{equation*} 
and the $L$-function defined on the right of \eqref{functionaleq} is defined by 
\begin{equation} \label{funceq}
L^*(s,f,\chi, \chi \widetilde{\psi}_Q):=(4 Q^2)^{\frac{s}{2}} (2 \pi)^{-(s+\frac{k-1}{2})} \Gamma \Big( s+\frac{k-1}{2} \Big) \frac{1}{\mathcal{G}_{\overline{\chi}}(1;Q) } \sum_{n=1}^{\infty} \frac{a(n) \mathcal{G}_{\chi \widetilde{\psi}_Q}(n;Q)}{n^s}.
\end{equation}
The Fourier coefficients in \eqref{fourier} are normalised as such to make \eqref{functionaleq} symmetric about the point $s=1/2$.

A computation following \cite[Chapter~5]{IK} shows that \eqref{functionaleq} implies an approximate functional equation. Let $V: \mathbb{R}_{>0} \rightarrow \mathbb{R}$ be defined by 
\begin{equation*}
V(x)=\frac{1}{2 \pi i } \int_{(3)} \frac{x^{-z}}{(2 \pi)^z} \frac{\Gamma \big(z+\frac{k}{2} \big)}{\Gamma(\frac{k}{2})} \frac{dz}{z}.
\end{equation*}
We have
\begin{equation} \label{approxfunceq}
L \Big(\frac{1}{2},f, \chi \Big)=\sum_{m=1}^{\infty} \frac{a(m) \chi(m)}{\sqrt{m}} V \Big(\frac{m}{2 Q} \Big)+ \frac{\varepsilon^*(f,\chi)}{\mathcal{G}_{\overline{\chi}}(1;Q)} \sum_{m=1}^{\infty} \frac{a(m) \mathcal{G}_{\chi \widetilde{\psi}_Q}(m;Q)}{\sqrt{m}} V \Big( \frac{m}{2 Q} \Big).
\end{equation}

\subsection{Voronoi summation}
Let $V:(0,\infty) \rightarrow \mathbb{C}$ be a smooth function with compact support. Define the Hankel transform
\begin{equation*}
\mathring{V}(y):=2 \pi i^k \int_{0}^{\infty} V(x) J_{k-1} (4 \pi  \sqrt{xy}) dx,
\end{equation*}
where $J$ denotes the usual $J$--Bessel function. Note that $\mathring{V}$ depends on $k$ but is not displayed in the notation. We now see that $\mathring{V}$ is a Schwartz function. By \cite[Section~2.6]{BM2} we have 
\begin{equation} \label{BM2}
\int_{0}^{\infty} V(x) J_{k-1}(4 \pi \sqrt{xy}) dx= \Big({-\frac{1}{2 \pi \sqrt{y}}} \Big)^j \int_0^{\infty} \frac{\partial^j}{\partial x^j} \Big(V(x) x^{-\frac{k-1}{2}} \Big) x^{\frac{k-1+j}{2}} J_{k-1+j}(4 \pi \sqrt{xy}) dx,
\end{equation}
for any $j \in \mathbb{N}_0$. One can then differentiate repeatedly under the integral sign in \eqref{BM2} using \cite[(8.471.2)]{GR}.

The next lemma follows from \cite[pg.~792]{DI}.
 
\begin{lemma}  \label{Voronoi}
Let $c \in \mathbb{N}$ such that $4 \mid c$ and
\begin{equation*}
\gamma=\begin{pmatrix}
a & b \\
c & d
\end{pmatrix} \in \Gamma_0(c).
\end{equation*}
Let $V:(0,\infty) \rightarrow \mathbb{C}$ be a smooth function with compact support. Suppose $k=\frac{1}{2}+2 j$ with $j \in \mathbb{N}$ and let $a(n)$ denote the normalised Fourier coefficients of $f \in \mathcal{S}_k(4)$ as in \eqref{fourier}. Then for $X>0$ we have
\begin{equation*}
\sum_n a(n) e \Big( \frac{an}{c} \Big) V \Big( \frac{n}{X} \Big)= \frac{X}{c} \nu_{\theta}(\gamma) \sum_n a(n) e \Big({-\frac{d n}{c}} \Big) \mathring{V} \Big( \frac{n}{c^2/X} \Big).
\end{equation*}
\end{lemma}

\subsection{Half integral weight Kloosterman sums and Sali\'{e} sums}
Let $\kappa,c,m,n \in \mathbb{N}$ such that $\kappa \equiv 1 \pmod{2}$ and $4 \mid c$. Then for any Dirichlet character $\chi$ modulo $c$ define
\begin{equation} \label{Kkappadef}
K_{\kappa,\chi}(m,n;c):=\sideset{}{^*} \sum_{d \hspace{-0.2cm} \pmod{c}} \varepsilon^{-\kappa}_d \Big( \frac{c}{d} \Big) \chi(d) e \Big( \frac{m d +n \overline{d}}{c} \Big).
\end{equation}
When $\chi=\mathbf{1}_c$ in \eqref{Kkappadef} we suppress the subscript.
For $q \in \mathbb{N}$ with $q \equiv 1 \pmod{2}$ and any Dirichlet character $\Psi$ modulo $q$ define the twisted sums
\begin{equation} \label{Sdef}
S_{\Psi}(m,n;q)=\sideset{}{^*} \sum_{x \hspace{-0.2cm} \pmod{q}} \Big( \frac{x}{q} \Big) \Psi(x) e \Big(\frac{m x+n \overline{x}}{q}  \Big).
\end{equation}
When $\Psi=\mathbf{1}_q$ in \eqref{Sdef}, we recover the well known Sali\'{e} sum, and suppress the subscript. These sums have a closed form evaluation, unlike the Kloosterman sums attached to the trivial multiplier. Sarnak \cite[pg~90]{Sa} asserts that this phenomenon is the finite analogue the of Bessel function being an elementary function when its order is an odd half integer. For example,
\begin{equation*}
J_{\frac{1}{2}}(x)=\sqrt{\frac{2}{\pi x}} \sin x, \quad J_{-\frac{1}{2}}(x)=\sqrt{\frac{2}{\pi x}} \cos x.
\end{equation*}
When $q=p$ a prime and $(mn,p)=1$, \cite{Sal} gives 
\begin{equation} \label{salieval}
S(m,n;p)=
\begin{cases}
\big( \frac{n}{p} \big) \varepsilon_p \sqrt{p} \hspace{0.15cm} \sum_{\substack{x \hspace{-0.2cm} \pmod{p} \\ x^2 \equiv mn \hspace{-0.2cm} \pmod{p}}} e \big( \frac{2x}{p} \big) & \text{ if } \quad  \big( \frac{mn}{p}  \big)=1   \\
0 & \text{if} \quad \big( \frac{mn}{p}  \big)=-1,
\end{cases}
\end{equation}
where $\varepsilon_d$ is given by \eqref{epsdefn}.

We have the following useful twisted multiplicativity lemma.  
\begin{lemma} \label{twistlem}
Suppose $c=qr$ with $r \equiv 0 \pmod{4}$ and $(q,r)=1$ any Dirichlet character $\Psi$ is a Dirichlet character mod $c$. Let $\Psi_r$ and $\Psi_q$ are Dirichlet characters modulo $r$ and $q$ respectively such that $\Psi=\Psi_r \Psi_q$. Then 
\begin{equation*}
K_{\kappa,\Psi}(m,n;c)=K_{\kappa-q+1,\Psi_r}(m \overline{q},n \overline{q};r) S_{\Psi_q}(m \overline{r}, n \overline{r};q).
\end{equation*}
\end{lemma}

\subsection{Eisenstein series and Rankin--Selberg $L$-functions}
We give a brief background on the Eisenstein series and Rankin--Selberg $L$--functions relevant to the main terms appearing Theorem \ref{secondmoment} and in Section \ref{mainterms}. One can consult Section \ref{autprelim2} below and \cite[Chapter~13]{I1} for more details.
Let 
\begin{equation} \label{level4eis}
E_{\infty}(\tau,s):=\sum_{\small{\gamma \in \Gamma_0(4)_{\infty} \backslash \Gamma_0(4)}} (\Im(\gamma \tau))^s, \quad \Re(s)>1,
\end{equation}
denote the weight zero Eisenstein series of level $4$ attached to the cusp $\infty$. This Eisenstein series has a meromorphic continuation to all of $\mathbb{C}$ with its only pole in the region $\Re s \geq \frac{1}{2}$ being simple and at $s=1$, with residue 
\begin{equation} \label{volgamma04}
\text{Res}_{s=1} E_{\infty}(\tau,s)=\frac{1}{\text{Vol}(\Gamma_0(4) \backslash \mathbb{H})}=\frac{1}{2\pi}.
\end{equation}
For $f \in \mathcal{S}_k(4)$, consider the Rankin--Selberg $L$-function 
\begin{equation*}
L(s, f \times \overline{f})=\sum_{n=1}^{\infty} \frac{|a(n)|^2}{n^s} \quad \text{for} \quad \Re s>1.
\end{equation*}
The analytic continuation of $L(s, f \times \overline{f})$ is afforded by the above Eisenstein series. In particular, \cite[Proposition~13.1]{I1} asserts (after taking into account the
normalisations in the first two displays on \cite[pg.~233]{I1}):
\begin{equation} \label{rankin}
(4 \pi)^{-s-(k-1)} \Gamma(s+k-1) L(s, f \times \overline{f})=\int_{\Gamma_0(4) \backslash \mathbb{H}} y^k |f(\tau)|^2  E_{\infty}(\tau,s) d \mu(\tau), \quad \Re s>1.
\end{equation}
The function $L(s,f \times \overline{f})$ also satisfies a vector functional equation ($L(s,f \times \overline{f})$ is one of the vector entries)
with scattering matrix $\Phi(s,\chi_0)$,
where $\chi_0$ denotes the principal character modulo $4$, see \cite[Theorem~13.4]{I1}. By \cite[pg.~240]{I1} the relevant scattering matrix has finite order and hence $L(s,f \times \overline{f})$ has polynomial
growth in fixed vertical strips of $\mathbb{C}$ by the Phragmen-Lindel\"{o}f principle.
Consulting \cite[(13.34)]{I1}, there is a simple pole at $s=1$ with 
\begin{equation*}
\text{Res}_{s=1} L(s,f \times \overline{f})= \frac{(4\pi)^k}{\Gamma(k)} \frac{1}{\text{Vol}(\Gamma_0(4) \backslash \mathbb{H})}
\int_{\Gamma_0(4) \backslash \mathbb{H}} y^k |f(\tau)|^2 d \mu(\tau).
\end{equation*}

We will need the pole and constant term in the Laurent expansion of $E_{\infty}(\tau,s)$. We use M\"{o}bius inversion and the bijection
\cite[Lemma~7.1.6(1)]{Mi}
\begin{equation*}
\Gamma_0(4)_{\infty} \backslash \Gamma_0(4) \simeq  \{(c,d): c \equiv 0 \pmod{4}, \quad (c,d)=1 \quad \text{ and }  \quad d>0  \},
\end{equation*}
in \eqref{level4eis} to obtain
\begin{equation} \label{level4eis2}
E_{\infty}(\tau,s):=\frac{1}{2} \frac{1}{4^s} \frac{1}{L(2s,\chi_0)} \Big( \sideset{}{^\prime} \sum_{c,d} \frac{(4y)^s}{|c(4 \tau)+d |^{2s}} -\frac{1}{2^s} \sideset{}{^\prime}  \sum_{c,d} \frac{(2y)^s}{|c(2\tau)+d|^{2s}}  \Big), \quad \Re s>1,
\end{equation}
where $\chi_0$ denotes the principal character modulo $4$ and $^{\prime}$ denotes the exclusion of $(c,d)=(0,0)$. Both of the series on the right of \eqref{level4eis2} have meromorphic continuation to all of $\mathbb{C}$ with only simple poles at $s=1$ \cite[pg.~273]{La}. A computation that uses the Taylor expansion of $2^{-s}$
at $s=1$, and that applies Kronecker's first limit formula \cite[pg.~273]{La} to the right side of \eqref{level4eis2} gives (for each $\tau \in \mathbb{H}$) the Laurent expansion
\begin{equation} \label{kronecker1}
E_{\infty}(\tau,s)=\frac{1}{2} \frac{1}{4^s} \frac{1}{L(2s,\chi_0)} \Big(\frac{\pi/2}{s-1}+\pi \Big(\gamma-\frac{5}{2} \log 2 \Big)+2 \pi \log \Big( y^{-\frac{1}{4}} \Big | \frac{\eta(2 \tau)}{\eta(4 \tau)^2}  \Big | \Big)+O_{\tau}(s-1) \Big),
\end{equation} 
for $s \in \mathbb{C}$ such that $|s-1|<1$.

\subsection{Functional equation for the second moment}
Here we take $Q=p$ a prime with $p \equiv 1 \pmod{4}$, $\chi$ a primitive character modulo $p$ such that $\chi \neq \widetilde{\psi}_p$. Thus $\chi \widetilde{\psi}_p$ is primitive, so we may appeal to the properties of Gauss sums attached to primitive characters. Let $f \in \mathcal{S}^{+}_k(4)$ be a simultaneous Hecke eigenform, so it is automatically an eigenfunction of $W_4$. Also suppose $f$ normalised so that its coefficients are totally real algebraic numbers (cf. Section \ref{holoaut}). Applying \cite[Theorem~8.15]{A} we write \eqref{approxfunceq} as 
\begin{equation} \label{approxfunceq2}
L \Big(\frac{1}{2},f, \chi \Big)=\sum_{m=1}^{\infty} \frac{a(m) \chi(m)}{\sqrt{m}} V \Big(\frac{m}{2 p} \Big)  \\
+ \frac{\varepsilon^*(f,\chi)}{\mathcal{G}_{\overline{\chi}}(1;p)} \sum_{m=1}^{\infty} \frac{a(m) \overline{\chi} \widetilde{\psi}_p (m) \mathcal{G}_{\chi \widetilde{\psi}_p}(1;p)}{\sqrt{m}} 
V \Big( \frac{m}{ 2p} \Big).
\end{equation}
A computation with \eqref{approxfunceq2} using the fact that the Fourier coefficients are real shows that 
\begin{equation*}
\overline{L \Big(\frac{1}{2},f, \chi \Big)}=L \Big(\frac{1}{2},f, \overline{\chi} \Big).
\end{equation*}
Define $W: \mathbb{R}_{>0} \rightarrow \mathbb{R}$ by
\begin{equation*}
W(x)=\frac{1}{2 \pi i } \int_{(3)} \frac{x^{-z}}{(2 \pi)^{2z}} \frac{\Gamma(z+\frac{k}{2})^2}{\Gamma(\frac{k}{2})^2} \frac{dz}{z}.
\end{equation*}
A computation following \cite[Chapter~5]{IK} shows that \eqref{functionaleq} implies a second approximate functional equation
\begin{align} 
L \Big( \frac{1}{2},f,\chi \Big) L \Big( \frac{1}{2},f, \overline{\chi} \Big)&=\sum_{m,n} \frac{a(m) \overline{\chi}(m) a(n) \chi(n)}{\sqrt{mn}} W \Big( \frac{mn}{4p^2} \Big) \nonumber \\
&+\sum_{m,n} \frac{a(m)  \overline{\chi}(m)  \widetilde{\psi}_p(m) a(n) \chi(n) \widetilde{\psi}_p(n) }{\sqrt{mn}} W \Big( \frac{mn}{4p^2} \Big).  \label{approxsec}
\end{align}
Note that we have for all $A>0$ and $j \geq 0$ we have 
\begin{equation} \label{decaybd}
W^{(j)}(x) \ll_{A,j} (1+x)^{-A}.
\end{equation}

\section{The core argument} \label{core}
We have the following orthogonality lemma.
\begin{lemma}  \label{orthog}
Let $p$ be prime and $\widetilde{\psi}_p=\big( \frac{\bullet}{p} \big)$. For $m,n \in \mathbb{N}$ with $(nm,p)=1$, we have 
\begin{equation}
\sideset{}{^*} \sum_{\substack{\chi \hspace{-0.2cm} \pmod{p} \\ \chi \neq \widetilde{\psi}_p }} \overline{\chi}(m) \chi(n)=\sum_{\substack{d \mid p \\ d \mid (m-n) }} \phi(d) \mu \Big(\frac{p}{d}  \Big )-\widetilde{\psi}_p(m) \widetilde{\psi}_p(n).
\end{equation}
\end{lemma}
Recalling \eqref{kiral} we have 
\begin{equation} \label{kiralquad}
\Big | L \Big (\frac{1}{2},f,\widetilde{\psi}_p \Big) \Big |^2 \ll_{\varepsilon} p^{\frac{3}{4}+\frac{\theta}{2}+\varepsilon}.
\end{equation}
Summing \eqref{approxsec} over all primitive characters $\chi \neq \widetilde{\psi}_p$ and applying Lemma \ref{orthog} we obtain 
\begin{equation} \label{first}
\sideset{}{^*} \sum_{\substack{\chi \hspace{-0.2cm} \pmod{p} \\ \chi \neq \widetilde{\psi}_p  }} L \Big (\frac{1}{2}, f, \chi \Big)  L \Big (\frac{1}{2}, f, \overline{\chi} \Big) =\mathcal{D}_1+\mathcal{D}_2+\mathcal{E},
\end{equation}
where
\begin{equation} \label{D1}
\mathcal{D}_1:=\sum_{d \mid p} \hspace{0.15cm} \phi(d) \mu \Big( \frac{p}{d} \Big)  \sum_{\substack{m \equiv n \hspace{-0.2cm} \pmod{d} \\ (mn,p)=1}}  \frac{a(m) a(n)}{\sqrt{mn}} W \Big( \frac{mn}{4p^2} \Big),
\end{equation}
\begin{equation} \label{D2}
\mathcal{D}_2:=\sum_{d \mid p} \hspace{0.15cm} \phi(d) \mu \Big( \frac{p}{d} \Big)  \sum_{\substack{m \equiv n \hspace{-0.2cm} \pmod{d} \\ (mn,p)=1}}  \frac{a(m) \widetilde{\psi}_p(m) a(n) \widetilde{\psi}_p(n)}{\sqrt{mn}} W \Big( \frac{mn}{4p^2} \Big).
\end{equation}
and 
\begin{equation} \label{E}
\mathcal{E}:=-\sum_{m,n} \frac{a(m) \widetilde{\psi}_p(m) a(n) \widetilde{\psi}_p(n)}{\sqrt{mn}} W \Big( \frac{mn}{4p^2} \Big)-\sum_{\substack{m,n \\ (mn,p)=1}} \frac{a(m) a(n)}{\sqrt{mn}} W \Big( \frac{mn}{4p^2} \Big).
\end{equation}
Combining \eqref{kiralquad}--\eqref{E} we obtain
\begin{equation} \label{auxmain}
\sideset{}{^*} \sum_{\chi \hspace{-0.2cm} \pmod{p}} L \Big (\frac{1}{2}, f, \chi \Big)  L \Big (\frac{1}{2}, f, \overline{\chi} \Big) =\mathcal{D}_1+\mathcal{D}_2+\mathcal{E}+O \big(p^{\frac{3}{4}+\frac{\theta}{2}+\varepsilon} \big).
\end{equation}
\begin{remark} \label{subconvexexponent}
Note that any subconvex exponent $3/4+\theta/2<1$ in \eqref{kiralquad} would be sufficient.
\end{remark}

\subsection{Main terms} \label{mainterms}
The main terms come from the diagonal $m=n$ in $\mathcal{D}_1$ and $\mathcal{D}_2$. This gives
\begin{equation} \label{premain}
\mathcal{M}(f,p):=2 \psi(p) \sum_{\substack{n \\ (n,p)=1 }} \frac{a(n)^2}{n} W \Big( \frac{n^2}{4p^2} \Big).
\end{equation}
Using \eqref{iwaniec} and \eqref{decaybd}, we can write 
\begin{equation} \label{mainterminit}
\mathcal{M}(f,p):=2 \psi(p) \sum_{n} \frac{a(n)^2}{n} W \Big( \frac{n^2}{4p^2} \Big)+O(p^{\frac{1}{2}+\varepsilon} ).
\end{equation}
Using Mellin inversion we obtain 
\begin{equation} \label{intmain}
\mathcal{M}(f,p):=2 \frac{\psi(p)}{2 \pi i} \int_{(3)}  4^s  L(1+2s,f \times \overline{f}) \hspace{0.1cm} p^{2s} \hspace{0.1cm} \widehat{W}(s) ds +O(p^{\frac{1}{2}+\varepsilon})
\end{equation}
where
\begin{equation*}
\widehat{W}(s)=\frac{1}{(2 \pi )^{2s}} \frac{\Gamma(\frac{k}{2}+s)^2}{\Gamma(\frac{k}{2})^2} \frac{1}{s}.
\end{equation*}
We shift the contour in \eqref{intmain} to $\Re s=-\frac{1}{4}+\varepsilon$ and pick up a double pole at $s=0$. Recalling \eqref{rankin} and applying the residue theorem, we see that 
\begin{equation} \label{Mintermed}
\mathcal{M}(f,p)=2 \psi(p) \lim_{s \rightarrow 0} \hspace{0.1cm} \frac{d}{ds} \Big ( \int_{\Gamma_0(4)  \backslash \mathbb{H}} y^k |f(\tau)|^2 h(\tau,s)  d \mu(\tau)  \Big )+O(p^{\frac{1}{2}+\varepsilon}),
\end{equation}
where
\begin{equation*}
h(\tau,s):=s^2  \frac{(4 \pi)^{2s+k}}{\Gamma(2s+k)}  \widehat{W}(s) 4^s p^{2s} E_{\infty}(\tau,1+2s). 
\end{equation*}
We interchange the derivative and the integral by \cite[Lemma~1.1, pg.~409]{La2} and \cite[Theorem~2, pg.~130]{B}. We then interchange the limit and integral using uniform convergence. Thus, for each fixed $\tau=x+iy \in \mathbb{H}$, it suffices to compute
\begin{equation} \label{task}
y^k |f(\tau)|^2 \lim_{s \rightarrow 0} \frac{d}{ds} h(\tau,s).
\end{equation}
Recalling \eqref{kronecker1} we have
\begin{equation} \label{expansion}
sE_{\infty}(\tau,1+2s)=\frac{1}{2} \frac{1}{4^{1+2s}} \frac{1}{L(2+4s,\chi_0)} \Big(\frac{\pi}{4} +\Big ( \pi \Big(\gamma- \frac{5}{2} \log 2 \Big)+2 \pi \log \Big( y^{-\frac{1}{4}} \Big | \frac{\eta(2 \tau)}{\eta(4 \tau)^2}  \Big | \Big) \Big ) s+O_{\tau}(s^2) \Big).
\end{equation}
Thus  
\begin{align} 
h(\tau,s)&=\frac{(4 \pi)^k \Gamma(s+\frac{k}{2})^2 p^{2s}}{8 \Gamma(\frac{k}{2})^2 \Gamma(2s+k)} \frac{1}{L(2+4s,\chi_0)} \nonumber \\
& \times \Big(\frac{\pi}{4} +\Big ( \pi \Big(\gamma- \frac{5}{2} \log 2 \Big)+2 \pi \log \Big( y^{-\frac{1}{4}} \Big | \frac{\eta(2 \tau)}{\eta(4 \tau)^2}  \Big | \Big) \Big ) s+O_{\tau}(s^2) \Big).
\label{hexp}
\end{align}
Performing \eqref{task} using \eqref{hexp} and the product rule, and substituting the result into \eqref{Mintermed},
we obtain the constants 
\begin{align}
c_1(f)&:= \frac{(4 \pi)^k}{\pi \Gamma(k)} \int_{\Gamma_0(4) \backslash \mathbb{H}} y^k |f(\tau)|^2 d \mu(\tau), \label{c1comp} \\
c_2(f)&:=\frac{(4 \pi)^{k+1}}{\pi^2 \Gamma(k)} \int_{\Gamma_0(4) \backslash \mathbb{H}} y^k |f(\tau)|^2  \log \Big( y^{-\frac{1}{4}} \Big | \frac{\eta(2 \tau)}{\eta(4 \tau)^2} \Big | \Big) d \mu(\tau) \nonumber \\
 &+ \frac{(4 \pi)^{k-1}}{\Gamma(k)} \Big(- \frac{8 \log 2}{3} +\frac{4 \Gamma^{\prime} (\frac{k}{2}) }{\Gamma(\frac{k}{2})}-\frac{4\Gamma^{\prime} (k) }{\Gamma(k)} -\frac{48 \zeta^{\prime}(2)}{\pi^2}+ 8 \Big(\gamma-\frac{5}{2} \log 2 \Big)    \Big) \nonumber \\
 & \times  \int_{\Gamma_0(4) \backslash \mathbb{H}} y^k |f(\tau)|^2 d \mu(\tau), \label{c2comp}
\end{align}
in Theorem \ref{secondmoment}.

\subsection{Error terms from $\mathcal{D}_1$, $\mathcal{D}_2$ and $\mathcal{E}$} \label{errorarg}
Let $V_{1,2}:(0,\infty) \rightarrow \mathbb{R}_{\geq 0}$ be smooth functions compactly supported on $[1,2]$ that satisfy 
\begin{equation} \label{smoothderiv}
V^{(j)}_{1,2}(x) \ll_{j} (\log 5p)^{2j} \ll_{j,\varepsilon} p^{\varepsilon}.
\end{equation}
In \eqref{D1}--\eqref{E} we place a smooth partition of unity, perform Mellin inversion and truncate the resulting integrals as in \cite[Section~5]{BM1}. We localise the variables to $M \leq m \leq 2M$ and $N \leq n \leq 2N$ satisfying  
\begin{equation} \label{dyadiccond}
N \geq M \geq 1 \quad (\text{by symmetry}) \quad \text{and} \quad \quad 1 \leq MN \leq p^{2+\varepsilon}.
\end{equation}
We find it is sufficient to bound $O(\log^2 p)$ sums of the shape
\begin{equation} \label{suff1}
\mathcal{S}_{N,M,p,d}:= \frac{d}{(MN)^{\frac{1}{2}}} \sum_{\substack{m \equiv n \hspace{-0.2cm} \pmod{d} \\ m \neq n  \\ (mn,p)=1}} a(m) a(n) V_1 \Big( \frac{m}{M} \Big) V_2 \Big( \frac{n}{N} \Big),
\end{equation}
and 
\begin{equation} \label{suff2}
\widetilde{\mathcal{S}}_{N,M,p,d}:= \frac{d}{(MN)^{\frac{1}{2}}} \sum_{\substack{m \equiv n \hspace{-0.2cm} \pmod{d} \\ m \neq n  \\ (mn,p)=1}} a(m) \widetilde{\psi}_p(m) a(n) \widetilde{\psi}_p(n) V_1 \Big( \frac{m}{M} \Big) V_2 \Big( \frac{n}{N} \Big),
\end{equation}
for each $d=1$ or $p$. Using \eqref{ram} we have the bound 
\begin{equation} \label{trivram}
\mathcal{S}_{N,M,p,d} ,\widetilde{\mathcal{S}}_{N,M,p,d} \ll (MN)^{\frac{1}{2}+\varepsilon}.
\end{equation}
We will not make use of \eqref{trivram}. We have the weaker, but unconditional version in the next lemma.

\begin{lemma} \label{ramsub}
For all $N \geq 20M$ we have
\begin{equation*}
\mathcal{S}_{N,M,p,d}, \widetilde{\mathcal{S}}_{N,M,p,d} \ll M^{\frac{1}{2}+\varepsilon} N^{\frac{3}{4}+\varepsilon}.
\end{equation*}
\end{lemma}
\begin{proof}
We have 
\begin{equation*}
\mathcal{S}_{N,M,p,d},  \widetilde{\mathcal{S}}_{N,M,p,d}  \ll  \frac{d}{(MN)^{\frac{1}{2}}} \sum_{\substack{M \leq m \leq 2M \\ (m,p)=1}} |a(m)| \sum_{\substack{n \equiv m \hspace{-0.1cm} \pmod{d} \\ N \leq n \leq 2N \\ (n,p)=1}} |a(n)|. 
\end{equation*}
Observe that if $d=p$ and $N \leq p/2$, then the congruence condition $m \equiv n \pmod{d}$ implies that $m=n$. However, this is 
not possible since $M \leq m \leq 2M$, $N \leq n \leq 2N$, and $N \geq 20M$. 
In all other cases we have $\# \{ N \leq n \leq 2N: n \equiv 0 \pmod{d} \} \asymp N/d$, 
and so
\begin{equation*}
\mathcal{S}_{N,M,p,d},  \widetilde{\mathcal{S}}_{N,M,p,d}
 \ll \frac{d}{(MN)^{\frac{1}{2}}} \frac{N^{\frac{5}{4}}}{d} \sum_{M \leq m \leq 2M} |a(m)|,
\end{equation*}
where the last inequality follows from positivity and \eqref{iwaniec}. The Lemma now follows by applying Cauchy--Schwarz and then \eqref{L2}.
\end{proof}

\begin{remark}
Instead of \eqref{iwaniec}, we could have applied the Conrey--Iwaniec bound \eqref{CIbound} in the proof of Lemma
\ref{ramsub} to obtain an improved bound. Optimality is not the purpose of this paper and we prefer to show that our method works (i.e. obtaining a power saving error term
in Theorem \ref{secondmoment}) with weaker inputs.
\end{remark}

We now remove the greatest common divisor condition in \eqref{suff1}. The removal of the gcd condition in \eqref{suff2} is automatic because of the presence of the character. 
\begin{lemma} \label{trivlem}
For all $M,N$ satisfying \eqref{dyadiccond} and $d=1$ or $p$ we have 
\begin{equation} \label{triv1}
\mathcal{S}_{N,M,p,d}:= \frac{d}{(MN)^{\frac{1}{2}}} \sum_{\substack{m \equiv n \hspace{-0.2cm} \pmod{d} \\ m \neq n}} a(m) a(n) V_1 \Big( \frac{m}{M} \Big) V_2 \Big( \frac{n}{N} \Big)+O(p^{\frac{1}{2}+\varepsilon}).
\end{equation}
\end{lemma}
\begin{proof}
We prove \eqref{triv1} by estimating the contribution from pairs in the set
\begin{equation*}
\mathcal{B}:=\big \{(m,n): m \equiv 0 \pmod{p} \quad \text{or} \quad n \equiv 0 \pmod{p}  \quad \text{such that} \quad m \neq n \big \}.
\end{equation*}
If $d=1$ and $1 \leq M \leq (p-1)/2$, then there are no pairs with $m \equiv 0 \pmod{p}$. Pairs $(m,n)$ with $n \equiv 0 \pmod{p}$ contribute $O(p^{\frac{1}{2}+\varepsilon})$ by \eqref{iwaniec}, Cauchy-Schwarz (on the $m$ sum) and \eqref{L2}. 

If $1 \leq M \leq (p-1)/2$ and $d=p$ then there are no $(m,n)$ such that $m \equiv n \equiv 0 \pmod{p}$.

Now suppose that $(p-1)/2 \leq M \leq p^{1+\frac{\varepsilon}{2}}$ and $d=1$. The contribution from all $(m,n)$ with $m \equiv 0 \pmod{p}$ and $n \not \equiv 0 \pmod{p}$ is $O(p^{\frac{1}{2}+\varepsilon})$ by a similar argument to the above. The contribution from all $(m,n)$ with $n \equiv 0 \pmod{p}$ and $m \not \equiv 0 \pmod{p}$ is at most $O(p^{\frac{1}{2}+\varepsilon})$ by a similar argument to the above. The contribution from $(m,n)$ with $m \equiv n \equiv 0 \pmod{p}$ is negligible by 
\eqref{iwaniec}.

If $(p-1)/2 \leq M \leq p^{1+\frac{\varepsilon}{2}}$ and $d=p$ then the contribution from $(m,n)$ with $m \equiv n \equiv 0 \pmod{p}$ is
$O(p^{\frac{1}{2}+\varepsilon})$ by \eqref{iwaniec}.
\end{proof}

\begin{lemma}  \label{plain1}
For all $M,N$ satisfying \eqref{dyadiccond} we have 
\begin{equation*}
\mathcal{S}_{N,M,p,1} \ll_{\varepsilon} p^{\frac{1}{2}+\varepsilon}.
\end{equation*}
\end{lemma}
\begin{proof}
We write \eqref{triv1} as 
\begin{align} 
\mathcal{S}_{N,M,p,1}&= \frac{1}{(MN)^{\frac{1}{2}}}  \Big ( \sum_{m} a(m) V_1 \Big( \frac{m}{M} \Big) \Big ) \Big (\sum_n a(n) V_2 \Big( \frac{n}{N} \Big) \Big ) \nonumber \\
&-\frac{1}{(MN)^{\frac{1}{2}}} \sum_{m} a(m)^2 V_1 \Big(\frac{m}{M} \Big)  V_2 \Big( \frac{m}{N} \Big)+O(p^{\frac{1}{2}+\varepsilon}). \label{easy}
\end{align}
Applying partial summation, \eqref{L2} and \eqref{wilton} guarantee that all but the last term in \eqref{easy} are $O(p^{\varepsilon})$.
\end{proof}

\begin{lemma} \label{kiraltwist}
Let $f \in \mathcal{S}_k(4)$ have normalised coefficients $a(n)$. Let $\chi$ be a primitive character modulo $p$ and $V:(0,\infty) \rightarrow \mathbb{R}_{\geq 0}$  be a smooth function with support contained in $[1,2]$. Then for $X \geq 1$ we have 
\begin{equation*}
\sum_m a(m) \chi(m) V \Big( \frac{m}{X} \Big) \ll_{\varepsilon}  X^{\frac{1}{2}} p^{\frac{3}{8}+\frac{\theta}{4}+\varepsilon}+X^{\frac{3}{8}} p^{\frac{1}{2}+\varepsilon},
\end{equation*}
where $\theta$ represents the best progress toward the Ramanujan--Petersson conjecture for weight zero Maass forms.
\end{lemma}
\begin{proof}
Let $S,S_j$ for $j=1,2,3$ be defined as in \cite[Proposition~3]{K}. Each $S_j$ is an averaged shifted convolution sum, depending on $p,X,f, \ell_1$ and $\ell_2$. Let $L \geq 1$ be a parameter to be chosen later. Taking $\chi^{\prime}$ to be $\chi$ in \cite[(11)]{K} yields 
\begin{equation} \label{adaptineq}
\frac{L^2}{(\log L)^2} \Big | \sum_{m} a(m) \chi(m) V \Big(\frac{m}{X} \Big)  \Big |^2  \leq \phi(p) \sum_{\substack{L \leq \ell_1,\ell_2 \leq 2L \\ \ell_j \text{ prime} }} \chi(\ell_1) \overline{\chi(\ell_2)}(S_1+S_2+S_3).
\end{equation}
Note the square is missing in \cite[(11)]{K}. Invoking the bounds for the $S_j$ given in \cite[Proposition~4 and Theorem~16]{K}, the right side of \eqref{adaptineq} is bounded by
\begin{equation*}
\ll p (XL+X^{\frac{1}{2}} L^3+X^{1+\varepsilon} L^{3+\varepsilon} p^{\theta-\frac{1}{2}}  ).
\end{equation*}
Note that the statement of \cite[Theorem 16]{K} has a typographic error (inequality is missing a $Q^{\theta}$). However, the correct bound is stated at the end of proof (cf. \cite[pg.~713]{K}) and that is the one used here. Thus 
\begin{equation*}
\Big | \sum_m a(m) \chi(m) V \Big( \frac{m}{X} \Big) \Big | \ll L^{\varepsilon} \Big( \frac{p^{\frac{1}{2}} X^{\frac{1}{2}}}{L^{\frac{1}{2}}}+L^{\frac{1}{2}} ( X^{\frac{1}{4}} p^{\frac{1}{2}} +X^{\frac{1}{2}} p^{\frac{1}{4}+\frac{\theta}{2}})  \Big).
\end{equation*}
Choosing 
\begin{equation*}
L:=\frac{p^{\frac{1}{2}} X^{\frac{1}{2}}  }{X^{\frac{1}{4}} p^{\frac{1}{2}}+X^{\frac{1}{2}} p^{\frac{1}{4} +\frac{\theta}{2}} } \geq 1,
\end{equation*}
yields the result.
\end{proof}

\begin{lemma} \label{tilde1}
For $M,N$ satisfying \eqref{dyadiccond} we have 
\begin{equation*}
\widetilde{\mathcal{S}}_{N,M,p,1} \ll_{\varepsilon} p^{\varepsilon} ( (MN)^{\frac{1}{4}} p^{\frac{3}{8}+\frac{\theta}{4}}+N^{\frac{1}{4}} M^{\frac{3}{16}} p^{\frac{7}{16}+\frac{\theta}{8}}+(MN)^{\frac{3}{16}} p^{\frac{1}{2}} ).
\end{equation*}
\end{lemma}
\begin{proof}
We write \eqref{suff2} as 
\begin{align}
\widetilde{\mathcal{S}}_{N,M,p,1}&= \frac{1}{(MN)^{\frac{1}{2}}}  \Big ( \sum_{m} a(m) \widetilde{\psi}_p(m) V_1 \Big( \frac{m}{M} \Big) \Big ) \Big (\sum_n a(n) \widetilde{\psi}_p(n) V_2 \Big( \frac{n}{N} \Big) \Big ) \nonumber \\
&-\frac{1}{(MN)^{\frac{1}{2}}} \sum_{\substack{m \\ (m,p)=1}} a(m)^2 V_1 \Big(\frac{m}{M} \Big) V_2 \Big( \frac{m}{N} \Big).  \label{resuff}
\end{align}
Observe that \eqref{L2} implies that
\begin{equation} \label{L2trivial}
\frac{1}{(MN)^{\frac{1}{2}}}\sum_{\substack{m \\ (m,p)=1}} |a(m)|^2 V_1 \Big(\frac{m}{M} \Big) V_2 \Big( \frac{m}{N} \Big) \ll p^{\varepsilon}.
\end{equation}
By the Cauchy--Schwarz inequality and \eqref{L2} we have 
\begin{equation} \label{ramreplace}
\sum_{y} a(y) \widetilde{\psi}_p(y) V_1 \Big( \frac{y}{X} \Big)  \ll X^{1+\varepsilon}.
\end{equation}
Using Lemma \ref{kiraltwist}, \eqref{ramreplace} and the fact 
\begin{equation*}
\min (A+B,C) \leq \sqrt{AC}+\sqrt{BC} \quad \text{for} \quad A,B,C>0,
\end{equation*}
we obtain 
\begin{align} \label{nbound}
\sum_{y} a(y) \psi_p(y) V_2 \Big( \frac{y}{X} \Big) 
&\ll \min ( X^{\frac{1}{2}} p^{\frac{3}{8}+\frac{\theta}{4}+\varepsilon}+X^{\frac{3}{8}} p^{\frac{1}{2}+\varepsilon}, X^{1+\varepsilon}) \nonumber  \\
& \ll p^{\varepsilon} ( X^{\frac{3}{4}} p^{\frac{3}{16}+\frac{\theta}{8}}+X^{\frac{11}{16}} p^{\frac{1}{4}}).
\end{align}
After applying the triangle inequality in \eqref{resuff}, we insert \eqref{L2trivial} and \eqref{nbound} into \eqref{resuff} to obtain the result.
\end{proof}

We are now left to treat
\begin{equation} \label{keyquant}
\widetilde{\mathcal{S}}_{N,M,p,p}=\mathcal{S}_{N,M,p,p}.
\end{equation}
It will be convenient to consider $M \leq N \leq 20 M$ and $N \geq 20M$ separately. 

\begin{lemma} \label{secondregime}
For all $M,N$ satisfying \eqref{dyadiccond} and $M \leq N \leq 20 M$, we have 
\begin{equation*}
\mathcal{S}_{N,M,p,p} \ll_{\varepsilon} p^{\frac{1}{2}+\theta+\varepsilon}.
\end{equation*}
\end{lemma}
\begin{proof}
Lemma \ref{trivlem} and the bounds for $S_2$ and $S_3$ on \cite[pg.~713]{K} together imply that 
\begin{equation} 
\mathcal{S}_{N,M,p,p} \ll_{\varepsilon} p^{\frac{1}{2}+\theta+\varepsilon}.
\end{equation}
\end{proof}

Now consider $N \geq 20M$. The following lemma in conjunction with parts of Lemma \ref{ramsub} and Proposition \ref{spectralest}  will essentially serve as a weaker, but unconditional replacement for the ``trivial bound" in \eqref{trivram}.
\begin{lemma} \label{subby}
Let $\delta_0,\delta_1>0$ and suppose $M,N$ satisfy \eqref{dyadiccond} as well as that $M \ll p^{\frac{1}{2}-\delta_0}$ and $N \gg p^{1+\delta_1}$. Then
\begin{equation*}
\mathcal{S}_{N,M,p,p} \ll_{\varepsilon,\delta_0,\delta_1} p^{\varepsilon} \Big(  p^{\frac{3}{4}} M^{\frac{1}{2}}+\frac{p^{\frac{3}{2}}}{N^{\frac{1}{2}}} \Big).
\end{equation*}
\end{lemma}

\begin{lemma} \label{valest}
Let $\delta>0$ and suppose that $M,N$ satisfy \eqref{dyadiccond} as well a $M/N<p^{-1-\delta}$. Then we have 
\begin{equation*}
\mathcal{S}_{N,M,p,p} \ll_{\varepsilon, \delta} p^{\varepsilon} \Big( p^{\frac{3}{2}} \Big(\frac{M}{N} \Big)^{\frac{1}{2}}+p^{\frac{1}{2}} \Big).
\end{equation*}
\end{lemma}

\begin{prop} \label{critandupper} 
Let $M,N$ be as in \eqref{dyadiccond} and also satisfy
\begin{equation} \label{vrange}
\quad p^{\frac{3}{2}-\frac{9}{100}} \leq N \leq p^{\frac{3}{2}+\frac{9}{100}}.
 \end{equation}
 Then 
\begin{align*} 
\mathcal{S}_{N,M,p,p} \ll_{\varepsilon} p^{\varepsilon} \Big( & \frac{M^{\frac{1}{2}}}{N^{\frac{1}{2}}} p^{\frac{3}{2}-\frac{1}{108}}+\frac{M^{\frac{1}{2}}}{N^{\frac{1}{4}}} p^{1+\frac{25}{216}}+ \frac{M^{\frac{1}{4}}}{N^{\frac{1}{2}}} p^{\frac{3}{2}+\frac{25}{216}}  \nonumber \\
&+\frac{M^{\frac{1}{4}}}{N^{\frac{1}{4}}} p^{1+\frac{13}{54}}+ \frac{M^{\frac{1}{4}}}{N^{\frac{1}{8}}} p^{1+\frac{23}{432}}+M^{\frac{1}{4}} p^{1-\frac{29}{216}}+p^{1-\frac{1}{10}} \Big). \label{dp}
\end{align*}
\end{prop}
We defer the proof of Proposition \ref{critandupper} (assuming the truth of Theorem \ref{critrangethm}) 
to Section \ref{ressec}. The proof of Theorem \ref{critrangethm} is given in Section \ref{critsec}.

The last estimate we require in order to obtain bounds for \eqref{keyquant} uses spectral methods. For $\ell_1,\ell_2, h \in \mathbb{Z}_{\geq 1}$, define the shifted convolution sum
\begin{equation} \label{Dshift}
\mathcal{D}(\ell_1,\ell_2,h,N,M):=\sum_{\ell_1 n-\ell_2 m=h} a(m) a(n) V_1 \Big( \frac{\ell_2 m}{M} \Big) V_2 \Big( \frac{\ell_1 n}{N} \Big),
\end{equation}
and for $d$ a positive integer, 
\begin{equation} \label{Sshift}
\mathcal{S}(\ell_1,\ell_2,d,N,M):=\sum_{r \geq 1} \mathcal{D}(\ell_1,\ell_2,rd,N,M).
\end{equation}
Lemma \ref{trivlem} gives (for all $N \geq 20 M$), 
\begin{equation} \label{simple}
\mathcal{S}_{N,M,p,p}= \frac{p}{(MN)^{\frac{1}{2}}} \mathcal{S}(1,1,p,N,M)+O(p^{\frac{1}{2}+\varepsilon}).
\end{equation}
\begin{prop} \label{spectralest} 
Suppose $M,N$ are as in \eqref{dyadiccond} and $N \geq 20M$. Then we have 
\begin{equation*}
\mathcal{S}(\ell_1,\ell_2,p,N,M) \ll_{\varepsilon} p^{\varepsilon}(\ell_1 \ell_2,p)^{\frac{1}{2}} \Big( \frac{N}{p^{\frac{1}{2}}}+\frac{N^{\frac{5}{4}} M^{\frac{1}{4}} }{p}+\frac{N^{\frac{3}{4}} M^{\frac{1}{4}} }{p^{\frac{1}{4}}}+\frac{NM^{\frac{1}{2}}}{p^{\frac{3}{4}}} \Big).
\end{equation*}
Thus, 
\begin{equation} \label{spectralest1}
\mathcal{S}_{N,M,p,p} \ll_{\varepsilon} p^{\varepsilon} \Big( \frac{(Np)^{\frac{1}{2}}}{M^{\frac{1}{2}}}+\frac{N^{\frac{3}{4}}}{M^{\frac{1}{4}}}+\frac{p^{\frac{3}{4}} N^{\frac{1}{4}} }{M^{\frac{1}{4}}}+p^{\frac{1}{4}} N^{\frac{1}{2}} \Big)+p^{\frac{1}{2}+\varepsilon}.
\end{equation}
\end{prop}
We defer the proof of Proposition \ref{spectralest} to Section \ref{spectralsec}.

\subsection{Coefficients in residue classes} \label{ressec}
Here we prove Lemmas \ref{subby} and \ref{valest}, as well as Proposition \ref{critandupper}. 
The starting point for these results is Voronoi summation in the long variable.

For $r,v \in \mathbb{N}$, let 
\begin{equation*}
\mathrm{c}_v(r)=\sum_{\substack{u=1 \\ (u,v)=1 } }^{v} e \Big( \frac{ru}{v} \Big).
\end{equation*} 
denote the usual Ramanujan sum.  Ramanujan sums are multiplicative in the modulus variable,
\begin{equation} \label{multram}
\mathrm{c}_{st}(r)=\mathrm{c}_{s}(r) \mathrm{c}_{t}(r) \quad \text{for} \quad (s,t)=1.
\end{equation}

\begin{proof}[Proof of Lemma \ref{valest}]
Applying Lemma \ref{trivlem}, we write the right side of \eqref{triv1} as the sum of three subsums 
\begin{equation} \label{subsums}
\mathcal{S}_{N,M,p,p}^{\ell}=\frac{p}{(MN)^{1/2}} \sum_{\substack{m \equiv n \hspace{-0.2cm} \pmod{p} \\ m-n \equiv \ell \hspace{-0.2cm} 
\pmod{4}}} a(m) a(n) V_1 \Big( \frac{m}{M} \Big) V_2 \Big( \frac{n}{N} \Big), \quad \ell \in \{0,-1,1\},
\end{equation}   
incurring an error of $O(p^{\frac{1}{2}+\varepsilon})$. Kohnen's plus space condition in \eqref{kohnen}
explains why we need only consider $\ell \in \{0,-1,1\}$ in \eqref{subsums}.

Here, the condition $m \neq n$ is moot. In order to apply Lemma \ref{Voronoi} we will need the moduli
occurring in additive characters
to be divisible by $4$.  
Hence, we will use orthogonality in the form 
\begin{equation} \label{orthog1}
\frac{1}{2p} \mathrm{c}_4(r)  \big(1+\mathrm{c}_p(r) \big)=
\begin{cases}
0 & \text{if} \quad 2 \nmid r \\
-1 & \text{if} \quad 2 \mid r \text{ and } 4 \nmid r \\
1&  \text{if} \quad r \equiv 0 \pmod{4} 
\end{cases}
\times 
\begin{cases}
0 & \text{if} \quad r \not \equiv 0 \pmod{p}  \\
1 & \text{if} \quad r \equiv 0 \pmod{p}.
\end{cases}
\end{equation}
By \eqref{multram}, the left side \eqref{orthog1} is 
\begin{equation}  \label{orthog2}
\frac{1}{2p} \mathrm{c}_4(r)+\frac{1}{2p} \mathrm{c}_{4p}(r).
\end{equation}
Without loss of generality we now focus on the case $\ell=0$ in \eqref{subsums} and use \eqref{orthog2} to detect 
both congruence conditions in \eqref{subsums}.
The $\ell \in \{\pm 1 \}$ cases follows from similar arguments (replace $\mathrm{c}_4(r)$ with $1-\mathrm{c}_4(r)$ above).

We apply \eqref{orthog2} to $\mathcal{S}^{0}_{N,M,p,p}$ to remove both congruences. 
Thus $\mathcal{S}^{0}_{N,M,p,p}$ is the sum of the following two expressions
\begin{equation} \label{suff}
\frac{1}{2 (MN)^{1/2}}
 \sideset{}{^*} \sum_{j \hspace{-0.2cm} \pmod{4p} } \hspace{0.15cm} \Big (\sum_{m} a(m) e \Big({\frac{jm}{4p}} \Big) V_1 \Big( \frac{m}{M} \Big) \Big)
 \Big ( \sum_{n} a(n) e \Big({-\frac{jn}{4p}} \Big) V_2 \Big( \frac{n}{N} \Big) \Big ) ,
\end{equation}
and
\begin{equation} \label{firstterm}
\frac{1}{2(MN)^{1/2}} \hspace{0.15cm}  \sideset{}{^*} \sum_{j \hspace{-0.2cm} \pmod{4} }
\Big (\sum_{m} a(m) e \Big(\frac{jm}{4} \Big) V_1 \Big( \frac{m}{M} \Big) \Big ) \Big ( \sum_{n} a(n) e \Big({-\frac{jn}{4}} \Big) 
V_2 \Big( \frac{n}{N} \Big) \Big ).
\end{equation}
Partial summation and \eqref{wilton} guarantees that \eqref{firstterm} is negligible.
Let $u$ be any integer with $u \equiv -j \pmod{4p}$.
Applying Lemma \ref{Voronoi} to the $n$ summation in \eqref{suff} we obtain 
\begin{equation*} 
\sum_n a(n) e \Big({-\frac{jn}{4p}} \Big) V_2 \Big( \frac{n}{N} \Big)=\frac{N}{4p}  
\nu_{\theta} (\gamma) \sum_n a(n) e \Big({\frac{n \overline{j}}{4p}} \Big) \mathring{V}_{2} \Big( \frac{nN}{16 p^2} \Big),
\end{equation*}
where 
\begin{equation*}
\gamma:=\begin{pmatrix}
u & * \\
4p & *
\end{pmatrix} \in \Gamma_0(4p).
\end{equation*}
Observe that
\begin{equation*}
\nu_{\theta}(\gamma)=\varepsilon^{-1}_{-\overline{j}} \Big ( \frac{4p}{-\overline{j}} \Big)=\varepsilon^{-1}_{-j} \Big ( \frac{4p}{j} \Big).
\end{equation*}
Thus \eqref{suff} becomes 
\begin{equation} \label{penul}
\frac{N^{1/2}}{8 p M^{1/2}} \sum_m a(m) V_1 \Big( \frac{m}{M} \Big) \sum_n a(n) \mathring{V}_{2} \Big( \frac{nN}{16 p^2} \Big)
  K_{1}(-m,-n;4p). 
\end{equation}

Given the rapid decay of the Hankel transform we can truncate the $n$ summation $1 \leq n \leq p^{2+\varepsilon}/N$ in \eqref{penul} up to negligible error. 
By hypothesis we have that $N>Mp^{1+\delta}$, and thus any $m \asymp M$ and  $1 \leq n \leq p^{2+\varepsilon}/N$ satisfies 
$(mn,p)=1$.
Then by Lemma \ref{twistlem} and \eqref{salieval} we have 
\begin{equation} \label{Kweil}
 |K_1(-m,-n;4p)| \leq 4 p^{\frac{1}{2}}.
\end{equation}
Inserting \eqref{Kweil} into \eqref{penul}, we see that \eqref{penul} is
\begin{equation*}
\ll \frac{N^{1/2}}{ p^{1/2} M^{1/2}}  \Big (\sum_m  \Big \lvert a(m) V_1 \Big( \frac{m}{M} \Big) \Big \rvert \Big )   \Big (\sum_n \Big \lvert a(n) \mathring{V}_{2} \Big( \frac{nN}{16 p^2} \Big) \Big \rvert \Big ).
\end{equation*}
Both the $m$ and $n$ summations can be estimated trivially using Cauchy--Schwarz and \eqref{L2}.
\end{proof}

\begin{proof}[Proof of Lemma~\ref{subby}]
Repeat the proof of Lemma \ref{valest} to the display \eqref{penul}. Note this incurs an error of $O(p^{\frac{1}{2}+\varepsilon})$. 
As in the last proof, we give details of the argument when $\ell=0$.
The $\ell \in \{\pm 1\}$ cases follows from similar arguments.

Interchanging the $m$ and $n$ summation in \eqref{penul} gives 
\begin{equation} \label{swappy} 
\frac{N^{1/2}}{8 p M^{1/2}} \sum_n a(n) \mathring{V}_{2} \Big( \frac{nN}{16 p^2} \Big)
 \sum_m a(m) V_1 \Big( \frac{m}{M} \Big)  K_{1}(-m,-n;4p). 
\end{equation}
We apply Cauchy--Schwarz (now in the $n$ variable) and Lemma \ref{twistlem}, and we see that \eqref{swappy} is
\begin{equation} \label{swappy2}
\ll  \frac{1}{M^{\frac{1}{2}}} \max_{e,f,g \hspace{-0.15cm} \pmod{4}} \Big(\sum_{\substack{M \leq m_1,m_2 \leq 2M \\ m_1 \equiv e \hspace{-0.15cm} \pmod{4} \\ m_2 \equiv f \hspace{-0.15cm}  \pmod{4}}} |a(m_1) a(m_2) | \Big | \sum_{\substack{ 1 \leq n \ll p^{2+\varepsilon}/N \\ n \equiv g \hspace{-0.15cm} \pmod{4}}}  S(m_1,\overline{16} n;p) \overline{S(m_2,\overline{16} n;p)} \Big  |   \Big)^{\frac{1}{2}}
\end{equation}
Using \eqref{salieval} the summation over $n$ becomes 
\begin{equation} \label{intermedsum}
p \sum_{\substack{1 \leq n \ll p^{2+\varepsilon}/N \\ n \equiv g \hspace{-0.15cm} \pmod{4}}} \hspace{0.1cm} \sum_{\substack{x,y \hspace{-0.15cm} \pmod{p} \\ x^2 \equiv \overline{16} m_1 n \hspace{-0.15cm} \pmod{p} \\ y^2 \equiv  \overline{16} m_2 n \hspace{-0.15cm} \pmod{p}  }} e \Big( \frac{2(x-y)}{p} \Big) .
\end{equation}

When $m_1=m_2$, we estimate \eqref{intermedsum} trivially by $p^{3+\varepsilon}/N$. Then using \eqref{L2} yields a contribution of  $p^{3/2+\varepsilon}/N^{\frac{1}{2}}$ to \eqref{swappy2}.

Since $m_1,m_2 \asymp M$ and $M \ll p^{1/2-\delta_0}$ by hypothesis, we have that $m_1 \neq m_2$ implies $m_1 \not \equiv m_2 \pmod{p}$. Thus we can write \eqref{intermedsum} in terms of complete sums
\begin{equation*}
p \sum_{n=0}^{p-1} \sum_{\substack{x,y \hspace{-0.15cm} \pmod{p} \\ x^2 \equiv \overline{16} m_1 n \hspace{-0.15cm} \pmod{p} \\ y^2 \equiv  \overline{16} m_2 n \hspace{-0.15cm} \pmod{p}  }} e \Big( \frac{2(x-y)}{p} \Big)  \sum_{\substack{1 \leq w \leq p^{2+\varepsilon}/N \\ w \equiv g \pmod{4} }} \frac{1}{p} \sum_{t=0}^{p-1} e \Big( \frac{t(n-w)}{p} \Big).
\end{equation*}
After interchanging the $t$ and $w$ summation, and estimating the sum over $w$ in the usual way, it suffices to estimate the maximum of
\begin{equation} \label{supfourier}
p \log p \max_{t \in \mathbb{Z} \cap [0,p-1]} \Big | \sum_{n=0}^{p-1} \sum_{\substack{x,y \hspace{-0.15cm} \pmod{p} \\ x^2 \equiv \overline{16} m_1 n \hspace{-0.15cm} \pmod{p} \\ y^2 \equiv  \overline{16} m_2 n \hspace{-0.15cm} \pmod{p}  }}  e \Big(\frac{2(x-y)}{p} \Big) e \Big(\frac{tn}{p} \Big) \Big |.
\end{equation}
The $n=0$ term contributes $O(p \log p)$. Let $2 \leq a \leq p-2$ (since $m_1 \not \equiv m_2 \pmod{p})$ be such that $a^2 \equiv m_1 \overline{m_2} \pmod{p}$. We have $(x \overline{y})^2 \equiv m_1 \overline{m_2} \pmod{p}$.  This implies that if  $y \equiv v \pmod{p}$ for some $1 \leq v \leq p-1$, then $x \equiv \pm av \pmod{p}$ and $n \equiv 16 \overline{m_2} v^2 \equiv 16 \overline{m_1} a^2 v^2  \pmod{p}$. Conversely, given $1 \leq v \leq p-1$, let $y \equiv v \pmod{p}$ and $x \equiv \pm av \pmod{p}$. Then $(x \overline{y})^2 \equiv m_1 \overline{m_2} \pmod{p}$. Moreover $16 \overline{m_2} y^2 \equiv 16 \overline{m_1} x^2 \pmod{p}$. Thus \eqref{supfourier} becomes (with $n=0$ excluded)
\begin{equation*}
p \log p \max_{t \in \mathbb{Z} \cap [0,p-1]} \Big | \sum_{v=1}^{p-1} e \Big( \frac{16 t \overline{m_2} v^2+(\pm a-1)v}{p} \Big ) \Big |.
\end{equation*}
If $t \neq 0$, then the above line is bounded above by $p^{\frac{3}{2}} \log p$ since it is a Gauss sum. If $t=0$, then it is a Ramanujan sum and is $p \log p$ since $2 \leq a \leq p-2$. Applying Cauchy--Schwarz and \eqref{L2} we obtain a contribution of $p^{\frac{3}{4}+\varepsilon} M^{\frac{1}{2}}$ to \eqref{swappy2} from this case.
\end{proof}

Now we prove Proposition \ref{critandupper} assuming the validity of Theorem \ref{critrangethm}, whose proof is deferred to Section \ref{critsec}.

\begin{proof}[Proof of Proposition~\ref{critandupper} (assuming Theorem \ref{critrangethm})]
Repeat the proof of Lemma \ref{valest} to the display \eqref{penul}. Note this incurs an error $O( p^{\frac{1}{2}+\varepsilon})$. 
As in the previous two proofs, we give details of the argument when $\ell=0$.
The $\ell \in \{\pm 1\}$ cases follows from similar arguments. 

We may restrict the $n$ summation in \eqref{penul} to $n \geq p^{\frac{2}{5}}$ using the same argument 
following \eqref{penul} with error $O(p^{\frac{9}{10}+\varepsilon})$. Define the sets
\begin{equation*}
\mathcal{N}_{0}(f):=\{n \in \mathbb{N}:  0 \leq |a(n)| \leq 1 \},
\end{equation*}
and for all $r \geq 1$,
\begin{equation*}
\mathcal{N}_{r}(f):=\{n \in \mathbb{N}:  2^{r-1}<|a(n)| \leq 2^r \}.
\end{equation*}
We decompose \eqref{penul} 
into $O(\log^2 p)$ subsums 
\begin{equation}
\frac{N^{1/2}}{8 p M^{1/2}}  \sum_{m} a(m) V_1 \Big( \frac{m}{M} \Big) 
\sum_{\substack{n \asymp A \\ n \in \mathcal{N}_r(f)}} a(n) 
\mathring{V}_{2} \Big( \frac{nN}{16 p^2} \Big) K_1(-m,-n;4p),   \label{penuluv}
\end{equation}
where
\begin{equation} \label{Acond}
p^{2/5} \leq A \leq \frac{p^{2+\varepsilon}}{N}.
\end{equation}

Observe that \eqref{L2} implies that for any $X \geq 1$ we have
\begin{equation} \label{size}
|\mathcal{N}_r(f) \cap [0,X] | \ll \frac{X^{1+\varepsilon}}{2^{2r}}.
\end{equation}
Applying Cauchy--Schwarz to the $m$-summation in \eqref{penuluv} and then using \eqref{L2} 
we see that \eqref{penuluv} is
\begin{equation} \label{CR}
\ll \frac{N^{\frac{1}{2}}}{p^{1-\varepsilon}} \Big( \sum_{M \leq m \leq 2M} \Big | \sum_{\substack{n \asymp A \\ n \in \mathcal{N}_r(f)}} a(n) K_1(-m,-n;4p)  \mathring{V}_{2} \Big( \frac{nN}{16 p^2} \Big) \Big |^2 \Big)^{\frac{1}{2}}. 
\end{equation}
Expanding the square and interchanging the summations, the expression inside the square root  in \eqref{CR} becomes
\begin{align} 
& \sum_{\substack{n_1,n_2 \asymp A \\ n_1,n_2 \in \mathcal{N}_r(f)}} a(n_1) \overline{a(n_2)} \mathring{V}_{2} \Big( \frac{n_1 N}{16 p^2} \Big) \overline{\mathring{V}_{2} \Big( \frac{n_2N}{16 p^2} \Big)}  \nonumber \\ & \times \sum_{M \leq m \leq 2 M}  K_1(-m,-n_1;4p) \overline{K_1(-m,-n_2;4p)}. 
\label{bracketed}
\end{align}
We use the bound $|K_1(-m,-n,p)| \leq 4p^{1/2}$ (we have $(mn,p)=1$ in the relevant ranges), Cauchy-Schwarz, and \eqref{size} to estimate \eqref{bracketed} by
\begin{equation*}
\ll \frac{p^{5}}{N^2} \frac{M}{2^{2r}},
\end{equation*}
for any $A$ satisfying \eqref{Acond}.
Inserting this into \eqref{CR}, we see that \eqref{penuluv} is
\begin{equation} \label{trivv}
\ll \frac{1}{2^r} p^{\frac{3}{2}} \Big( \frac{M}{N} \Big)^{\frac{1}{2}}.
\end{equation}

We now estimate \eqref{bracketed} non-trivially to obtain another upperbound for \eqref{penuluv}. Applying triangle inequality to \eqref{bracketed}, we see that \eqref{bracketed} is
\begin{equation} \label{newintermed}
\ll  2^{2r} \sum_{\substack{ n_1,n_2 \asymp A  \\ n_1,n_2 \in \mathcal{N}_r(f) } } \Big | \sum_{M \leq m \leq 2 M}  K_1(-m,-n_1;4p) \overline{K_1(-m,-n_2;4p)} \Big |. 
\end{equation}
By positivity we can extend  the summation over all $n_1,n_2 \asymp A$ in \eqref{newintermed}. Applying Lemma \ref{twistlem}, \eqref{salieval} and the fact that $p \equiv 1 \pmod{4}$ we see that the summand in the $m$ summation is  
\begin{equation} \label{id}
K_1(-m,-n_1;4) \overline{K_1(-m,-n_2;4)} S(m,\overline{16} n_1;p) \overline{S(m,\overline{16} n_2;p)}.
\end{equation}
Note that $K_1(-m,-n_1;4) \overline{K_1(-m,-n_2;4)}$ is an absolute constant depending only on $m,n_1,n_2$ modulo $4$. Thus we rewrite \eqref{bracketed} 
so that each summation variable runs in a fixed congruence class modulo $4$. Thus it suffices to bound
\begin{equation} \label{intermed}
2^{2r} \sum_{\substack{n_1,n_2 \asymp A \\ n_1 \equiv e \hspace{-0.15cm} \pmod{4} \\ n_2 \equiv f \hspace{-0.15cm}  \pmod{4}}} \Big | \sum_{\substack{M \leq m \leq 2M \\ m \equiv g \hspace{-0.15cm} \pmod{4}}}  S(m,\overline{16} n_1;p) \overline{S(m,\overline{16} n_2;p)} \Big  |,
\end{equation}
for $e,f,g$ modulo $4$. The hypothesis on $N$ ensures that
\begin{equation*}
p^{\frac{41}{100}} \leq \frac{p^{2}}{N} \leq p^{\frac{59}{100}}.
\end{equation*}
Recall \eqref{Acond}.
The bound in Theorem \ref{critrangethm} applied to \eqref{intermed} is monotone increasing in $A$, thus by Theorem \ref{critrangethm} and Remark \ref{mod4} we can bound \eqref{intermed} by 
\begin{equation} \label{keyineq2}
\ll 2^{2r} \Big(\frac{M}{N^2} p^{\frac{134}{27}}+\frac{M}{N} p^{\frac{187}{54}}+\frac{p^{\frac{295}{54}}}{N^2}+\frac{p^{\frac{107}{27}}}{N}+\frac{p^{\frac{347}{108}}}{N^{\frac{1}{2}}}+p^{\frac{133}{54}} \Big),
\end{equation}
uniformly in $A$ satisfying \eqref{Acond}. Inserting \eqref{keyineq2} into \eqref{CR}, we see that \eqref{penuluv} is
\begin{equation} \label{oldbd}
\ll 2^r p^{\varepsilon} \Big( \frac{M^{\frac{1}{2}}}{N^{\frac{1}{2}}} p^{\frac{40}{27}}+M^{\frac{1}{2}} p^{\frac{79}{108}}+\frac{p^{\frac{187}{108}}}{N^{\frac{1}{2}}}+p^{\frac{53}{54}}+N^{\frac{1}{4}} p^{\frac{131}{216}}+N^{\frac{1}{2}} p^{\frac{25}{108}} \Big).
\end{equation} 
 Denoting the right side of \eqref{oldbd} as $2^r X$, one can choose 
 \begin{equation*}
 2^r=\frac{p^{\frac{3}{4}} \big(\frac{M}{N} \big)^{\frac{1}{4}} }{X^{\frac{1}{2}}}
\end{equation*}
to balance \eqref{trivv} and \eqref{oldbd}. Substituting this quantity back into \eqref{oldbd} and noting the error $O(p^{9/10+\varepsilon})$ inherited at the start of the argument yields Proposition \ref{critandupper}.
\end{proof}

\subsection{Proof of Theorem~\ref{secondmoment}}
\begin{proof}[Proof of Theorem~\ref{secondmoment} (assuming Proposition \ref{spectralest} and Theorem \ref{critrangethm})]
The main terms in Theorem \ref{secondmoment} were computed in Section \ref{mainterms}, incurring a cost (cf. \eqref{auxmain}) 
\begin{equation} \label{bd1}
\ll p^{\frac{3}{4}+\frac{\theta}{2}+\varepsilon}.
\end{equation}
Without loss of generality we have $0 \leq \theta \leq \frac{7}{64}$ by Kim--Sarnak \cite{Ki}. It suffices to bound $\mathcal{S}_{N,M,p,d}$ and $\widetilde{\mathcal{S}}_{N,M,p,d}$ for $d=1$ and $p$ and all $M, N$ satisfying \eqref{dyadiccond}. 

Applying Lemma \ref{plain1} we obtain
\begin{equation} \label{bd2}
\mathcal{S}_{N,M,p,1} \ll p^{\frac{1}{2}+\varepsilon} \quad \text{ for all } \quad M, N \quad \text{ satisfying } \quad \eqref{dyadiccond}.
\end{equation}
Applying Lemma \ref{tilde1} we have 
\begin{equation} \label{bd3}
\widetilde{\mathcal{S}}_{N,M,p,1} \ll p^{\frac{7}{8}+\frac{\theta}{4}+\varepsilon}+p^{\frac{15}{16}+\frac{\theta}{8}+\varepsilon} \ll p^{\frac{15}{16}+\frac{\theta}{8}+\varepsilon}
\text{ for all } \quad M,N \quad \text{ satisfying } \quad \eqref{dyadiccond},
\end{equation}
where the last inequality follows by the above fact about $\theta$.

Recalling \eqref{keyquant}, it suffices to consider $\mathcal{S}_{N,M,p,p}$. Applying Lemma \ref{secondregime} we have 
\begin{equation} \label{bd4}
\mathcal{S}_{N,M,p,p} \ll p^{\frac{1}{2}+\theta+\varepsilon} \quad \text{ for all } \quad M \leq N \leq 20M.
\end{equation}

We now can assume $N \geq 20M$.  Let $\alpha$ and $\beta$ be such that 
\begin{equation*}
M:=p^{\alpha}, \quad N:=p^{\beta} \quad \text{such that} \quad \alpha+\beta \leq 2, \quad \alpha \geq 0, \quad \text{and} \quad \beta \geq 0.
\end{equation*}
Let 
\begin{equation*}
\eta_0:=\frac{1}{600}.
\end{equation*}
We will now prove that 
\begin{equation} \label{thebd}
\mathcal{S}_{M,N,p,p} \ll p^{1-\eta_0}. 
\end{equation}
Lemma \ref{ramsub} guarantees that \eqref{thebd} holds when
\begin{equation} \label{Sp1}
 \beta<\frac{4}{3}-\frac{4}{3} \eta_0-\frac{2}{3} \alpha.
\end{equation}
Lemma \ref{valest} guarantees that \eqref{thebd} holds when
\begin{equation} \label{Sp2}
\quad \beta>1+\alpha+2 \eta_0.
\end{equation}
Proposition \ref{spectralest} guarantees that \eqref{thebd} holds when
\begin{equation} \label{Sp3}
 0 \leq \alpha \leq \frac{151}{300} \quad \text{and} \quad \beta<\frac{149}{150}+\alpha,
 \end{equation}
 or
 \begin{equation*}
 \alpha>\frac{151}{300}   \quad \text{and} \quad \beta<\frac{449}{300}.  
 \end{equation*}
 Lemma \ref{subby} guarantees that \eqref{thebd} holds when
 \begin{equation} \label{Sp4}
\alpha<\frac{1}{2}-2 \eta_0 \quad \text{and} \quad \beta>1+2 \eta_0.
\end{equation}
Plotting the inequalities \eqref{Sp1}--\eqref{Sp4} shows that only the solid trapezoid in the $\alpha \beta$--plane with vertices 
\begin{equation*}
\Big (\frac{299}{600},\frac{901}{600} \Big), \quad \Big (\frac{151}{300},\frac{449}{300} \Big ), \quad \Big (\frac{149}{300},\frac{149}{100} \Big) \quad \text{and} \quad \Big (\frac{149}{300},\frac{3}{2} \Big).
\end{equation*}
need only be considered now. Writing out the exponents of Proposition \ref{critandupper} (which was established under the assumption of Theorem \ref{critrangethm}) we have 
\begin{align*}
 \frac{\alpha}{2}-\frac{\beta}{2}+\frac{161}{108}, &\quad  \frac{\alpha}{2}-\frac{\beta}{4}+\frac{241}{216}, \quad \frac{\alpha}{4}-\frac{\beta}{2}+\frac{349}{216}, \\
  \frac{\alpha}{4}-\frac{\beta}{4}+\frac{67}{54},  & \quad \frac{\alpha}{4}-\frac{\beta}{8}+\frac{455}{432},  \quad \frac{\alpha}{4}+\frac{187}{216} \quad \text{and} \quad \frac{9}{10}.
\end{align*} 
A computation shows that each linear function evaluated at each of the four vertices is less than $1-\eta_0$. Thus \eqref{bd1}--\eqref{bd4} are subsumed by \eqref{thebd}. The rest of the paper is dedicated to proving Proposition \ref{spectralest} and Theorem \ref{critrangethm}, and that will complete the proof of Theorem \ref{secondmoment}.
\end{proof}

\section{Automorphic preliminaries II (integral weight)} \label{autprelim2}
\subsection{Maass forms}
We give a brief background on Maass forms relevant to our setting. One can see \cite[Section~4]{DFI1}, \cite[Section~2]{HM} and \cite[Sections~2 and 5]{BM1} for supplementary material.

Throughout $\kappa=0$ or $1$. For $\gamma \in \text{SL}_2(\mathbb{R})$, define the weight $\kappa$ slash operator for real analytic forms by
\begin{equation*}
g \vert_{\kappa} \gamma:=j(\gamma,\tau)^{-\kappa} g(\gamma \tau), \quad j(\gamma,\tau):=\frac{c \tau+d}{|c \tau+d|}=e^{i \text{arg}(c \tau+d)},
\end{equation*}
where the argument is always chosen in $(-\pi,\pi]$. The weight $\kappa$ Laplacian is defined by 
\begin{equation*}
\Delta_{\kappa}:=y^2  \Big(\frac{\partial^2}{\partial x^2}+\frac{\partial^2}{\partial y^2}  \Big)-i \kappa y \frac{\partial }{\partial x}.
\end{equation*}
A smooth function $g:\mathbb{H} \rightarrow \mathbb{C}$ is an eigenfunction of $\Delta_{\kappa}$ with eigenvalue $\lambda \in \mathbb{C}$ if
\begin{equation*}
(\Delta_{\kappa}+\lambda)g=0.
\end{equation*}
All eigenfunctions of $\Delta_{\kappa}$ are real analytic since it is an elliptic operator.

Let $\chi$ be a character modulo $D$ with $\chi(-1)=(-1)^{\kappa}$. A function $g:\mathbb{H} \rightarrow \mathbb{C}$ is automorphic of weight $\kappa$ and nebentypus $\chi$ for $\Gamma_0(D)$ if 
\begin{equation*}
g \vert_{\kappa} \gamma =\chi(d)  g \quad \text{for all} \quad \gamma:=\begin{pmatrix}
a & b \\
c & d
\end{pmatrix} \in \Gamma_0(D).
\end{equation*}
Let $\mathcal{A}_{\kappa}(D,\chi)$ denote the space of such functions. If $g \in \mathcal{A}_{\kappa}(D,\chi)$ is a smooth eigenfunction of $\Delta_{\kappa}$ that also satisfies the growth condition
\begin{equation*}
g(\tau) \ll y^{\sigma}+y^{1-\sigma} \quad \text{for all} \quad \tau:=x+iy \in \mathbb{H} \quad \text{and some} \quad \sigma>0, 
\end{equation*}
then it is called a Maass form. Let 
\begin{equation*}
\mathcal{L}_{\kappa}(D,\chi):=\{g \in \mathcal{A}_{\kappa}(D,\chi): \| g\| < \infty \},
\end{equation*}
where the norm is induced by the Petersson inner product 
\begin{equation*}
\langle g_1,g_2 \rangle:=\int_{\Gamma_0(D) \backslash \mathbb{H}} g_1(\tau) \overline{g_2(\tau)} d \mu(\tau), \quad d \mu:=\frac{dx dy}{y^2}.
\end{equation*} 
Let $\mathcal{R}_{\kappa}(D, \chi)$ denote the subspace of $\mathcal{L}_{\kappa}(D, \chi)$ consisting of smooth functions $g$ such that $g$ and $\Delta_{\kappa} g$ are bounded on $\mathbb{H}$.  One can show that $\mathcal{R}_{\kappa}(D, \chi)$ is dense in $\mathcal{L}_{\kappa}(D, \chi)$. For all $g_1,g_2 \in \mathcal{R}_{\kappa}(D,\chi)$ we have 
\begin{equation*}
\langle \Delta_{\kappa} g_1,g_2 \rangle=\langle g_1, \Delta_{\kappa} g_2 \rangle. 
\end{equation*}
Furthermore, for any $g \in \mathcal{R}_{\kappa}(D,\chi)$ we have 
\begin{equation*}
\langle g,-\Delta_{\kappa} g \rangle \geq \frac{|\kappa|}{2} \Big(1-\frac{|\kappa|}{2} \Big) \geq 0.
\end{equation*}
Thus by a theorem of Friedrichs, the operator $-\Delta_{\kappa}$  has a unique self-adjoint extension to $\mathcal{L}_{\kappa}(D, \chi)$ (which we also denote $-\Delta_{\kappa}$). Then by a theorem of von Neumann, the space $\mathcal{L}_{\kappa}(D,\chi)$ has a complete spectral resolution with respect $-\Delta_{\kappa}$. There is a continuous, discrete, and residual spectrum, worked out in detail by Maass and Selberg.

Let $\mathfrak{a}$ be a cusp of $\Gamma_0(D)$ and 
\begin{equation*}
{\Gamma_0(D)}_{\mathfrak{a}}:= \{ \gamma \in \Gamma_0(D) : \gamma \mathfrak{a}=\mathfrak{a} \} ,
\end{equation*}
denote its stability group. Let $\sigma_{\mathfrak{a}}$ denote the unique up to translation on the right matrix in $\text{SL}_2(\mathbb{R})$ satisfying $\sigma_{\mathfrak{a}} \infty =\mathfrak{a}$ and $\sigma_{\mathfrak{a}}^{-1} {\Gamma_0(D)}_{\mathfrak{a}} \sigma_{\mathfrak{a}}={\Gamma_0(D)}_{\infty}$. We say $\mathfrak{a}$ is \emph{singular} when 
\begin{equation*}
\chi \Big(\sigma_{\mathfrak{a}}^{} \begin{pmatrix}
1 & 1\\
0 &1 
\end{pmatrix} \sigma_{\mathfrak{a}}^{-1} \Big)=1.
\end{equation*}

For each singular cusp $\mathfrak{a}$ (and only at such cusps), the Eisenstein series is defined by 
\begin{equation*}
E_{\mathfrak{a}}(\tau,s,\chi)=\sum_{\gamma \in {\Gamma_0(D)}_{\mathfrak{a}} \backslash \Gamma_0(D)} \overline{\chi}(\gamma) j(\sigma_{\mathfrak{a}}^{-1} \gamma,\tau)^{-\kappa} (\Im \sigma_{\mathfrak{a}}^{-1} \gamma \tau)^s, \quad \Re s>1 \quad \text{and} \quad \tau \in \mathbb{H}.
\end{equation*}
One can check that each $E_\mathfrak{a}$ is independent of the choice of scaling matrix $\sigma_{\mathfrak{a}}$. Moreover if $\mathfrak{b}=\gamma \mathfrak{a}$ are $\Gamma_0(D)$-equivalent cusps, then (cf. \cite[(3.2)]{Y1})
\begin{equation} \label{welldef}
E_{\gamma \mathfrak{a}}(\tau,s,\chi):=\overline{\chi}(\gamma) E_{\mathfrak{a}}(\tau,s,\chi).
\end{equation}
Selberg \cite{Sel} proved that $E_{\mathfrak{a}}(\tau,s,\chi)$ has analytic continuation to the whole complex plane with only finitely many simples poles $s$ with $\frac{1}{2}<s \leq 1$. In particular, when $\chi$ is non-principal there are no poles in the region $\Re s \geq \frac{1}{2}$. When $\chi$ is principal, there is only one simple pole at $s=1$ in this region with constant (but automorphic) residue 
\begin{equation*}
\text{Res}_{s=1} E_{\mathfrak{a}}(\tau,s,\chi)=\frac{1}{\text{Vol}(\Gamma_0(D) \backslash \mathbb{H})}.
\end{equation*}
If $s$ is not a pole of $E_{\mathfrak{a}}(\tau,s,\chi)$, then $E_{\mathfrak{a}}(\tau,s,\chi)$ is a Maass form with eigenvalue $\lambda(s)$, but is not in $\mathcal{L}_{\kappa}(D,\chi)$.  
The continuous spectrum is composed of all the Eisenstein series on the critical line $s=\frac{1}{2}+it$.

The reminder of the spectrum is discrete and is spanned by \emph{Maass cusp forms}. It is countable and of finite multiplicity (with $\infty$ being the only limit point). We denote it by
\begin{equation*}
\lambda_1 \leq \lambda_2 \leq \dots
\end{equation*}
To summarise, every $g \in \mathcal{L}_{\kappa}(D,\chi)$ decomposes as 
\begin{equation} \label{gexpand}
g(\tau)=\sum_{j \geq 0} \langle g,u_j \rangle u_j(\tau)+\sum_{\mathfrak{a}} \frac{1}{4 \pi } \int_{\Re s=\frac{1}{2}} \Big \langle g, E_{\mathfrak{a}} \Big(\star,\frac{1}{2}+it,\chi \Big) \Big \rangle E_{\mathfrak{a}} \Big(\tau,\frac{1}{2}+it,\chi \Big) dt,
\end{equation}
where $u_0(\tau)$ is the constant function of Petersson norm $1$ (if $\kappa=0$), $\mathcal{C}_{\kappa}(D,\chi)=\{u_j \}_{j \geq 1}$ denotes an orthonormal basis of Maass cusp forms, and $\{\mathfrak{a}\}$ runs over all singular cusps of $\Gamma_0(D)$ relative to $\chi$. The convergence in \eqref{gexpand} is with respect to the norm topology.
If $g \in \mathcal{R}_{\kappa}(D, \chi)$, then \eqref{gexpand} converges pointwise absolutely and uniformly on compacta.

A Maass cusp form decays exponentially at the cusps and admits a Fourier expansion with the zeroth Fourier coefficient vanishing. At $\infty$, such an expansion is given by 
\begin{equation} \label{maassfourier}
g(\tau)=\sum_{\substack{n=-\infty}}^{\infty} \rho_g(n) W_{\frac{kn}{2|n|},it}(4 \pi |n| y) e(nx),
\end{equation}
where $W_{\alpha,\beta}(y)$ is the usual Whittaker function and $\lambda_g:=\frac{1}{4}+t_g^2$ is the Laplace eigenvalue of $g$. We call $t_g$ the spectral parameter of $g$. When $\kappa=0$, note that $t_g \in [-i \theta,i \theta] \cup [0,\infty)$, where $\theta=\frac{7}{64}$ is the best currently known \cite{Ki}.

The Eisenstein series has the expansion
\begin{equation} \label{eisexpand}
E_{\mathfrak{a}} \Big(\tau,\frac{1}{2}+it,\chi \Big)=\delta_{\mathfrak{a}=\infty} y^{\frac{1}{2}+it}+\phi_{\mathfrak{a}} \Big(\frac{1}{2}+it \Big) y^{\frac{1}{2}-it}+\sum_{\substack{n=-\infty \\ n \neq 0}}^{\infty} \rho_{\mathfrak{a}}(n,t) W_{\frac{kn}{2|n|},it}(4 \pi |n| y) e(nx),
\end{equation}
where $\phi_{\mathfrak{a}} \big( \frac{1}{2}+it \big)$ is the $(\mathfrak{a},\infty)$ entry of the relevant scattering matrix.

\subsection{Holomorphic forms}
Let $\chi$ be a Dirichlet character modulo $D$ with $\chi(-1)=(-1)^{\kappa}$ ($\kappa=0$ or $1$). For $\ell \in \mathbb{N}$ with $\ell \equiv {\kappa} \pmod{2}$, let $\mathcal{S}_{\ell}(D,\chi)$ denote the space of holomorphic cusp forms of level $D$, weight $\ell$ and nebentypus $\chi$. This space is equipped with the Petersson inner product
\begin{equation*}
\langle g_1,g_2 \rangle=\int_{\Gamma_0(D) \backslash \mathbb{H}} y^{\ell} g_1(\tau) \overline{g_2(\tau)} d \mu(\tau).
\end{equation*}
We also have the Fourier expansion (at $\infty$)
\begin{equation} \label{holofourier}
g(\tau)=\sum_{n \geq 1} \rho_g(n) (4 \pi n)^{\frac{\ell}{2}} e(n \tau).
\end{equation}

\subsection{Hecke operators and newform theory}
Recall that $\mathcal{L}_{\kappa}(D,\chi)$ (and the subspace generated by Maass cusp forms) is acted on by an algebra $\mathbf{T}$ generated by Hecke operators $\{T_n\}_{n \geq 1}$.  Each operator is defined by 
\begin{equation*}
(T_n g)(\tau)=\frac{1}{\sqrt{n}} \sum_{ad=n} \chi(a) \sum_{b \hspace{-0.15cm} \pmod{d}} g \Big(\frac{a \tau+b}{d} \Big)
\end{equation*}
These operators are commutative and multiplicative. They also satisfy the relation 
\begin{equation} \label{heckerel}
T_m T_n=\sum_{d \mid (m,n)} \chi(d) T_{\frac{mn}{d^2}}.
\end{equation}
Let $\mathbf{T}^{(D)}$ denote the subalgebra generated by $\{T_{n}\}_{(n,D)=1}$. We call a Maass cusp form which is an eigenform for $\mathbf{T}^{(D)}$ a \emph{Hecke--Maass} cusp form. The elements of $\mathbf{T}^{(D)}$ are normal with respect to the Petersson inner product, so the cuspidal subspce of $\mathcal{L}_{\kappa}(D,\chi)$ admits an orthonormal basis of Hecke--Maass cusp forms. For a Hecke--Maass cusp form $g$, the following relations hold 
\begin{equation} \label{rel1}
\sqrt{n} \rho_g( \pm n)=\rho_g(\pm 1) \lambda_g(n) \quad \text{for} \quad (n,D)=1,
\end{equation}
where $\lambda_g(n)$ denotes the eigenvalue of $T_n$, and 
\begin{align}
\sqrt{m} \rho_g(m) \lambda_g(n)&=\sum_{d \mid (m,n)} \chi(d) \rho_g \Big(\frac{mn}{d^2} \Big) \sqrt{\frac{mn}{d^2}}, \label{rel2} \\
\sqrt{mn} \rho_g(mn)&=\sum_{d \mid (m,n)} \chi(d) \mu(d) \rho_g \Big( \frac{m}{d} \Big) \sqrt{\frac{m}{d}} \lambda_g \Big( \frac{n}{d} \Big) \label{rel3}.
\end{align}
The space of newforms is defined to be the space spanned by Hecke--Maass cusp forms orthogonal to the subspace spanned by the oldforms. If $g$ is a Hecke form and in the new subspace, then $g$ is a Hecke eigenform of all Hecke operators by Atkin--Lehner theory and the above relations are satisfied for all $n$.

For a Hecke--Maass cusp form $g$, we have the pointwise bound 
\begin{equation*}
|\lambda_g(n)| \leq  n^{\theta+\varepsilon},
\end{equation*}
and the Rankin--Selberg bound
\begin{equation*}
\sum_{n \leq x} |\lambda_g(n)|^2 \ll_{\varepsilon} \big( D (1+|t|) x \big)^{\varepsilon} x.
\end{equation*}

If $g$ is an $L^2$--normalised newform of weight $\kappa \in \{0,1\}$ and level $D$, then by \cite[(30)]{HM} we have
\begin{equation} \label{coeff1}
( D(1+|t_g|) \big)^{-\varepsilon}  \Big( \frac{\cosh(\pi t_g)}{D (1+|t_g|)^{\kappa}} \Big)^{\frac{1}{2}}  \ll_{\varepsilon} |\rho_g(1)| \ll_{\varepsilon} \big( D(1+|t_g|) \big)^{\varepsilon} \Big( \frac{\cosh(\pi t_g)}{D (1+|t_g|)^{\kappa} } \Big)^{\frac{1}{2}}.
\end{equation} 
These upper bound (resp. lower bound) are a consequence of the seminal works of Hoffstein--Lockhart \cite{HL} (resp. Duke--Friedlander--Iwaniec \cite{DFI1}).

We now give a more explicit treatment of bases and newforms due to Blomer and Mili\'{c}evi\'{c} \cite[Section~5]{BM1} and \cite{ISL}. Let $\mathcal{B}_{\kappa}(D,t,\chi)$ (resp. $\mathcal{H}_{\ell}(D,\chi)$) denote an $L^2$--basis for $\mathcal{A}_{\kappa}(D,t,\chi)$ (resp. $\mathcal{S}_{\ell}(D,\chi)$). In general, both of these bases will include oldforms. We will focus on Maass forms since the holomorphic case will be the same, only requiring small notational changes. Suppose $\chi$ has conductor $D^{\star}_{\chi}$ and underlying primitive character $\chi^{\star}$. For $u \mid D$, let $\widetilde{\chi}$ modulo $u$ be the character induced by $\chi^{\star}$ and $\mathcal{B}_{\kappa}^{\star}(u,D,t, \widetilde{\chi}) \subseteq \mathcal{B}_{\kappa}(D,t,\chi)$ denote the set of all $L^2(\Gamma_0(D) \backslash \mathbb{H})$--normalised newforms of level $u$ and spectral parameter $t$. We write $g \vert_d (\tau):=g(d \tau)$. By Atkin--Lehner theory we have 
\begin{equation}\label{newform}
\mathcal{A}_{\kappa}(D, t,\chi) = \underset{\substack{\vspace{1mm} \\\ \substack{D^{\star}_{\chi} \mid u\\
u \mid D}}}{\text{\LARGE $\bigcirc\!\!\!\!\!\!\!\perp$}} \,\, \underset{\substack{\vspace{1mm}\\ g \in \mathcal{B}_{\kappa}^{\star}(u, D,t , \widetilde{\chi})}}{ \text{\LARGE $\bigcirc\!\!\!\!\!\!\!\perp$}}  \,\,  \bigoplus_{d \mid \frac{D}{u}}  g|_d \cdot \mathbb{C}.
 \end{equation}
The first two sums in \eqref{newform} are orthogonal, but the last is not orthogonal in general. Gram--Schmidt is required to make this sum orthogonal.

An orthogonal basis $\mathcal{B}_{\kappa}(D,\chi)$ for $\mathcal{A}_{\kappa}(D,\chi)$ is produced by collecting all spectral parameters
\begin{equation*}
\mathcal{B}_{\kappa}(D,\chi):=\coprod_t \mathcal{B}_{\kappa}(D,t,\chi).
\end{equation*}
Correspondingly,  
\begin{equation*}
\mathcal{B}_{\kappa}(u,D,\widetilde{\chi}):=\coprod_t \mathcal{B}^{\star}_{\kappa}(u,D,t,\widetilde{\chi}).
\end{equation*}

For a newform $g \in \mathcal{B}_{\kappa}(u,D,\widetilde{\chi})$, define the arithmetic functions 
\begin{equation*}
r_g(c)=\sum_{b \mid c} \frac{\mu(b) \lambda_g(b)^2}{b} \Big(\sum_{d \mid b} \frac{\chi(b)}{b}  \Big)^{-2} , \quad \alpha(c):=\sum_{b \mid c} \frac{\mu(b) \chi(b)^2}{b^2}, \quad \beta(c):=\sum_{b \mid c} \frac{\mu^2(b) \chi(b)}{b},
\end{equation*}
and 
\begin{equation*}
L(g,s)^{-1}=\sum_{c} \frac{\mu_g(c)}{c^s} \,
\end{equation*}
where 
\begin{equation*}
\mu_g(p)=-\lambda_g(p), \quad \mu_g(p^2)=\chi(p) \quad \text{and} \quad \mu_g(p^{\nu})=0 \quad \text{for} \quad \nu>2. 
\end{equation*}
For $d \mid e$, define 
\begin{equation*}
\xi^{\prime}_e(d):=\frac{\mu(e/d) \lambda_g(e/d)}{r_g(e)^{\frac{1}{2}} (e/d)^{\frac{1}{2}} \beta(e/d)}, \quad \xi^{\prime \prime}_e(d):=\frac{\mu_g(e/d)}{(e/d)^{\frac{1}{2}} (r_g(e) \alpha(e) )^{\frac{1}{2}} }
\end{equation*}
Now write $e=e_1 e_2$ uniquely with $e_1$ squarefree, $e_2$ squarefull and $(e_1,e_2)=1$. Then for $d \mid e$ define
\begin{equation*}
\xi_e(d):=\xi_{e_1}^{\prime} ( (e_1,d) ) \xi^{\prime \prime}_{e_2} ( (e_2,d) ) \ll e^{\varepsilon} (e/d)^{\theta-\frac{1}{2}}.
\end{equation*}

\begin{lemma} \emph{\cite[Lemma~9]{BM1}}
Let $u \mid D$ and $g^{\star} \in \mathcal{B}_{\kappa}(u,D, \widetilde{\chi}) \subset  \mathcal{B}_{\kappa}(D,\chi)$ be an $L^2(\Gamma_0(D) \backslash \mathbb{H})$ normalised newform of level $u$. Then the set of functions 
\begin{equation*}
\Big \{ g^{(e)}:= \sum_{d \mid e} \xi_{e}(d) g^{\star} \vert_d(\tau) : e \mid \frac{D}{u} \Big \}
\end{equation*}
is an orthonormal basis for the space $\bigoplus_{d \mid \frac{D}{u}} g^{\star} \vert_d \cdot \mathbb{C}$. If $g$ is any member of this basis, then its Fourier coefficients satisfy the bound
\begin{equation} \label{specialfbd}
\sqrt{n} \rho_g(n) \ll (nD)^{\varepsilon} n^{\theta} (D,n)^{\frac{1}{2}-\theta} | \rho_{g^{\star}}(1) |.
\end{equation}
\end{lemma}

Note that one can also see \cite[Theorem~3.2]{SPY} for the above lemma.

\begin{proof}
The proof is verbatim \cite[Lemma~9]{BM1} with the replacement of the principal character modulo $D$ with $\chi$.
\end{proof}

Elements in bases $g \in \mathcal{B}_{\kappa}(D,\chi)$ (or $g \in \mathcal{H}_{\ell}(D,\chi)$) have multiplicative-like properties. Given $r \in \mathbb{N}$, if $m=rm^{\prime} \in \mathbb{N}$ with $(m^{\prime},r)=1$, then by \cite[pg.~74]{BHM} we have
\begin{equation} \label{multbase1}
\sqrt{m} \rho_g(m)=\sum_{d \mid (D, (r/(r,D))} \mu(d) \chi(d) \lambda_{g^{\star}} \Big( \frac{r}{d (r,D)} \Big) \Big( \frac{(r,D) m^{\prime}}{d} \Big)^{\frac{1}{2}} \rho_g \Big( \frac{(r,D) m^{\prime}}{d} \Big),
\end{equation}
where $g^{\star}$ is the underlying newform. In particular, if $(r,D)=1$ we have 
\begin{equation} \label{multbase2}
\sqrt{m} \rho_g(m)=\lambda_{g^{\star}}(r) \sqrt{m^{\prime}} \rho_g(m^{\prime})
\end{equation}
Moreover if $g^{\star}$ satisfies the Ramanujan--Petersson conjecture (i.e. integral weight holomorphic eigenforms) and $a_m$ is any finite sequence of complex numbers then 
\begin{equation} \label{ramult}
\Big | \sum_m a_m \sqrt{m} \rho_g(m) \Big |^2 \leq \tau(r)^2 \sum_{d \mid (r,D)} \Big | \sum_{m^{\prime}} a_{rm^{\prime}} \sqrt{d m^{\prime}} \rho_g(d m^{\prime}) \Big |^2. 
\end{equation} 

\subsection{Kuznetsov--Proskurin formula and spectral inequalities}
 Let $\chi$ be a character modulo $D$ with $\chi(-1)=1$. Let $\phi: [0,\infty) \rightarrow \mathbb{C}$ have continuous derivatives up to third order and satisfy 
\begin{equation*}
\phi(0)=\phi^{\prime}(0)=0, \quad \phi^{j}(x) \ll (1+x)^{-3} \quad \text{for} \quad j=1,2,3.
\end{equation*}
Define the transforms 
\begin{align*}
\dot{\phi}(\ell)&:=4 i^{\ell}  \int_{0}^{\infty} \phi(x) J_{\ell-1}(x) \frac{dx}{x}, \\
\widetilde{\phi}(t)&:=2 \pi i \int_{0}^{\infty} \phi(x) \frac{J_{2it}(x)-J_{-2it}(x)}{\text{sinh}(\pi t)} \frac{dx}{x}, \\
\check{\phi}(t)&:=8 \int_{0}^{\infty} \phi(x) \cosh(\pi t) K_{2it}(x) \frac{dx}{x}.
\end{align*}
These transforms are normalised like those occurring in \cite[Section~5]{BM1} and \cite{Dra}. For $D \mid c$ and $m,n \in \mathbb{Z}$, define the Kloosterman sum (at $\infty \infty$) by 
\begin{equation*}
K(m,n,c,\chi):=\sideset{}{^*} \sum_{d \hspace{-0.15cm} \pmod{c}}  \overline{\chi}(d) e \Big( \frac{md+n \overline{d}}{c} \Big),
\end{equation*}
where the superscript $*$ denotes the condition that $(d,c)=1$.
\begin{lemma} \emph{\cite[Lemma~4.5]{Dra}} \emph{ and } \emph{\cite[Lemma~10]{BM1}}.  \label{kuz}
Let $\phi$ be as above and $\chi$ be a character modulo $D$ with $\chi(-1)=1$. For $\ell \geq 2$ and $\ell \equiv 0 \pmod{2}$, let  $\mathcal{B}_0(D,\chi)$ (resp. $\mathcal{H}_{\ell}(D,\chi)$) denote the orthonormal basis of Maass cusp forms (resp. holomorphic cusp forms) given above. Recalling the notations \eqref{maassfourier}--\eqref{holofourier}, for $m,n \in \mathbb{N}$ we have
\begin{multline*}
\sum_{D \mid c} \frac{1}{c} K (m,n,c, \chi) \phi \Big( \frac{4 \pi \sqrt{mn}}{c} \Big)=\sum_{\substack{\ell \geq 2 \\ \ell \equiv 0 \hspace{-0.15cm} \pmod{2}}} \sum_{g \in \mathcal{H}_{\ell}(D,\chi)} \Gamma(\ell) \dot{\phi}(\ell) \sqrt{mn} \hspace{0.1cm} \overline{\rho_g(m)} \rho_g(n) \\
+\sum_{g \in \mathcal{B}_{0}(D,\chi)} \widetilde{\phi}(t_g) \frac{\sqrt{mn}}{\cosh(\pi t_g)} \overline{\rho_g(m)} \rho_g(n)+\frac{1}{4 \pi} \sum_{\mathfrak{a} \emph{ sing.}} \sqrt{mn} \int_{-\infty}^{\infty} \frac{\widetilde{\phi}(t)}{\emph{cosh}(\pi t)} \overline{\rho_{\mathfrak{a}}(m,t)} \rho_{\mathfrak{a}}(n,t) dt 
\end{multline*}
and 
\begin{multline*}
\sum_{D \mid c} \frac{1}{c} K (m,-n,c, \chi) \phi \Big( \frac{4 \pi \sqrt{mn}}{c} \Big)=\sum_{g \in \mathcal{B}_{0}(D,\chi)} \check{\phi}(t_g) \frac{\sqrt{mn}}{\cosh(\pi t_g)} \overline{\rho_g(m)} \rho_g(-n)\\
+\frac{1}{4 \pi} \sum_{\mathfrak{a} \emph{ sing.}} \sqrt{mn} \int_{-\infty}^{\infty} \frac{\check{\phi}(t)}{\emph{cosh}(\pi t)} \overline{\rho_{\mathfrak{a}}(m,t)} \rho_{\mathfrak{a}}(-n,t) dt 
\end{multline*}
\end{lemma}

\subsection{Spectral large sieve and multiplicative sequences I}
In this section we record a spectral large sieve inequality for coefficients of Maass forms supported on sequences with multiplicative structure. This approach is originally due to Blomer and Mili\'{c}evi\'{c} \cite[Theorem~13]{BM1}, and is crucial in avoiding the Ramanujan--Petersson conjecture in their treatment of shifted convolution sums. We also record some standard spectral tools in enough generality for our purpose.

\begin{lemma} \label{weilsq}
Suppose $\chi$ is a character of modulus $D$ with $\chi(-1)=(-1)^{\kappa}$ and conductor $D^{\star}_{\chi}:=N$ or $4N$, where $N$ is odd and squarefree. Suppose $c \in \mathbb{N}$ such that $D \mid c$. Then for $a,b \in \mathbb{N}$ we have
\begin{equation} 
|K(a,b,c,\chi)| \leq  16 \tau(c) (a,b,c)^{\frac{1}{2}} c^{\frac{1}{2}}.
\end{equation}
\begin{proof}
This bound follows from \cite[Cor.~9.14, Prop.~9.4, Prop.~9.7 and Prop.~9.8]{KL}.
\end{proof}
\end{lemma}

We now present the standard spectral large sieve.
\begin{lemma} \label{largesieve}
Let $\chi$ be as in Lemma \ref{weilsq} with $\kappa=0$ and $\{a_m \}$ be a sequence of complex numbers. Suppose $T \geq 1$ and $M \geq \frac{1}{2}$. Then each of the three quantities 
\begin{equation*} 
\sum_{\substack{\kappa<\ell \leq T \\ \ell \equiv \kappa \hspace{-0.15cm} \pmod{2}}} \Gamma(\ell) \sum_{g \in \mathcal{H}_{\ell}(D,\chi)} \Big | \sum_{M<m \leq 2M} a_m \sqrt{m} \rho_g(m) \Big |^2, 
\end{equation*}
\begin{equation*}
\sum_{\substack{g \in \mathcal{B}_{0}(D,\chi) \\ |t_g| \leq T }} \frac{1}{\cosh(\pi t_g)} \Big | \sum_{M<m \leq 2M} a_m \sqrt{m} \rho_g(\pm m) \Big |^2,
\end{equation*}
\begin{equation*}
\sum_{\mathfrak{a} \emph{ sing.}} \int_{-T}^{T} \frac{1}{\cosh(\pi t)} \Big | \sum_{M<m \leq 2M} a_m \sqrt{m} \rho_{\mathfrak{a}}(\pm m,t) \Big |^2,
\end{equation*}
is bounded up to a constant depending on $\varepsilon$, by 
\begin{equation*}
\Big(T^2+\frac{M^{1+\varepsilon}}{D} \Big) \sum_{M<m \leq 2M} |a_m|^2 .  
\end{equation*}
\end{lemma}
\begin{proof}
Observe that this result has been proved in \cite[Proposition~4.7]{Dra}, except that the bound that appears there is 
\begin{equation*}
\Big(T^2+ (D^{\star}_{\chi})^{\frac{1}{2}} \frac{M^{1+\varepsilon}}{D} \Big) \sum_{M<m \leq 2M} |a_m|^2.
\end{equation*}
The appearance of the conductor in \cite{Dra} is due to the general estimate for Kloosterman sums in \cite[Theorem~9.2]{KL}. Since the conductor is either $D^{\star}_{\chi}=N$ or $4N$ with $N$ odd and squarefree, we can apply Lemma \ref{weilsq} to remove the factor of $(D^{\star}_{\chi})^{\frac{1}{2}}$ in \cite[(4.20)]{Dra}. The rest of the proof is verbatim the same in \cite[Proposition~4.7]{Dra}.
\end{proof}

\begin{lemma} \label{Mot}
Let $\chi$ be as in Lemma \ref{weilsq} with $\kappa=0$ and $m \in \mathbb{N}$. Then
\begin{equation*}
\sum_{\substack{|t_g| \leq T \\ g \in \mathcal{B}_0(D,\chi)}} \frac{1}{\emph{cosh}(\pi t_g)} | \sqrt{m} \rho_g(m)|^2  \ll_{\varepsilon} \Big( T^2 +\frac{(D,m)^{\frac{1}{2}} m^{\frac{1}{2}}}{D} \Big) (Tm)^{\varepsilon}.
\end{equation*}
\end{lemma}
\begin{proof}
One starts with the ``pre--Kuznetsov formula" asserted in \cite[Lemma~3]{PR}. The proof is then verbatim that of \cite[Lemma 2.4]{M1}, except that in \cite[(2.7.3),(2.7.10)]{M1}, the 
Kloosterman sum is $S(m,m,c,\chi)$, and an extra divisibility condition $D \mid c$ is added to the summation. One then appeals to  Lemma \ref{weilsq} and modifies the last two displays in the proof accordingly.
\end{proof}

Blomer and Mili\'{c}evi\'{c} prove the following result for Maass forms using a fourth moment approach. 
The main feature is that it allows one to avoid invoking the Ramanujan--Petersson conjecture (see Remark \ref{naiveapproach} below).
\begin{theorem} \emph{\cite[Theorem~13]{BM1}} \label{BM1theorem}
Let $\chi$ be as in Lemma \ref{weilsq} with $\kappa=0$. Let $r \in \mathbb{N}$, $M,T \geq 1$, and let $\{\alpha_{m^{\prime}} \}_{M  \leq  m^{\prime} \leq 2M}$ be any sequence of complex numbers with $|\alpha_{m^{\prime}}| \leq 1$. Then 
\begin{align} \label{BM1ineq}
& \sum_{\substack{|t_g| \leq T \\ g \in \mathcal{B}_0(D,\chi)}} \frac{1}{\emph{cosh}(\pi t_g)} \Big | \sum_{\substack{M \leq m^{\prime} \leq 2M \\ (m^{\prime},rD)=1 }} \alpha_{m^{\prime}} \sqrt{r m^{\prime}} \rho_g(r m^{\prime})  \Big |^2 \nonumber \\
& \ll_{\varepsilon} (D r M T)^{\varepsilon} (r,D) \Big(T+\frac{r^{1/2}}{D^{1/2}} \Big) \Big(T+\frac{M}{D^{1/2}} \Big) M.
\end{align}
\end{theorem}
\begin{proof}
This is verbatim the proof of \cite[Theorem~13]{BM1} using \eqref{heckerel} (for eigenvalues), \eqref{rel1}, \eqref{coeff1}--\eqref{multbase2}, and  Lemmas \ref{largesieve} and \ref{Mot} whenever their principal character analogues are used in the proof.
\end{proof}

\begin{remark}  \label{naiveapproach}
For the sake of argument, suppose that $(r,D)=1$.
One could naively apply the Hecke relation \eqref{multbase2} to the left side of \eqref{BM1ineq}, then estimate $\sqrt{r} \rho_g(r)$ by \eqref{specialfbd}, and finally apply the usual spectral large sieve in
Lemma \ref{largesieve} we would obtain 
\begin{align} \label{weakversion}
& \sum_{\substack{|t_g| \leq T \\ g \in \mathcal{B}_0(D,\chi)}} \frac{1}{\text{cosh}(\pi t_g)} \Big | \sum_{\substack{M \leq m^{\prime} \leq 2M \\ (m^{\prime},rD)=1 }} \alpha_{m^{\prime}} \sqrt{r m^{\prime}} \rho_g(r m^{\prime})  \Big |^2 \nonumber \\ 
& \ll_{\varepsilon}  (DrMT)^{\varepsilon} r^{2 \theta}  \Big( T^2+\frac{M}{D} \Big)  M,
\end{align}
which is insufficient for our purposes. For the analogous case when $g$ runs over holomorphic forms this naive approach is sufficient because Deligne's bound
for $\sqrt{r} \rho_g(r)$ is available.
\end{remark}

\subsection{Spectral large sieve and multiplicative sequences II}
We now prove a version of the spectral large sieve inequality for coefficients of Eisenstein series supported on a sequences with multiplicative structure. This will be a generalisation of \cite[(5.5)]{BM1} for a more general nebentypus. The proof and uses ideas from the explicit computations in \cite{Mic}, \cite{HM} and \cite[pp.~76--80]{BHM}. The machinery set out for Eisenstein series in \cite{KY1} and \cite{Y1} will be useful throughout the proof.

We setup the notation and preliminary results required for the proof of Lemma \ref{speceisen}. Let $D \in \mathbb{N}$. A full set of inequivalent cusps of $\Gamma_0(D)$ is given by 
\begin{equation} \label{cusps}
\Big \{ \mathfrak{a}:=\frac{1}{w}=\frac{1}{uf} : f \mid D, u \in \mathcal{U}_f  \Big \},
\end{equation}
where for each $f \mid D$, $\mathcal{U}_f$ is a set of integers coprime to $f$ representing each reduced residue class modulo $\widetilde{f}:=(f,D/f)$ exactly once. Moreover, one may always further choose representative $u \pmod{\widetilde{f}}$ such that $(u,D)=1$ by adding a suitable multiple of $\widetilde{f}$ (cf. \cite[Corollary~3.2]{KY1}). Note that with these choices of representatives $u$ we have $\frac{u}{f} \sim_{\Gamma} \frac{1}{uf}$.

Set
\begin{equation} \label{notation}
D^{\prime}=\frac{D}{f}, \quad D^{\prime \prime}=\frac{D^{\prime}}{(f,D^{\prime})} \quad \text{and} \quad w^{\prime}=u.
\end{equation}
The stabiliser group of an arbitrary cusp $\mathfrak{a}=\frac{1}{w}$ \cite[Proposition~3.3]{KY1} is given by
\begin{equation*}
\Gamma_{\mathfrak{a}}=\big \{ \pm \tau^{t}_{\mathfrak{a}}: t \in \mathbb{Z}  \big \}, \quad \text{where} \quad  \tau^{t}_{\mathfrak{a}}=\begin{pmatrix}
1-w D^{\prime \prime} t & D^{\prime \prime} t \\
-w^2 D^{\prime \prime} t &1+w D^{\prime \prime} t
\end{pmatrix},
\end{equation*}
and one can take the choice of scaling matrix 
\begin{equation} \label{scalematrix}
\sigma_{\mathfrak{a}}=\begin{pmatrix}
\sqrt{D^{\prime \prime}} & 0 \\
w \sqrt{D^{\prime \prime}} & 1/\sqrt{D^{\prime \prime}}
\end{pmatrix}.
\end{equation}
Observe that \cite[Lemma~3.5]{KY1} (expanding around the cusp at $1/D \sim \infty$) asserts
\begin{equation} \label{scaled}
\sigma_{\mathfrak{a}}^{-1} \Gamma \sigma_{1/D}:= \Big \{ \begin{pmatrix} 
\frac{a}{\sqrt{D^{\prime \prime}}} & \frac{b}{\sqrt{D^{\prime \prime}}}\\
c \sqrt{D^{\prime \prime}} & d \sqrt{D^{\prime \prime}} \end{pmatrix}: \begin{pmatrix}
a & b \\
c & d
\end{pmatrix} \in \text{SL}_2(\mathbb{Z}) \quad \text{and} \quad c \equiv -wa \pmod{D} \Big \}.
\end{equation}
In particular, \cite[(6.4) and (6.5)]{Y1} asserts
\begin{multline} \label{coset}
\Gamma_{\infty} \backslash \sigma_{\mathfrak{a}}^{-1} \Gamma:=\delta_{f=D} \Gamma_{\infty} \cup \Big \{  \begin{pmatrix}
* & * \\
c \sqrt{D^{\prime \prime}} & d \sqrt{D^{\prime \prime}} 
\end{pmatrix}:  \\
c>0, \quad (c,d)=1, \quad c=f \gamma, \quad (\gamma, D^{\prime})=1, \quad \text{and} \quad d \equiv -u \overline{\gamma} \mod(f,D^{\prime}) \Big \},
\end{multline}
where the union above is disjoint.

We recall a useful lemma giving necessary and sufficient conditions for a cusp to be singular.
\begin{lemma} \emph{\cite[Lemma~5.4]{Y1}} \label{youngsingular}
Let $\chi$ be a Dirichlet character modulo $D$ and let $\mathfrak{a}=\frac{1}{uf}$ with $u \mid D$ and $(u,D)=1$. Then $\mathfrak{a}$ is singular relative to $\chi$ if and only if $\chi$ is periodic modulo $D/(f,D/f)=[f,D/f]$, equivalently, the primitive character inducing $\chi$ has modulus dividing $D/(f,D/f)$.
\end{lemma}

We recall a decomposition for $\chi$ given on \cite[pg.~17]{Y1}. There exist integers $f_0$ and $D_0^{\prime}$ such that 
\begin{equation*}
f_0 \mid f  \quad \text{and} \quad D^{\prime}_0 \mid D^{\prime} \quad \text{with} \quad [f,D^{\prime}]:=f_0 D^{\prime}_0 \quad \text{and} \quad (f_0,D^{\prime}_0)=1,
\end{equation*}
such that 
\begin{equation} \label{charfac}
\chi:=\chi^{(D^{\prime}_0)} \chi^{(f_0)},
\end{equation}
where $\chi^{(D^{\prime}_0)}$ and $\chi^{(f_0)}$ are characters to moduli $D^{\prime}_0$ and $f_0$ respectively. The choices for $f_0$ and $D^{\prime}_0$ may not be unique. This decomposition will be useful in the proof of Lemma \ref{speceisen}, but not feature in the final statement. 

Given $\mathfrak{a}=\frac{1}{w}$ and $\mu \in  \sigma_{\mathfrak{a}}^{-1} \Gamma$ (written as in \eqref{scaled}), then the argument on \cite[pg.~19]{Y1} proves that $\overline{\chi}(\sigma_{\mathfrak{a}} \mu)$ depends only on the coset $\Gamma_{\infty} \mu$. Thus $\overline{\chi} (\sigma_{\mathfrak{a}} \mu)$ depends only on the data contained in \eqref{coset}. In particular, \cite[(6.6)]{Y1} asserts that
\begin{equation} \label{chieval}
\overline{\chi} (\sigma_{\mathfrak{a}} \mu)=\chi(a)=\chi^{(D^{\prime}_0)}(-\overline{u} \gamma) \chi^{(f_0)}(\overline{d}).
\end{equation}

We recall the definition of Gauss sum given by \eqref{gaussdef}. Let $\Psi$ be a Dirichlet character modulo $c$ and $c \mid s$. Then for $n \in \mathbb{N}$ we have the Gauss sum
\begin{equation} \label{gauss}
\mathcal{G}_{\Psi}(n;s):=\sideset{}{^*} \sum_{d \hspace{-0.2cm} \pmod{s}} \Psi(d) e \Big(\frac{nd}{s} \Big),
\end{equation}
where the superscript $*$ denotes the condition that $(d,s)=1$. Without loss of generality one can replace $\Psi$ with its underlying primitive character $\Psi^{\star}$ in \eqref{gauss}.

Given a cusp $\mathfrak{a}=\frac{1}{uf}$, $m \in \mathbb{N}$, $t \in \mathbb{R}$, and characters $\Psi_1$ and $\Psi_2$ whose moduli divide $D/f$ and $f$ respectively, 
consider the series
\begin{equation} \label{dirichlet}
\mathcal{S}(t,m;\Psi_1,\Psi_2,f):=\sum_{\substack{\gamma>0 \\ (\gamma,\frac{D}{f})=1}}  \frac{\Psi_1(\gamma)}{\gamma^{1+2it}} \mathcal{G}_{\Psi_2}(m;\gamma f),
\end{equation}
Note that \eqref{dirichlet} also satisfies 
\begin{equation} \label{prim}
\mathcal{S}(t,m;\Psi_1,\Psi_2,f)=\mathcal{S}(t,m;\Psi^{\star}_1,\Psi^{\star}_2,f).
\end{equation}
For technical convenience we restrict our attention to even characters in this next result.

\begin{lemma} \label{speceisen}
Suppose $D \in \mathbb{N}$ and $\chi$ is an even character modulo $D$ such that all cusps $\mathfrak{a}$ of $\Gamma_0(D)$ are singular with respect to $\chi$ and let $\rho_{\mathfrak{a}}(m,t)$ be as in \eqref{eisexpand}. For a given $r \in \mathbb{N}$, let $\{a_m\}$ be a finite sequence of complex numbers supported only on integers $m=r m^{\prime}$ with $(r,m^{\prime})=1$. Then we have 
\begin{equation} \label{eisenineq}
\sum_{\mathfrak{a}} \Big | \sum_m a_m \sqrt{|m|} \rho_{\mathfrak{a}}(m,t) \Big |^2 \leq 16 \tau(D)^4 \tau(r)^4 \sum_{d \mid (r,D)}  \sum_{\mathfrak{a}} \Big | \sum_{\substack{m^{\prime} \\ (m^{\prime},r)=1}} a_{r m^{\prime}} \sqrt{d m^{\prime}} \rho_{\mathfrak{a}}(d m^{\prime},t) \Big |^2,
\end{equation}
\end{lemma}
\begin{proof}
Without loss of generality we will work with a complete set of inequivalent cusps given in \eqref{cusps} where each $(u,D)=1$ (cf. \eqref{welldef}). Combining \cite[(12)--(14)]{PR},  \eqref{notation}--\eqref{dirichlet} and \cite[(13.14.31)]{DLMF} we obtain
\begin{align} 
\sqrt{|m|} \rho_{\mathfrak{a}}(m,t)&=\frac{|m|^{it} \pi^{\frac{1}{2}+it}}{\Gamma(\frac{1}{2}+it)} \Big( \frac{(f,D^{\prime})}{Df} \Big)^{\frac{1}{2}+it} \sum_{\substack{\gamma>0 \\ (\gamma,\frac{D}{f})=1}}  \frac{\chi^{(D^{\prime}_0)}(-\overline{u} \gamma)}{\gamma^{1+2it}} \nonumber \\
& \times \sum_{\substack{0 \leq d<\gamma f \\ (d,\gamma f)=1 \\ d \gamma \equiv -u \hspace{-0.2cm}  \pmod{\widetilde{f}}}} \overline{\chi^{(f_0)}}(d)  e \Big( \frac{m d }{\gamma f} \Big),
\label{initrho}
\end{align}
where $\widetilde{f}:=(f,D/f)$.
Detecting the congruence in \eqref{initrho} with multiplicative characters we obtain
\begin{equation} \label{exactexp}
\sqrt{|m|} \rho_{\mathfrak{a}}(m,t)=\frac{|m|^{it} \pi^{\frac{1}{2}+it}}{\Gamma(\frac{1}{2}+it)} \Big( \frac{(f,D^{\prime})}{Df} \Big)^{\frac{1}{2}+it} \frac{1}{\phi({\widetilde{f})}} \sum_{\psi \hspace{-0.15cm} \pmod{\widetilde{f}}}  \overline{\chi^{(D^{\prime}_0)} \psi}(-u) \mathcal{S}(t,m; \psi \chi^{(D^{\prime}_0)}, \psi  \overline{\chi^{(f_0)}},f),
\end{equation}
All cusps are singular by hypothesis. Summing \eqref{exactexp} over all $m$, squaring, and then summing over all of the cusps we obtain
\begin{multline*} 
\sum_{\mathfrak{a}} \Big | \sum_m a_m \sqrt{|m|} \rho_{\mathfrak{a}}(m,t) \Big |^2 =\frac{\pi}{|\Gamma(\frac{1}{2}+it)|^2} \sum_{f \mid D} \frac{\widetilde{f}}{Df \phi(\widetilde{f})^2} \\
\times \sum_{u \in \mathcal{U}_f} |\overline{\chi^{(D^{\prime}_0)}}(-u) |^2 \Big | \sum_{\psi \hspace{-0.2cm} \pmod{\widetilde{f}}} \overline{\psi}(-u) \sum_{m} a_m |m|^{it} \mathcal{S}(t,m;\psi \chi^{(D^{\prime}_0)}, \psi  \overline{\chi^{(f_0)}},f ) \Big |^2,
\end{multline*}
where $\mathcal{U}_f$ represents a complete system of residues coprime to $\widetilde{f}$ and also coprime to $D$. We may remove the $|\overline{\chi^{(D^{\prime}_0)}}(-u)|^2$
factor since $ |\overline{\chi^{(D^{\prime}_0)}}(-u) |=1$. Applying Parseval's identity above yields
\begin{align}
\sum_{\mathfrak{a}} \Big | \sum_m a_m \sqrt{|m|} \rho_{\mathfrak{a}}(m,t) \Big |^2&= \frac{\pi}{|\Gamma(\frac{1}{2}+it)|^2} \sum_{f \mid D} \frac{\widetilde{f}}{Df \phi(\widetilde{f})} 
\nonumber \\
& \times \sum_{\psi \hspace{-0.2cm} \pmod{\widetilde{f}}} \Big | \sum_m a_m |m|^{it} \mathcal{S}(t,m; \psi \chi^{(D^{\prime}_0)}, \psi \overline{\chi^{(f_0)}}, f) \Big |^2. \label{secondineq}
\end{align}

Next we set 
\begin{equation*}
 (\psi \chi^{(D^{\prime}_0)})^*=:\chi_1 \quad \text{and} \quad (\psi \overline{\chi^{(f_0)}})^*=:\chi_2,
\end{equation*}
where $\chi_i$ is primitive of modulus $q_i$. A necessary condition on $\chi_1$ and $\chi_2$ is that $\chi_1 \overline{\chi_2} \sim \chi$, where $\sim$ means that both sides are induced by the same primitive character.  Recalling \eqref{prim} and moving the sum on $\psi$ to the inside of \eqref{secondineq} gives
\begin{align} \label{developed}
& \sum_{\mathfrak{a}} \Big | \sum_m a_m \sqrt{|m|} \rho_{\mathfrak{a}}(m,t) \Big |^2 \nonumber \\
&=\frac{\pi}{|\Gamma(\frac{1}{2}+it)|^2} \sum_{f \mid D} \frac{\widetilde{f}}{Df \phi(\widetilde{f})}  \sum_{q_1 \mid \frac{D}{f}} \sum_{q_2 \mid f} 
\sideset{}{'}\sum_{\substack{\chi_1 \hspace{-0.2cm} \pmod{q_1} \\ \chi_2 \hspace{-0.2cm} \pmod{q_2} \\ \chi_1 \overline{\chi_2} \sim \chi}} \Big | \sum_m a_m |m|^{it} \mathcal{S}(t,m; \chi_1,\chi_2,f ) \Big |^2 \nonumber \\
& \times \sum_{\psi \hspace{-0.2cm} \pmod{\widetilde{f}}} \delta(\chi, \chi_1,\chi_2,\psi),
\end{align} 
where the $'$ denotes summation over primitive characters only, and 
\begin{equation*}
\delta (\chi, \chi_1,\chi_2,\psi):=\begin{cases}
1 &\text{if} \quad (\psi \chi^{(D^{\prime}_0)})^*=\chi_1 \quad \text{and} \quad (\psi \overline{\chi^{(f_0)}})^*=\chi_2 \\
0 & \text{otherwise}.
\end{cases}
\end{equation*}
The argument on \cite[pp.~21--22]{Y1} proves that 
\begin{equation*}
\sum_{\psi \hspace{-0.2cm} \pmod{\widetilde{f}}} \delta (\chi, \chi_1,\chi_2,\psi)=1,
\end{equation*}
under the conditions in the second, third and fourth summations in \eqref{developed}. Thus 
\begin{align} 
\sum_{\mathfrak{a}} \Big | \sum_m a_m \sqrt{|m|} \rho_{\mathfrak{a}}(m,t) \Big |^2&=\frac{\pi}{|\Gamma(\frac{1}{2}+it)|^2} \sum_{f \mid D} \frac{\tilde{f}}{Df \phi(\tilde{f})}  \sum_{q_1 \mid \frac{D}{f}} \sum_{q_2 \mid f} \nonumber \\ 
& \times \sideset{}{'} \sum_{\substack{\chi_1 \hspace{-0.2cm} \pmod{q_1} \\ \chi_2 \hspace{-0.2cm} \pmod{q_2} \\ \chi_1 \overline{\chi_2} \sim \chi}}  \Big | \sum_m a_m |m|^{it} \mathcal{S} (t,m; \chi_1,\chi_2,f) \Big |^2.   \label{goodform}
\end{align}

For a given $r \in \mathbb{N}$, recall that $\{a_m\}$ is supported on integers $m=r m^{\prime}$ such that $(r,m^{\prime})=1$. Our goal is now to write $\mathcal{S} (t,m; \chi_1,\chi_2;f)$ in terms of $\mathcal{S}(dm^{\prime}; \chi_1,\chi_2; \bullet)$ for $d \mid (r,D)$. For a given $f$ and $\gamma$ we write 
\begin{equation} \label{factorisation}
f=q_2 f^{\prime} f^{\prime \prime}, \quad f^{\prime} \mid q_2^{\infty} \quad \text{and} \quad (f^{\prime \prime}, q_2)=1,
\end{equation}
and 
\begin{equation}
\gamma=\gamma^{\prime} \gamma^{\prime \prime}, \quad \gamma^{\prime} \mid q_2^{\infty}, \quad (\gamma^{\prime \prime},q_2)=1.
\end{equation}
The Chinese remainder theorem, orthogonality, \cite[pg.~165 and Theorem~8.19]{A}, and the fact that $\chi_2$ is primitive is used in the following computation
\begin{align}
\mathcal{G}_{\chi_2}(m;\gamma f)&=\mathcal{G}_{\chi_2} (m;\gamma^{\prime} \gamma^{\prime \prime}  q_2 f^{\prime} f^{\prime \prime} ) \nonumber \\
&=\chi_2 (\gamma^{\prime \prime} f^{\prime \prime}) \mathcal{G}_{\chi_2} (m; q_2 f^{\prime} \gamma^{\prime} ) r (m;\gamma^{\prime \prime} f^{\prime \prime} ) \nonumber \\
&=\delta_{f^{\prime} \gamma^{\prime} \mid m} f^{\prime} \gamma^{\prime} \chi_2(\gamma^{\prime \prime} f^{\prime \prime}) \mathcal{G}_{\chi_2} \Big( \frac{m}{f^{\prime} \gamma^{\prime}}; q_2  \Big)  r(m;\gamma^{\prime \prime} f^{\prime \prime}),  \nonumber \\
&=\delta_{f^{\prime} \gamma^{\prime} \mid m} f^{\prime} \gamma^{\prime} \chi_2(\gamma^{\prime \prime} f^{\prime \prime}) \overline{\chi_2} \Big(\frac{m}{f^{\prime} \gamma^{\prime}} \Big) \mathcal{G}_{\chi_2} (1; q_2)  r(m;\gamma^{\prime \prime} f^{\prime \prime})  \label{simplified}
\end{align}
where $r(m;c):=G_{\mathbf{1}_c}(m,c)$ (Ramanujan sum) and $\delta_{f^{\prime} \gamma^{\prime} \mid m}=1$ is the indicator function for $f^{\prime} \gamma^{\prime} \mid m$. Let 
\begin{equation*}
r=r^{\prime}_{q_2} r^{\prime (q_2)} \quad \text{and} \quad \quad m^{\prime}=m^{\prime}_{q_2} m^{\prime (q_2)}
\end{equation*}
be such that 
\begin{equation*} 
r^{\prime}_{q_2}, m^{\prime}_{q_2} \mid q_2^{\infty} \quad \text{and} \quad (r^{\prime (q_2)}  m^{\prime (q_2)}, q_2 )=1.
\end{equation*}
In this notation we see that \eqref{simplified} is $0$ unless 
\begin{equation} \label{charcond}
f^{\prime} \gamma^{\prime}=r^{\prime}_{q_2} m^{\prime}_{q_2}.
\end{equation}
Recalling \eqref{dirichlet} and combining \eqref{factorisation}--\eqref{charcond} we obtain (after relabelling $\gamma^{\prime \prime}$ by $\gamma$),
\begin{align}
\mathcal{S}(t,r m^{\prime};\chi_1,\chi_2,f)&=\delta_{f^{\prime} \mid m} f^{\prime} \delta_{( r^{\prime}_{q_2} m^{\prime}_{q_2}/f^{\prime} ,D/f)=1}  \Big( \frac{f^{\prime}}{r^{\prime}_{q_2} m^{\prime}_{q_2}}  \Big)^{2it}   \chi_1 \Big( \frac{r^{\prime}_{q_2} m^{\prime}_{q_2}}{f^{\prime}} \Big)  \chi_2(f^{\prime \prime}) \nonumber \\ 
&\times \overline{\chi_2} (r^{\prime (q_2)} m^{\prime (q_2)} )  \mathcal{G}_{\chi_2} (1; q_2)  
 \sum_{\substack{\gamma>0 \\ (\gamma,D/f)=1 \\ (\gamma,q_2)=1  }}  \frac{\chi_1 \chi_2(\gamma)}{\gamma^{1+2it}}  r(m; \gamma f^{\prime \prime}).  \label{preexact}
\end{align}

A computation using \eqref{preexact} and \cite[Theorem~8.6]{A} gives
\begin{align} 
\mathcal{S}(t,m,\chi_1,\chi_2;f)&=\frac{\delta_{f^{\prime} \mid m}  f^{\prime} \delta_{( r^{\prime}_{q_2} m^{\prime}_{q_2}/f^{\prime} ,D/f)=1}\chi_2(f^{\prime \prime}) \overline{\chi_2} ( r^{\prime (q_2)} m^{\prime (q_2)} )  \mathcal{G}_{\chi_2}(1,q_2)}{L^{(D)}(\chi_1 \chi_2,1+2it)} \nonumber \\ 
& \times \mathcal{R}(t,m;\chi_1 \chi_2,f^{\prime \prime}) \eta_{\chi_1 \chi_2}(m) 
\Big( \frac{f^{\prime}}{r^{\prime}_{q_2} m^{\prime}_{q_2}}  \Big)^{2it}   \chi_1 \Big( \frac{r^{\prime}_{q_2} m^{\prime}_{q_2}}{f^{\prime}} \Big), \label{refine1}
\end{align}
where the superscript $(D)$ denotes that the local factors at the primes dividing $D$ have been removed and
\begin{equation} \label{Rterm}
\mathcal{R}(t,m ; \chi_1 \chi_2,f^{\prime \prime}):=\sum_{\substack{\gamma \mid D^{\infty} \\ (\gamma,D/f)=1 \\  (\gamma,q_2)=1 }} \frac{\chi_1 \chi_2(\gamma) }{\gamma^{1+2it}} r(m;\gamma f^{\prime \prime} )
\end{equation}
and 
\begin{equation*}
\eta_{\chi_1 \chi_2}(m):=\sum_{\substack{a \mid m \\ (a,D)=1} } \frac{\chi_1 \chi_2(a)}{a^{2it}}.
\end{equation*}

We now consider \eqref{Rterm}. Since $\gamma \mid D^{\infty}$ and $(\gamma,D/f)=(\gamma,q_2)=1$, we must have $\gamma \mid (f^{\prime \prime})^{\infty}$. We apply \cite[Theorem~8.7]{A} and write \eqref{Rterm} as
\begin{equation} \label{Reuler}
\mathcal{R}(t,m ; \chi_1 \chi_2 , f^{\prime \prime}):=\prod_{\substack{p^{\alpha} \mid \mid f^{\prime \prime} \\ p^{\alpha} \mid \mid D }} \sum_{\beta \geq 0} \frac{\chi_1 \chi_2 (p^{\beta}) }{p^{\beta(1+2it)}} r( p^{v_p(m)};p^{\alpha+\beta}). 
\end{equation}
We refine the factorisation in \eqref{factorisation} to 
\begin{equation*}
f^{\prime}=f_r^{\prime} f^{\prime (r)}, \quad f^{\prime \prime}=f^{\prime \prime}_{r} f^{\prime \prime (r)}, \quad \text{where} \quad f_r^{\prime}, f_r^{\prime \prime} \mid r^{\infty} \quad \text{and} \quad ( f^{\prime (r)} f^{\prime \prime (r)},r )=1.
\end{equation*}
Since $(r,m^{\prime})=1$ and $f^{\prime} \mid r m^{\prime}$, it follows that $f^{\prime}_{r}=(f^{\prime},r)$. Recalling \eqref{refine1}, we obtain 
\begin{align} 
& \mathcal{S}(t,r m^{\prime},\chi_1,\chi_2;f) \nonumber \\
&= \Big( \delta_{f^{\prime}_r \mid r}  f^{\prime}_r  \delta_{(r^{\prime}_{q_2}/ f^{\prime}_{r}, D/f)=1}  \overline{\chi_2} ( r^{\prime(q_2)} ) \chi_2 (f^{\prime \prime}_r)  \mathcal{R}(t,r; \chi_1 \chi_2 , f^{\prime \prime}_r ) \eta_{\chi_1 \chi_2}(r) \Big(\frac{f^{\prime}_r}{r^{\prime}_{q_2} }  \Big)^{2it} \chi_1 \Big(\frac{r^{\prime}_{q_2}}{f^{\prime}_r} \Big)  \Big ) \nonumber \\  
& \times \Big ( \delta_{f^{\prime (r)} \mid m^{\prime}} f^{\prime (r)} \delta_{(m^{\prime}_{q_2}/ f^{\prime(r)}, D/f)=1}   \overline{\chi_2} ( m^{\prime (q_2)} ) \chi_2( f^{\prime \prime (r)}  ) \mathcal{R}(t,m^{\prime}; \chi_1 \chi_2, f^{\prime \prime (r)}) \nonumber  \\
& \times \eta_{\chi_1 \chi_2}(m^{\prime}) \Big(\frac{f^{\prime (r)}}{m^{\prime}_{q_2}} \Big)^{2it} \chi_1 \Big( \frac{m^{\prime}_{q_2}}{f^{\prime(r)}}  \Big) \Big)  
\frac{\mathcal{G}_{\chi_2}(1,q_2)}{L^{(D)}(\chi_1 \chi_2,1+2it)}. \label{refine2}
\end{align}

We now define
\begin{equation*}
y_r:= \Big( \frac{f^{\prime \prime}}{(f^{\prime \prime},2)},r \Big) \quad \text{and} \quad \widetilde{y}_r:=\prod_{\substack{p^{\alpha} \mid \mid f^{\prime \prime}_r  \\ \alpha \leq v_p(r)+1 }} p^{\alpha}.
\end{equation*}
An explicit computation using \eqref{Reuler} and \cite[Theorem~8.6 and 8.7]{A} shows that  $R(y_r;\chi_1 \chi_2,\widetilde{y}_r) \neq 0$. In particular,
\begin{equation} \label{localcalc}
\Big | \frac{\mathcal{R}(t,r;\chi_1 \chi_2,f^{\prime \prime}_{r})}{\mathcal{R}(t,y_r;\chi_1 \chi_2,\widetilde{y}_r)}  \Big | \leq \prod_{p \mid (r,D)} (v_p(r)+1) \frac{\big( 1+\frac{1}{p} \big)}{\kappa(p)} \leq 4 \tau(D) \tau(r),
\end{equation}
where 
\begin{equation*}
\kappa(p):=\begin{cases}
\frac{1}{2} & \quad \text{if} \quad p=2 \\
1-\frac{2}{p}& \quad \text{if} \quad p>2.
\end{cases}
\end{equation*}
Observe that $(f^{\prime}_r, y_r)=1$ and $f^{\prime}_r, y_r \mid r$ and $f^{\prime}_r, y_r \mid D$. Thus $f^{\prime}_r y_r \mid (r,D)$. Furthermore,
$\widetilde{y}_r \mid f^{\prime \prime}_r$, $(f^{\prime}_r y_r,m^{\prime})=1$ and $\eta_{\chi_1 \chi_2}(f^{\prime}_r y_r)=1$.  Using \eqref{refine2} and \eqref{Reuler} we obtain 
\begin{align} 
\mathcal{S}(t,r m^{\prime}; \chi_1, \chi_2, f)&=\delta_{f^{\prime}_r \mid r} \delta_{(r^{\prime}_{q_2}/f^{\prime}_r,D/f)=1} \overline{\chi_2} \Big(\frac{r^{\prime (q_2)}}{y_r} \Big) \chi_2 \Big( \frac{f^{\prime \prime}_r }{\widetilde{y}_r} \Big) \frac{\mathcal{R}(t,r;\chi_1 \chi_2,f^{\prime \prime}_r) }{ \mathcal{R}(t,f^{\prime}_r y_r; \chi_1 \chi_2, \widetilde{y}_r )} \nonumber  \\
& \times \eta_{\chi_1 \chi_2}(r) \Big( \frac{f^{\prime}_r}{r^{\prime}_{q_2}} \Big)^{2it} \chi_1 \Big( \frac{r^{\prime}_{q_2}}{f^{\prime}_r} \Big) \mathcal{S}(t,f^{\prime}_r y_r m^{\prime}; \chi_1, \chi_2, q_2 f^{\prime} \widetilde{y}_r f^{\prime \prime (r)}). \label{reducedS}
\end{align}
Combining \eqref{localcalc}, \eqref{reducedS} and $| \eta_{\chi_1 \chi_2}(r) | \leq \tau(r)$ we obtain
\begin{multline} \label{keyineq}
 \Big | \sum_m a_m |m|^{it} \mathcal{S}(t,m; \chi_1,\chi_2, f) \Big |^2 \\
 \leq  16 \tau(D)^2 \tau(r)^4 \sum_{d \mid (r,Q)} \Big | \sum_{\substack{m^{\prime} \\ (m^{\prime},r)=1}}  a_{r m^{\prime}} |m^{\prime}|^{it} \mathcal{S}(t,d m^{\prime}; \chi_1,\chi_2 , q_2 f^{\prime} \widetilde{y}_r f^{\prime \prime (r)}) \Big |^2,
 \end{multline}
 where the summation over $d$ was introduced by positivity.
 
Given $f$ and $q_2$ (and $r$ as above) such that $q_2 \mid f \mid D$, the integer $\mathcal{F}(q_2,f,r):=q_2 f^{\prime} \widetilde{y}_r f^{\prime \prime (r)}$ satisfies 
 \begin{equation} \label{dividechain}
 q_2 \mid \mathcal{F}(q_2,f,r) \mid f \mid D.
 \end{equation}
Applying \eqref{keyineq} to each summand of the right side of  \eqref{goodform} yields
\begin{multline} \label{goodform2}
\sum_{\mathfrak{a}} \Big | \sum_m a_m \sqrt{|m|} \rho_{\mathfrak{a}}(m,t) \Big |^2 \leq 16 \tau(D)^2 \tau(r)^4 \sum_{d \mid (r,D)}  \frac{\pi}{|\Gamma(\frac{1}{2}+it)|^2}  \\
 \hspace{0.15cm} \sum_{f \mid D} \hspace{0.15cm} \sum_{q_2 \mid f} \hspace{0.15cm}  \sum_{q_1 \mid \frac{D}{f}} \frac{\widetilde{f}}{D f \phi(\widetilde{f})} \sideset{}{'}  \sum_{\substack{\chi_1 \hspace{-0.2cm} \pmod{q_1} \\ \chi_2 \hspace{-0.2cm} \pmod{q_2} \\ \chi_1 \overline{\chi_2} \sim \chi}} \Big | \sum_{\substack{m^{\prime} \\ (m^{\prime},r)=1}}  a_{r m^{\prime}} |m^{\prime}|^{it} \mathcal{S}(t,d m^{\prime}; \chi_1,\chi_2 , \mathcal{F}(q_2,f,r)) \Big |^2.
\end{multline}
We now relate the right side of \eqref{goodform2} to a sum over all cusps. Observe that we have
 \begin{equation}
 \frac{\widetilde{f}}{D f \phi(\widetilde{f})} \leq  \tau(D)  \frac{\widetilde{\mathcal{F}(q_2,f,r)}}{D \mathcal{F}(q_2,f,r) \phi(\widetilde{\mathcal{F}(q_2,f,r)})}, 
 \end{equation}
 so \eqref{goodform2} becomes 
 \begin{multline} 
\sum_{\mathfrak{a}} \Big | \sum_m a_m \sqrt{|m|} \rho_{\mathfrak{a}}(m,t) \Big |^2 \leq 16 \tau(D)^3 \tau(r)^4 \sum_{d \mid (r,D)}  \frac{\pi}{|\Gamma(\frac{1}{2}+it)|^2}  
 \hspace{0.15cm} \sum_{f \mid D} \hspace{0.15cm} \sum_{q_2 \mid f} \hspace{0.15cm}  \sum_{q_1 \mid \frac{D}{f}} \\ \frac{\widetilde{\mathcal{F}(q_2,f,r)}}{D \mathcal{F}(q_2,f,r) \phi(\widetilde{\mathcal{F}(q_2,f,r)})} \sideset{}{'}  \sum_{\substack{\chi_1 \hspace{-0.2cm} \pmod{q_1} \\ \chi_2 \hspace{-0.2cm} \pmod{q_2} \\ \chi_1 \overline{\chi_2} \sim \chi}} 
\Big | \sum_{\substack{m^{\prime} \\ (m^{\prime},r)=1}}  a_{r m^{\prime}} |m^{\prime}|^{it}      \mathcal{S}(t,d m^{\prime}; \chi_1,\chi_2 , \mathcal{F}(q_2,f,r)) \Big |^2.
\end{multline}
Recalling \eqref{dividechain}, by positivity we obtain
 \begin{multline} \label{goodform3}
\sum_{\mathfrak{a}} \Big | \sum_m a_m \sqrt{|m|} \rho_{\mathfrak{a}}(m,t) \Big |^2 \leq 16 \tau(D)^3 \tau(r)^4 \sum_{d \mid (r,D)}  \frac{\pi}{|\Gamma(\frac{1}{2}+it)|^2}  
 \hspace{0.15cm} \sum_{f \mid D} \hspace{0.15cm} \sum_{q_2 \mid f} \hspace{0.15cm} \sum_{\substack{\mathcal{F} \\ q_2 \mid \mathcal{F} \mid f }} \hspace{0.1cm} \\
 \sum_{q_1 \mid \frac{D}{f}} \frac{\widetilde{\mathcal{F}}}{D \mathcal{F} \phi(\widetilde{\mathcal{F}})} \sideset{}{'}  \sum_{\substack{\chi_1 \hspace{-0.2cm} \pmod{q_1} \\ \chi_2 \hspace{-0.2cm} \pmod{q_2} \\ \chi_1 \overline{\chi_2} \sim \chi}} \Big | \sum_{\substack{m^{\prime} \\ (m^{\prime},r)=1}}  a_{r m^{\prime}} |m^{\prime}|^{it}     
\mathcal{S}(t,d m^{\prime}; \chi_1,\chi_2 , \mathcal{F}) \Big |^2,
\end{multline}
where $\widetilde{\mathcal{F}}=(\mathcal{F},D/\mathcal{F})$.
By positivity we may extend the sum over $q_1$ in \eqref{goodform3} to all $q_1$ satisfying $q_1 \mid \frac{D}{\mathcal{F}}$. Then interchanging the summations over $f, q_2$ and $\mathcal{F}$ we obtain 
\begin{multline} \label{goodform4}
\sum_{\mathfrak{a}} \Big | \sum_m a_m \sqrt{|m|} \rho_{\mathfrak{a}}(m,t) \Big |^2 \leq 16 \tau(D)^3 \tau(r)^4 \sum_{d \mid (r,D)}  \frac{\pi}{|\Gamma(\frac{1}{2}+it)|^2}  \\
 \hspace{0.15cm} \sum_{\mathcal{F} \mid D} \hspace{0.15cm} \sum_{q_2 \mid \mathcal{F}} \hspace{0.15cm} \sum_{\mathcal{F} \mid f \mid D}  \hspace{0.15cm} \frac{\widetilde{\mathcal{F}}}{D \mathcal{F} \phi(\widetilde{\mathcal{F}})} \sum_{q_1 \mid \frac{D}{\mathcal{F}}}  \sideset{}{'}  \sum_{\substack{\chi_1 \hspace{-0.2cm} \pmod{q_1} \\ \chi_2 \hspace{-0.2cm} \pmod{q_2} \\ \chi_1 \overline{\chi_2} \sim \chi}} \Big | \sum_{\substack{m^{\prime} \\ (m^{\prime},r)=1}}  a_{r m^{\prime}} |m^{\prime}|^{it}    
\mathcal{S}(t,d m^{\prime}; \chi_1,\chi_2 , \mathcal{F} ) \Big |^2.
\end{multline} 
Now observe that all summands on the right side of \eqref{goodform4} are completely independent of $f$. Thus we can remove the summation on $f$ with the cost of an extra $\tau(D)$ factor. Then applying \eqref{goodform} to each summand of the $d$ summation yields \eqref{eisenineq}.
\end{proof}

\section{Spectral methods and shifted convolutions sums} \label{spectralsec}
 We choose a large parameter 
 \begin{equation} \label{largeparam}
 C:=N^{1000},
 \end{equation}
 and throughout this section makes the general assumption that
 \begin{equation*}
 h \asymp N \geq 20M.
 \end{equation*}
 For $\ell_1,\ell_2, h \in \mathbb{Z}_{\geq 1}$,
recall the definitions \eqref{Dshift} and \eqref{Sshift}. That is,
\begin{equation*}
\mathcal{D}(\ell_1,\ell_2,h,N,M)=\sum_{\ell_1 n-\ell_2 m=h} a(m) a(n) V_1 \Big( \frac{\ell_2 m}{M} \Big) V_2 \Big( \frac{\ell_1 n}{N} \Big),
\end{equation*}
and 
\begin{equation*}
\mathcal{S}(\ell_1,\ell_2,d,N,M)=\sum_{r \geq 1} \mathcal{D}(\ell_1,\ell_2,rd,N,M),
\end{equation*}
where $a(n)$ are the Fourier coefficients (normalised as in \eqref{fourier}) of $f \in \mathcal{S}_k(4)$ and $k:=\frac{1}{2}+2j$, $j \in \mathbb{N}$ . Without loss of generality we can assume that 
\begin{equation} \label{ellbd}
1 \leq \ell_1,\ell_2 \leq 2N,
\end{equation}
otherwise $\mathcal{D}(\ell_1,\ell_2,h,N,M)$ vanishes trivially. Slightly more generally than in Section \ref{errorarg} we assume 
\begin{equation*}
V_{1,2} \quad \text{supported in} \quad [1,2] \quad \text{and satisfy} \quad V^{(j)}_{1,2} \ll_{j,\varepsilon} C^{j \varepsilon}.
\end{equation*}

\begin{proof}[Proof of Proposition~\ref{spectralest}]
The argument here is similar to the proof of \cite[Proposition~8]{BM1}. Our exposition will be sparse, sketching only the details unique to our situation. We refer the reader to \
\cite[Sections~7 and 8]{BM1} for more details. Let $W$ be a smooth function with bounded derivatives such that 
\begin{equation*}
W(x) \equiv 1 \quad \text{for} \quad 1 \leq x \leq 2 \quad \text{and} \quad \text{supp}(W) \subseteq [1/2,3].
\end{equation*}
After attaching a smooth redundant weight function $W$ we obtain 
\begin{align*}
\mathcal{D}(\ell_1,\ell_2,h,N,M)&=\sum_{\substack{m,n \\ \ell_1 n- \ell_2 m=h}} a(m) a(n) V_1 \Big( \frac{\ell_2 m}{M} \Big) V_2 \Big(\frac{\ell_2 m+h}{N} \Big) W \Big( \frac{\ell_1 n-h}{M} \Big) \\
&=\int_{-\infty}^{\infty} V^{\dagger}_2(z) e \Big( \frac{zh}{N} \Big) \mathcal{D}_z (\ell_1,\ell_2,h,N,M) dz,
\end{align*}
where $V_2^{\dagger}$ denotes the Fourier transform, 
\begin{equation*}
\mathcal{D}_z(\ell_1,\ell_2,h,N,M):=\sum_{\substack{m,n \\ \ell_1 n-\ell_2 m=h}} a(m) a(n) V_{z} \Big( \frac{\ell_2 m}{M} \Big) W \Big( \frac{\ell_1 n-h}{M} \Big),
\end{equation*}
and 
\begin{equation*}
V_z(x):=V_1(x) e \Big( zx \frac{M}{N} \Big).
\end{equation*}
We truncate the $z$-integral at $|z| \leq C^{\varepsilon}$ with a small error, say $O(C^{-100})$. 

With the the notation as in \cite[Lemma~19]{BM1}, we make the choice of parameters
\begin{equation*}
Q:=C \quad \text{and} \quad \delta:=C^{-1}.
\end{equation*} 
Let $w_0$ be a fixed smooth function with support in $[1,2]$ and let 
\begin{equation*}
w(c)=\begin{cases}
w_0 \big( \frac{c}{C} \big) \quad \text{if} \quad 16 \ell_1 \ell_2 \mid c, \\
0 \quad \text{otherwise}.
\end{cases}
\end{equation*}
We see that 
\begin{equation} \label{lambdalbd}
\Lambda \gg C^2 (\ell_1 \ell_2)^{-1-\varepsilon},
\end{equation}
(cf. \cite[Lemma~19]{BM1}).
Applying Jutila's circle method \cite[Lemma~19]{BM1}, \eqref{L2} and arguing as in \cite[pp.~484--485]{BM1} we obtain 
\begin{equation} \label{redz}
\mathcal{D}_z(\ell_1,\ell_2,h,N,M)=\frac{1}{2 \delta} \int_{-\delta}^{\delta} \mathcal{D}_{z,\eta}(\ell_1,\ell_2,h,N,M) d \eta +O(C^{-\frac{2}{5}}),
\end{equation}
where 
\begin{align} 
D_{z,\eta}(\ell_1,\ell_2,h,N,M)&:=\frac{1}{\Lambda} \sum_{16 \ell_1 \ell_2 \mid c} w_0 \Big( \frac{c}{C} \Big) \sum_{\substack{d \hspace{-0.2cm} \mod c \\ (c,d)=1}} \sum_{m,n} a(m) a(n) e \Big( \frac{d}{c}(\ell_1 n-\ell_2 m-h) \Big)  \nonumber \\
& \times W_{\eta M} \Big( \frac{\ell_1 n-h}{M} \Big) V_{z,\eta M} \Big( \frac{\ell_2 m}{M} \Big), \label{Dzeta}
\end{align}
where 
\begin{equation*}
V_{z,\eta}(x):=V_z(x) e (-\eta x)=V_1(x) e \Big( x \Big(z \frac{M}{N} -\eta \Big) \Big) \quad \text{and} \quad W_{\eta} (x):=W(x) e( \eta x).
\end{equation*}
We stress that only Cauchy--Schwarz and \eqref{L2} were used to obtain the error term of $C^{-\frac{2}{5}}$ in \eqref{redz}, not \eqref{ram}. Since $|\eta| \leq C^{-1}=N^{-1000}$ (in particular $\eta \ll M^{-1}$), the functions $V_{z,\eta M}$ and $W_{\eta M}$ are well behaved. In particular, 
\begin{equation*}
W^{(j)}_{\eta M} \ll 1 \quad \text{and} \quad V^{(j)}_{z, \eta M} \ll C^{j \varepsilon} \quad \text{uniformly in} \quad |z| \ll C^{j \varepsilon}. 
\end{equation*}
Observe that $V_{z,\eta M}$ and $W_{\eta M}$ have support in $[1,2]$ and $[\frac{1}{2},3]$ respectively.

Here we will see that a Voronoi summation in $m,n$ variables of \eqref{Dzeta} leads to a twist by a quadratic character depending on $\ell_1$ and $\ell_2$. Applying Lemma \ref{Voronoi} to the $m$ summation in \eqref{Dzeta} we obtain 
\begin{equation} \label{voronoi1}
\sum_m a(m) e \Big({-\frac{dm}{c/\ell_2}} \Big) V_{z,\eta M} \Big( \frac{\ell_2 m}{M} \Big)=\frac{M}{c} \sum_m a(m) e \Big(\frac{ \ell_2 m \overline{d}}{c} \Big) \nu_{\theta}(\gamma_2)  \mathring{V}_{z, \eta M} \Big( \frac{\ell_2 m M}{c^2} \Big),
\end{equation}
where $d \overline{d} \equiv 1 \pmod{c}$, $\widetilde{d}$ is any integer such that $\widetilde{d} \equiv d \pmod{c}$ and
\begin{equation} \label{gamma2defn}
\gamma _2=\begin{pmatrix}
-\widetilde{d} & b_2 \\
c/\ell_2 & X_2
\end{pmatrix} \in \Gamma_0(c/\ell_2). 
\end{equation}
Applying Lemma \ref{Voronoi} to the $n$ summation in \eqref{Dzeta} we obtain
\begin{equation} \label{voronoi2}
\sum_n a(n) e \Big( \frac{d n}{c / \ell_1} \Big) W_{\eta M} \Big( \frac{\ell_1 n-h}{M} \Big)=\frac{M}{c} \sum_n a(n)  \nu_{\theta}(\gamma_1) e \Big({-\frac{\ell_1 n \overline{d}}{c}} \Big) W^{*}_{\eta M} \Big( \frac{h \ell_1 n}{c^2},\frac{M \ell_1 n}{c^2} \Big),
\end{equation}
where  
\begin{equation} \label{gamma1defn}
 \gamma _1=\begin{pmatrix}
\widetilde{d} & b_1 \\
c/\ell_1 & X_1
\end{pmatrix} \in \Gamma_0(c/\ell_1),
\end{equation} 
and 
\begin{equation*}
W^{*}_{\eta M}(z,w):=2 \pi i^{k} \int_{0}^{\infty} W_{\eta M}(y) J_{k-1} (4 \pi \sqrt{yw+z} ) dy.
\end{equation*}
It follows from \eqref{thetamultiplier}, \eqref{epsdefn}, \eqref{gamma2defn}, and \eqref{gamma1defn} that
\begin{equation} \label{quadchar}
\nu_{\theta} (\gamma_1) \nu_{\theta}(\gamma_2)=\Big(\frac{c/\ell_1}{X_1} \Big) \overline{\varepsilon}_{X_1} \Big(\frac{c/\ell_2}{X_2} \Big) \overline{\varepsilon}_{X_2}=-i \Big(\frac{4 \ell_1 \ell_2}{d}  \Big)
\end{equation}
for $d$ such that $(d,c)=1$. Note that in display \eqref{quadchar} we used the feature that the $X_i$ are defined via the determinants of the $\gamma_i$ for $i=1,2$.
Let 
\begin{equation*}
\chi:=\chi_{\ell_1 \ell_2}=\Big(\frac{4 \ell_1 \ell_2}{\bullet}  \Big).
\end{equation*}
Observe that $\chi$ is an even character modulo $4 \ell_1 \ell_2$ (in particular modulo $16 \ell_1 \ell_2$). Combining \eqref{Dzeta}--\eqref{quadchar} and \cite[Lemma~17]{BM1} we obtain 
\begin{multline} \label{usefulD}
D_{z,\eta}(\ell_1,\ell_2,h,N,M):=-\frac{M^2i}{\Lambda C} \sum_{16 \ell_1 \ell_2 \mid c} w_1 \Big( \frac{c}{C} \Big) \frac{1}{c} \sum_{m,n} a(m) a(n) K(\ell_1 n-\ell_2 m,h,c,\chi) \\
\times \sum_{\pm} W_{\pm} \Big( \frac{h \ell_1 n}{c^2},\frac{M \ell_1 n}{c^2} \Big) e \Big( \pm 2 \frac{\sqrt{h \ell_1 n}}{c} \Big) \mathring{V}_{z,\eta M} \Big( \frac{\ell_2 m}{c^2/M} \Big)+O(C^{-A}),
\end{multline} 
where 
\begin{equation*}
w_1(x)=\frac{w_0(x)}{x},
\end{equation*}
and $W_{\pm}$ are as in \cite[Lemma~17]{BM1}. By \cite[(6.15)]{BM1} and the fact that $\mathring{V}_{z,\eta M }$ is a Schwarz class function (cf. \eqref{BM2}) we can restrict to 
\begin{equation} \label{cutoffs}
\ell_1 n \leq \mathcal{N}_0:=\frac{C^{2+\varepsilon} N}{M^2} \quad \text{and} \quad \ell_2 m \leq \mathcal{M}_0:=\frac{C^{2+\varepsilon}}{M},
\end{equation}
with negligible error.  For $\mathcal{N} \leq \mathcal{N}_0$, $\mathcal{M} \leq \mathcal{M}_0$ and $\mathcal{K}>0$, we will restrict the right side of \eqref{usefulD} to subsums 
\begin{equation*}
n \asymp \mathcal{N}, \quad m \asymp \mathcal{M}, \quad \text{and} \quad |\ell_1 n -\ell_2 m| \asymp \mathcal{K}.
\end{equation*}
Here $x \asymp X$ denotes $X \leq x \leq 2X$. The arising subsums are then split into three sums $\Sigma_{+}, \Sigma_{0}$ and $\Sigma_{-}$ according to 
\begin{align*}
 \Sigma_{+} &:  \ell_1 n>\ell_2 m; \\
 \Sigma_{0} &:  \ell_1 n=\ell_2 m; \\
 \Sigma_{-} &:  \ell_1 n<\ell_2 m.
\end{align*}

\subsection{Treatment of $\Sigma_0$}
Let $\chi^{\star}$ be the primitive character of conductor $C_{\chi}^{\star} \mid 4 \ell_1 \ell_2$ inducing $\chi \mathbf{1}_c$ modulo $c$. Then by \cite[Lemma~3.1.3]{Mi} we have
\begin{equation*}
S(0,h,c,\chi)=\mathcal{G}_{\chi^{\star}}(1;C_{\chi}^{\star}) \sum_{\substack{d>0 \\ d \mid (h,c/C_{\chi}^{\star} ) }} d \mu \Big( \frac{c}{d C_{\chi}^{\star}} \Big) \chi^{\star} \Big(\frac{c}{d C_{\chi}^{\star}} \Big) \overline{\chi^{\star}} \Big( \frac{h}{d} \Big).
\end{equation*}
Thus 
\begin{equation} \label{gaussbound}
|S(0,h,c,\chi)| \leq |C_{\chi}^{\star}|^{\frac{1}{2}} \tau(h) (h,c) \ll (\ell_1 \ell_2)^{\frac{1}{2}} \tau(h) (h,c).
\end{equation}
A trivial estimate using \eqref{gaussbound}, Cauchy--Schwarz, \eqref{L2}, \eqref{lambdalbd}, and \eqref{cutoffs} we obtain 
\begin{align*}
\Sigma_0  &\ll \frac{M^2 \tau(h) (\ell_1 \ell_2)^{\frac{1}{2}}}{\Lambda C^{1-\varepsilon}} \sum_{C \leq c \leq 2C} \frac{(h,c)}{c} \sum_{\substack{\ell_1 n \asymp \mathcal{N}, \ell_2 m \asymp \mathcal{M} \\ \ell_1 n =\ell_2 m }} |a(m) a(n)| \\
& \ll  \frac{M^2 \tau(h)^2 (\ell_1 \ell_2)^{\frac{1}{2}}}{\Lambda C^{1-\varepsilon}} \Big( \sum_{m \ll \mathcal{M}} |a(m)|^2 \Big)^{\frac{1}{2}} \Big( \sum_{n \ll \mathcal{N}} |a(n)|^2 \Big)^{\frac{1}{2}} \\
& \ll \frac{M^2 \tau(h)^2 (\ell_1 \ell_2)^{\frac{1}{2}} (\mathcal{N}_0 \mathcal{M}_0)^{\frac{1}{2}}   }{\Lambda C^{1-\varepsilon}} \ll \frac{ (\ell_1 \ell_2)^{\frac{3}{2}+\varepsilon} (NM)^{\frac{1}{2}} }{C^{1-\varepsilon}} \ll C^{-\frac{1}{2}},
\end{align*} 
where the last equality follows from \eqref{largeparam} and \eqref{ellbd}.

\subsection{Spectral treatment of $\Sigma_+$}
Now we consider 
\begin{equation} \label{sigmaplus1}
\Sigma_+=-\frac{iM^2}{\Lambda C} \sum_{\substack{b>0 \\ |b| \asymp \mathcal{K}}} \hspace{0.2cm} \sum_{\substack{\ell_1 n -\ell_2 m=b \\ \ell_1 n \asymp \mathcal{N}, \ell_2 m \asymp \mathcal{M}}} a(m) a(n) \sum_{16 \ell_1 \ell_2 \mid c} \frac{K(b,h,c,\chi)}{c} \Phi \Big(4 \pi \frac{\sqrt{|b|} h}{c} \Big),
\end{equation}
where $\Phi$ is defined on \cite[pg.~487]{BM1} (with a $\sum_{\pm}$ inserted into their definition). In view of the transforms occurring in the Kuznetsov formula, define 
\begin{equation*}
\mathcal{J}^{+}_{2it}(x):=\pi i \frac{J_{2it}(x)-J_{-2it}(x)}{\text{sinh}(\pi t)} \quad \text{and} \quad
\mathcal{J}^{-}_{2it}(x):=4 \cosh(\pi t) K_{2it}(x).
\end{equation*}
Let $\widetilde{\Phi}$, $\dot{\Phi}$, $\Omega$ be defined as they are on \cite[pg.~487]{BM1}.  Also define the quantities
\begin{equation*}
 \mathcal{T}_+:=C^{\varepsilon} \Big(1+\Big(\frac{\mathcal{K} N}{C^2} \Big)^{\frac{1}{4}}+\Big(\frac{\mathcal{M} N}{C^2} \Big)^{\frac{1}{2}} \Big) 
 \quad \text{ and } \quad \mathcal{T}_h:=C^{\varepsilon} \Big(1+\Big( \frac{\mathcal{K} N}{C^2} \Big)^{\frac{1}{4}} \Big). 
\end{equation*}
By the argument on \cite[pg~487]{BM1}, the transforms $\widetilde{\Phi}(t)$ and $\dot{\Phi}(\ell)$ are negligible (cf. \cite[Lemma~16]{BM1}) unless 
\begin{equation*}
|t| \ll \mathcal{T}_{+} \quad \text{and} \quad \ell \ll \mathcal{T}_h
\end{equation*}
respectively. Applying Lemma \ref{kuz} (recalling that $\chi$ is an even Dirichlet character) to the summation over $c$ in \eqref{sigmaplus1} and then truncating the appropriate summations and integrations using the above remarks we obtain 
\begin{equation*} 
\Sigma_{+}=\mathcal{H}_+(h)+\mathcal{M}_+(h)+\mathcal{E}_{+}(h)+O(C^{-A}),
\end{equation*}
where the terms on the right side correspond to the holomorphic, Maass and Eisenstein components of the spectrum. They are 
\begin{align} 
\mathcal{H}_+(h)&:=-\frac{iM^2}{\Lambda C} \int_{0}^{\infty} \sum_{\substack{2 \leq \ell \leq \mathcal{T}_h \\ \ell \equiv 0 \hspace{-0.1cm} \pmod{2}}} \sum_{g \in \mathcal{H}_\ell(16 \ell_1 \ell_2,\chi)} 4 i^{\ell} \Gamma(\ell) J_{\ell-1}(x) \sqrt{h} \rho_g (h) \nonumber \\
& \times \sum_{\substack{b>0 \\ |b| \asymp \mathcal{K}}} w_1 \Big( \frac{4 \pi \sqrt{|b| h}}{Cx} \Big) \sqrt{|b|} \rho_g(b) \gamma_+(b,h,x) \frac{dx}{x}, \label{Hplus}
\end{align}
\begin{align} 
\mathcal{M}_+(h)&:=-\frac{2 i M^2}{\Lambda C} \int_{0}^{\infty} \sum_{\substack{g \in \mathcal{B}_0(16 \ell_1 \ell_2,\chi) \\ |t_g| \leq \mathcal{T}_+}} \frac{\mathcal{J}^{+}_{2it_g}(x)}{\cosh(\pi t_g)} \sqrt{h} \rho_g(h) \nonumber \\
& \times \sum_{\substack{b>0 \\ |b| \asymp \mathcal{K}}} w_1 \Big( \frac{4 \pi \sqrt{|b| h}}{Cx} \Big) \sqrt{|b|} \rho_g(b) \gamma_+(b,h,x) \frac{dx}{x}, \label{Mplus}
\end{align}
\begin{align} 
\mathcal{E}_+(h)&:=-\frac{2i M^2}{\Lambda C} \int_0^{\infty} \frac{1}{4 \pi} \sum_{\mathfrak{a}} \int_{-\mathcal{T}_+}^{\mathcal{T}_+} \frac{\mathcal{J}^{+}_{2it}(x)}{\cosh(\pi t)} \sqrt{h} \rho_{\mathfrak{a}}(h,t)  \nonumber \\
& \times \sum_{\substack{b>0 \\ |b| \asymp \mathcal{K}}} w_1 \Big( \frac{4 \pi \sqrt{|b| h}}{Cx} \Big) \sqrt{|b|} \rho_{\mathfrak{a}}(b,t) dt  \gamma_+(b,h,x) \frac{dx}{x}, \label{Eplus}
\end{align} 
where 
\begin{equation*}
\gamma_+(b,h,x):=\sum_{\substack{\ell_1 n -\ell_2 m=b \\ \ell_1 n \asymp \mathcal{N}, \ell_2 m \asymp \mathcal{M}}} a(m) a(n) \mathring{V}_{z,\eta M} \Big( \frac{x^2 \ell_2 m M}{(4 \pi)^2 |b| h} \Big) \sum_{\pm} W_{\pm} \Big( \frac{x^2 \ell_1 n}{(4 \pi)^2 |b|},\frac{x^2 \ell_1 n M}{(4 \pi)^2 |b| h} \Big) \theta^{\pm}_x \Big( \frac{\ell_2 m}{|b|} \Big),
\end{equation*}
and
\begin{equation*}
\theta^{\pm}_x(y):=\exp \Big(\pm i x \sqrt{1+y} \Big) v \Big( \frac{y}{\mathcal{M}/\mathcal{K}} \Big),
\end{equation*}
and $v$ a redundant smooth weight function of compact support on $[1/4,3]$ that is constantly 1 on $[1/2,2]$. 

\subsection{Spectral treatment of $\Sigma_{-}$}
Now consider
\begin{equation*}
\Sigma_-=-\frac{iM^2}{\Lambda C} \sum_{\substack{b<0 \\ |b| \asymp \mathcal{K}}} \hspace{0.2cm} \sum_{\substack{\ell_1 n -\ell_2 m=b \\ \ell_1 n \asymp \mathcal{N}, \ell_2 m \asymp \mathcal{M}}} a(m) a(n) \sum_{16 \ell_1 \ell_2 \mid c} \frac{K(b,h,c,\chi)}{c} \Phi \Big(4 \pi \frac{\sqrt{|b|} h}{c} \Big). 
\end{equation*}
Define 
\begin{equation*}
\mathcal{T}_{-}:=C^{\varepsilon} \Big(1+ \Big(\frac{\mathcal{M} N}{C^2} \Big)^{\frac{1}{2}} \Big).
\end{equation*}
Applying a similar argument to \cite[pp.~488--489]{BM1} using Lemma \ref{kuz} we obtain
\begin{equation*}
\Sigma_{-}=\mathcal{M}_{-}(h)+\mathcal{E}_{-}(h)+O(C^{-A}),
\end{equation*}
where 
\begin{align*}
\mathcal{M}_{-}(h)&=-\frac{2iM^2}{\Lambda C} \int_{0}^{\infty} \int_{\sigma-i C^{\varepsilon} \mathcal{T}_{-}}^{\sigma+i C^{\varepsilon} \mathcal{T}_{-}} \sum_{\substack{g \in \mathcal{B}_0(16 \ell_1 \ell_2,\chi) \\ |t_g| \leq \mathcal{T}_{-}}} \widehat{\mathcal{J}}^{-}_{2it_g}(s) \frac{\sqrt{h} \rho_g(h) }{\cosh(\pi t_g)} w_1 \Big( \frac{\sqrt{h}}{Cx} \Big) \\
& \times \sum_{\substack{b<0 \\ |b| \asymp \mathcal{K}}} (4 \pi \sqrt{|b|} x )^{-s} \sqrt{|b|} \rho_g(b) \gamma_{-}(b,h,x) \frac{ds}{2 \pi i} \frac{dx}{x},
\end{align*}
\begin{equation*}
\gamma_{-}(b,h,x)=\sum_{\substack{\ell_1 n -\ell_2 m=b \\ \ell_1 n \asymp \mathcal{N}, \ell_2 m \asymp \mathcal{M}}} a(m) a(n) \mathring{V}_{z, \eta M} \Big(\frac{x^2 \ell_2 m M}{h} \Big) \sum_{\pm} W_{\pm} \Big(x^2 \ell_1 n, \frac{x^2 \ell_1 n M}{h}  \Big) e \Big( \pm 2x \sqrt{\ell_1 n} \Big),
\end{equation*}
and $\sigma=\frac{7}{32}+\varepsilon$. An analogous formula holds for the Eisenstein contribution $\mathcal{E}_{-}(h)$. 

\subsection{Summary of setup}
By the above discussion it suffices to estimate the right side of 
 \begin{align} \label{summary}
\mathcal{D}(\ell_1,\ell_2,h,N,M)&=\frac{1}{2 \delta} \int_{-\delta}^{\delta} \int_{-C^{\varepsilon}}^{C^{\varepsilon}} V^{\dagger}_2(z) e \Big( \frac{zh}{N} \Big) \sum_{\mathcal{N} \leq \mathcal{N}_0} \sum_{\substack{\mathcal{M} \leq \mathcal{M}_0 \\ \mathcal{M} \leq \mathcal{N}}} \sum_{\substack{\mathcal{K} \leq \mathcal{N}_0 \\ \mathcal{K} \leq \mathcal{N}}} ( \mathcal{H}_{+}(h)+\mathcal{M}_{+}(h)+\mathcal{E}_{+}(h) ) dz d \eta \\
&+\frac{1}{2 \delta} \int_{-\delta}^{\delta} \int_{-C^{\varepsilon}}^{C^{\varepsilon}} V^{\dagger}_2(z) e \Big( \frac{zh}{N} \Big) \sum_{\mathcal{N} \leq \mathcal{M}_0} \sum_{\mathcal{M} \leq \mathcal{M}_0} \sum_{\mathcal{K} \leq \mathcal{M}_0} (\mathcal{M}_{-}(h)+\mathcal{E}_{-}(h) ) dz d \eta +O(C^{-\frac{1}{3}}) \nonumber,
\end{align}
where $\mathcal{N}$, $\mathcal{M}$ and $\mathcal{K}$ run over numbers $\geq 1$ of the form $\mathcal{N}_0 2^{-\nu}$ or $\mathcal{M}_0 2^{-\nu}$, $\nu \in \mathbb{N}$.

\subsection{Shifted convolutions sums on average}
We now turn our attention to the averaged convolution sum
\begin{equation*}
\mathcal{S}(\ell_1, \ell_2, d, N,M)=\sum_r \mathcal{D}(\ell_1, \ell_2, rd, N,M). 
\end{equation*}
Let 
\begin{equation*}
\beta:=\text{lcm}(16,\ell_1,\ell_2,d) \quad \text{and} \quad B:=\{ n \in \N : p \mid n \implies p \mid \beta \text{ for all primes } p \}.
\end{equation*}
Observe that $\mathcal{D}(\ell_1,\ell_2,rd,N,M)$ vanishes unless $r \asymp N/d$. Recall \eqref{summary}. To begin, consider the decomposition 
\begin{equation*}
\sum_{r \asymp N/d} e \Big( \frac{zrd}{N} \Big) \mathcal{H}_{+}(rd)=\sum_{\substack{r_2 \ll N/d \\ r_2 \in B}} \hspace{0.15cm} \sum_{\substack{ r_1 \asymp N/(dr_2) \\ (r_1,\beta)=1}} 
e \Big( \frac{z r_1 r_2 d}{N} \Big) \mathcal{H}_{+}(r_1 r_2 d),
\end{equation*}
where $\mathcal{H}_{+}$ was defined in \eqref{Hplus}. Observe that the range of integration for $x$ in \eqref{Hplus} is
\begin{equation*}
x \asymp X_{+}:=\frac{\sqrt{\mathcal{K}N}}{C}.
\end{equation*}
We follow verbatim the Mellin inversion argument on \cite[pp.~490--491]{BM1} that separates the variables scattered throughout the various weight functions. In that argument,
\begin{equation*}
S:=C^{\varepsilon} \Big(1+\frac{X_{+} \mathcal{M}}{\sqrt{\mathcal{K} \mathcal{N}}} \Big).
\end{equation*}
This yields 
\begin{equation} \label{Xineq}
\sum_{r_1 \asymp N/(dr_2)} e \Big(\frac{z r_1 r_2 d}{N} \Big) \mathcal{H}_+(r_1 r_2 d) \ll \frac{C^{\varepsilon} M^2}{\Lambda C} \frac{\mathcal{K}^{\frac{1}{4}}}{X^{\frac{1}{2}}_+ \mathcal{N}^{\frac{1}{4}}} S (\Xi_{1,+}^{\mathcal{H}} \hspace{0.1cm} \Xi_{2,+}^{\mathcal{H}} )^{\frac{1}{2}},
\end{equation} 
where 
\begin{equation*}
\Xi_{1,+}^{\mathcal{H}}:=\max_{|u_4| \leq C^{\varepsilon}} \sum_{\substack{0< \ell \leq \mathcal{T}_h \\ \ell \equiv 0 \hspace{-0.1cm} \pmod{2}}} \Gamma(\ell) \sum_{g \in \mathcal{H}_{\ell}(16 \ell_1 \ell_2,\chi)} \Big | \sum_{\substack{r_1 \asymp N/(dr_2) \\ (r_1,\beta)=1}} e \Big( \frac{z r_1 r_2 d}{N} \Big) r_1^{2 \varepsilon+i u_4} \sqrt{r_1 r_2 d} \rho_g(r_1 r_2 d) \Big |^2, 
\end{equation*}
and 
\begin{equation*}
\Xi_{2,+}^{\mathcal{H}}=\max_{\substack{|u_2| \leq C^{\varepsilon} \\ |u_1|, |u_3| \leq S \\ x \asymp X_+}} \sum_{\substack{0<\ell \leq \mathcal{T}_h \\ \ell \equiv 0 \hspace{-0.15cm} \pmod{2}}} |J_{\ell-1}(x)|^2 \Gamma(\ell) \sum_{g \in \mathcal{H}_{\ell}(4 \ell_1 \ell_2, \chi)} \Big | \sum_{|b| \asymp \mathcal{K}} \sqrt{|b|} \rho_g(b) \gamma^*(b) \Big |^2,
\end{equation*}
with
\begin{equation*}
\gamma^*(b)=\Big(\frac{|b|}{\mathcal{K}} \Big)^{\frac{1}{4}+\varepsilon+i u_3} \sum_{\substack{\ell_1 n-\ell_2 m=b \\ \ell_1 n \asymp \mathcal{N}, \ell_2 m \asymp \mathcal{M}}}  \Big( \frac{\ell_1 n}{\mathcal{N}} \Big)^{-\frac{1}{4}+i u_2} \Big(\frac{\ell_2 m}{\mathcal{M}} \Big)^{-\varepsilon+i u_1} a(m) a(n).
\end{equation*}
The same analysis works mutatis mutandis for the Eisenstein and Maass spectrum, giving analogous expressions for $\Xi^{\mathcal{E}}_{i,+}$ and $\Xi^{\mathcal{M}}_{i,+}$
for $i=1,2$. Note that breakdown of the Archimedean weight functions in the Cauchy-Schwarz inequality in $x$ and $g$ is analogous: 
 $\Xi^{\mathcal{M}}_{1,+}$ has a factor of $1/\cosh(\pi t_g)$ and $\Xi^{\mathcal{M}}_{2,+}$ has a factor $|\mathcal{J}^{+}_{2it_g}(x)|^2/\cosh(\pi t_g)$.
 
We now bound the various $\Xi_{i,+}^{\star}$ for $i=1,2$. Applying the argument at the beginning of \cite[pg~492]{BM1} using \eqref{L2} and \eqref{wilton} we obtain
\begin{equation} \label{lsieveinput}
\sum_b |\gamma^*(b)|^2 \ll C^{\varepsilon} \frac{\mathcal{N} \mathcal{M}}{\ell_1 \ell_2},
\end{equation}
uniformly in $u_1,u_2$ and $u_3$. Also by \cite[pg~492]{BM1} we have 
\begin{equation}  \label{Jbd}
J_{\ell-1}(x) \ll C^{\varepsilon} x^{-\frac{1}{2}},
\end{equation}
uniformly for all $x>0$ and $1 \leq \ell \leq \mathcal{T}_{h}$. Thus by Lemma \ref{largesieve}, \eqref{lsieveinput} and \eqref{Jbd} we have 
\begin{equation} \label{Hplus2bd}
\Xi^{\mathcal{H}}_{2,+} \ll  \frac{C^{\varepsilon}}{X_{+}} \Big( \mathcal{T}_h^2+\frac{\mathcal{K}}{\ell_1 \ell_2} \Big) \frac{\mathcal{N} \mathcal{M}}{\ell_1 \ell_2}.
\end{equation}
Similarly, 
\begin{equation} \label{EMplus2bd}
|\Xi_{2,+}^{\mathcal{E}}|+|\Xi_{2,+}^{\mathcal{M}}| \ll \frac{C^{\varepsilon}}{X_{+}} \Big( \mathcal{T}_+^2+\frac{\mathcal{K}}{\ell_1 \ell_2} \Big) \frac{\mathcal{N} \mathcal{M}}{\ell_1 \ell_2}.
\end{equation}
By \eqref{ramult} and the fact that integral weight Hecke cusp forms satisfy Deligne's bound (also see Remark \ref{naiveapproach})
we have 
\begin{equation} \label{multapply}
\Xi^{\mathcal{H}}_{1,+} \ll \max_{|u_4| \leq C^{\varepsilon}} C^{\varepsilon} \sum_{\delta \mid 16 \ell_1 \ell_2} \sum_{\substack{2 \leq \ell \leq \mathcal{T}_h \\ \ell \equiv 0 \hspace{-0.1cm} \pmod{2}}} \Gamma(\ell) \sum_{g \in \mathcal{H}_{\ell}(16 \ell_1 \ell_2,\chi)} \Big | \sum_{\substack{r_1 \asymp N/(dr_2) \\ (r_1,\beta)=1  }} \alpha(r_1) \sqrt{r_1 \delta} \rho_g(r_1 \delta)  \Big |^2,
\end{equation}
where 
\begin{equation*}
\alpha(r_1)=\alpha_{r_2 d,u_4}(r_1)=e \Big( \frac{z r_1 r_2 d}{N} \Big) r_1^{2 \varepsilon+iu_4}.
\end{equation*}
Applying Lemma \ref{largesieve} to the right side of \eqref{multapply} yields
\begin{equation} \label{Hplus1bd}
\Xi^{\mathcal{H}}_{1,+} \ll C^{\varepsilon} \sum_{\delta \mid 16 \ell_1 \ell_2 } \Big( \mathcal{T}_h^2+\frac{N \delta}{dr_2 \ell_1 \ell_2} \Big) \frac{N}{dr_2} \ll C^{\varepsilon} \Big( \mathcal{T}_h^2+\frac{N}{dr_2} \Big) \frac{N}{dr_2}. 
\end{equation}
Observe that all cusps of $\Gamma_0(16 \ell_1 \ell_2)$ are singular relative to $\chi$ by Lemma \ref{youngsingular} (i.e. the conductor of $\chi$ is of the form $Z$ or $4Z$ where $Z$ is odd and squarefree since it is a quadratic character). Thus we can apply Lemma \ref{speceisen} and Lemma \ref{largesieve} to similarly obtain 
\begin{equation} \label{Eplus1bd}
\Xi^{\mathcal{E}}_{1,+} \ll C^{\varepsilon} \Big( \mathcal{T}_{+}^2+\frac{N}{dr_2} \Big) \frac{N}{dr_2}.
\end{equation}
Applying Theorem \ref{BM1theorem} we obtain 
\begin{align} \label{Mplus1bd}
\Xi_{1,+}^{\mathcal{M}}&=\max_{|u_4| \leq C^{\varepsilon}} \sum_{\substack{|t_g| \leq \mathcal{T}_+ \\ g \in \mathcal{B}_0(16 \ell_1 \ell_2,\chi) }} \frac{1}{\text{cosh}(\pi t_g)} \Big | \sum_{\substack{r_1 \asymp N/(dr_2) \\ (r_1,\beta)=1}} \alpha(r_1) \sqrt{r_1 r_2 d} \rho_g(r_1 r_2 d) \Big |^2 \nonumber \\
& \ll C^{\varepsilon} (\ell_1 \ell_2,r_2 d) \Big( \mathcal{T}_{+}+\frac{(r_2 d)^{\frac{1}{2}}}{(\ell_1 \ell_2)^{\frac{1}{2}}} \Big) \Big( \mathcal{T}_{+}+\frac{N}{d r_2 (\ell_1 \ell_2)^{\frac{1}{2}} } \Big) \frac{N}{dr_2}.
\end{align}
Combining \eqref{Hplus2bd}, \eqref{EMplus2bd}, \eqref{Hplus1bd}, \eqref{Eplus1bd}, \eqref{Mplus1bd}, we obtain
\begin{multline} \label{combinedeq}
\big(|\Xi^{\mathcal{H}}_{1,+}|+|\Xi^{\mathcal{M}}_{1,+}|+|\Xi^{\mathcal{E}}_{1,+}| \big)\big(|\Xi^{\mathcal{H}}_{2,+}|+|\Xi^{\mathcal{M}}_{2,+}|+|\Xi^{\mathcal{E}}_{2,+}| \big) \ll \frac{N}{dr_2} \frac{(\ell_1 \ell_2,r_2 d) \mathcal{N} \mathcal{M}}{\ell_1 \ell_2}   \frac{C^{\varepsilon}}{X_+} \\
\times \Big( \Big( \mathcal{T}_++\frac{(r_2 d)^{\frac{1}{2}}}{(\ell_1 \ell_2)^{\frac{1}{2}}} \Big) \Big( \mathcal{T}_++\frac{N}{d r_2 (\ell_1 \ell_2)^{\frac{1}{2}}}\Big) +\frac{N}{dr_2}  \Big) \Big( {\mathcal{T}_{+}}^2+\frac{\mathcal{K}}{\ell_1 \ell_2} \Big).
\end{multline}
Inserting \eqref{combinedeq} into \eqref{Xineq} (recalling that \eqref{Xineq} has both $\mathcal{M}$ and $\mathcal{E}$ analogues), the ensuing brute force computation on \cite[pp.~493--494]{BM1} guarantees that 
\begin{multline} \label{plusconc}
\sum_{\substack{r_1 \asymp N/(dr_2) \\ (r_1, \beta)=1}} e \Big( \frac{zr_1 r_2 d}{N} \Big) \big(\mathcal{H}_{+}(r_1 r_2 d) +\mathcal{M}_{+}(r_1 r_2 d) +\mathcal{E}_{+}(r_1 r_2 d)   \big)
 \\
\ll C^{\varepsilon} (\ell_1 \ell_2,d)^{\frac{1}{2}} \Big( \frac{N}{d^{\frac{1}{2}}}+\frac{N^{\frac{5}{4}} M^{\frac{1}{4}}}{d (\ell_1 \ell_2)^{\frac{1}{4}}}+\frac{N^{\frac{3}{4}} M^{\frac{1}{4}}} {d^{\frac{1}{4}}}+\frac{NM^{\frac{1}{2}}}{d^{\frac{3}{4}} (\ell_1 \ell_2)^{\frac{1}{2}} }+\frac{NM^{\frac{1}{2}}}{d}  \Big).
\end{multline}

One can follow the argument \cite[pp.~494--495]{BM1} appealing to Lemmas \ref{largesieve} and \ref{speceisen}, as well as \eqref{weakversion}, whenever their principal character analogues are used, to conclude that
\begin{equation} \label{minusconc}
\sum_{\substack{r_1 \asymp N/(dr_2) \\ (r_1, \beta)=1}} e \Big( \frac{zr_1 r_2 d}{N} \Big) \big(\mathcal{M}_{-}(r_1 r_2 d) +\mathcal{E}_{-}(r_1 r_2 d)     \big) \ll C^{\varepsilon}  d^{\theta} (\ell_1 \ell_2,d)^{\frac{1}{2}} \Big(\frac{M^{\frac{1}{4}} N^{\frac{3}{4}}  }{d^{\frac{1}{2}}}+\frac{M^{\frac{3}{4}} N^{\frac{3}{4}} }{d}  \Big).
\end{equation}
Summing \eqref{plusconc} and \eqref{minusconc} over $r_2 \in B$ using Rankin's trick, and using $\theta \leq 1/4$, we obtain Proposition \ref{spectralest}.
\end{proof}

\section{Critical range and $\alpha n^2$ modulo one} \label{critsec}
\begin{proof}[Proof of Theorem~\ref{critrangethm}]
Recall that 
\begin{equation} \label{rangerecall}
1 \leq M \leq p/2, \quad   p^{\frac{1}{2}-\frac{1}{10}} \leq N \leq p^{\frac{1}{2}+\frac{1}{10}},
\end{equation}
and consider
\begin{equation*}
\sum_{N \leq n_1, n_2 \leq 2N} \Big | \sum_{M \leq m \leq 2M} S(m, c n_1,p) \overline{S(m, c n_2,p)} \Big |
\end{equation*} 
for $c \in \mathbb{F}^{\times}_p$. The estimates we obtain will not depend on $c$. Substituting \eqref{salieval} we obtain 
\begin{equation} \label{critset}
p \sum_{N \leq n_1,n_2 \leq 2N} \Big | \sum_{M \leq m \leq 2 M} \sum_{\substack{u,v \hspace{-0.2cm} \pmod{p} \\ u^2 \equiv cmn_1 \hspace{-0.1cm} \pmod{p} \\ v^2 \equiv cmn_2 \hspace{-0.1cm} \pmod{p}}} e \Big( \frac{2(u+v)}{p}  \Big)  \Big |.
\end{equation}
Let $R:=R(M,N,p)$ denote the multiple summation in \eqref{critset}, excluding the factor of $p$. We write 
\begin{equation*}
R:=R_1+R_{-1} 
\end{equation*}
where $R_i$ restricts the summation variables in the definition of $R$ to
\begin{equation} \label{exampcase}
\Big(\frac{n_1}{p} \Big)=\Big(\frac{n_2}{p} \Big)=\Big(\frac{cm}{p} \Big)=i.
\end{equation}
We first consider $R_1$. For $\ell \in \mathbb{F}_p$, define
\begin{equation*}
A_{\ell,c}:=\sum_{M \leq m \leq 2 M} \hspace{0.15cm} \sum_{t^2 \equiv cm \hspace{-0.2cm} \pmod{p}} e \Big( \frac{2 t \ell}{p} \Big), 
\end{equation*}
and 
\begin{equation} \label{newS}
\mathbb{S}_{\ell}:= \Big \{(u,v)  \in (\mathbb{F}_p^{\times})^2:  (u^2, v^2) \pmod{p} \in   [N,2N] \times [N,2N]  \quad \text{ and } \quad u + v \equiv \ell \hspace{-0.2cm} \pmod{p} \Big \}.
\end{equation}
Applying the triangle inequality we obtain 
\begin{align} \label{Rtri} 
R_1&=\frac{1}{2} \sum_{N \leq n_1, n_2 \leq 2N}   \Big | \sum_{M \leq m \leq 2M} \hspace{0.15cm} \sum_{\substack{t \hspace{-0.2cm} \pmod{p} \\ t^2 \equiv cm}} \hspace{0.15cm} \sum_{\substack{u,v \hspace{-0.2cm} \pmod{p} \\ u^2 \equiv n_1 \hspace{-0.2cm} \pmod{p} \\ v^2 \equiv n_2 \hspace{-0.2cm} \pmod{p}}} e \Big( \frac{2 t(u+v)}{p}  \Big)  \Big | \nonumber \\
 & \leq \frac{1}{2} \sum_{\ell \hspace{-0.2cm} \pmod{p}}  | A_{\ell,c} | |\mathbb{S}_{\ell}|.
\end{align}

Observe that the contribution from $\ell \equiv 0 \pmod{p}$ to the right hand side of \eqref{Rtri} is
\begin{equation}  \label{trivzero}
|A_{0,c}| |\mathbb{S}_0| \ll MN.
\end{equation}  
Thus it suffices to consider the right side of \eqref{Rtri} for $\ell \not \equiv 0 \pmod{p}$. For each $\ell \not \equiv 0 \pmod{p}$, we argue that the elements in $\mathbb{S}_{\ell}$ satisfy a strong diophantine property. Recall that $(u,v) \in \mathbb{S}_{\ell}$ means that 
\begin{equation*} 
u+v \equiv \ell  \pmod{p},
\end{equation*}
and $u^2, v^2 \pmod{p}$ lie in the interval $[N,2N]$. After an algebraic manipulation we see that $(u,v) \in \mathbb{S}_{\ell}$ must satisfy
\begin{equation} \label{cong}
\overline{\ell}^2 (u^2-v^2)^2+\ell^2 \equiv 2(u^2+v^2) \pmod{p}.
\end{equation} 
We set
\begin{equation*}
\alpha_{\ell}:=\frac{\overline{\ell}^2}{p} \in \mathbb{Q} / \mathbb{Z} \quad \text{and} \quad \beta_{\ell}:=\frac{\ell^2}{p} \in \mathbb{Q} / \mathbb{Z}.
\end{equation*}
Thus \eqref{cong} implies 
\begin{equation} \label{diophantine}
\| \alpha_{\ell} (u^2-v^2)^2 +\beta_{\ell} \| \leq \frac{8N}{p},
\end{equation}
where $\| \bullet \|$ denotes the distance to the closest integer. Therefore the pairs $(u,v) \in \mathbb{S}_{\ell}$ produce elements of the sequence $\{\alpha_{\ell} n^2\}_{0 \leq n \leq N}$ modulo $1$ and lie in a cluster around $-\beta_{\ell}$. 

It is now sufficient to bound the right side of \eqref{Rtri}. We fix a tuple 
\begin{equation*}
\boldsymbol{\delta}:=(\delta_j) \in (0,1)^6
\end{equation*}
to be chosen later. Let 
\begin{equation} \label{Ldef1}
\mathcal{L}(c,\boldsymbol{\delta}):=\{\ell \in \mathbb{F}^{\times}_p: |A_{\ell,c}| \geq M p^{-\delta_1} \}.
\end{equation} 
Thus \eqref{trivzero} and \eqref{Ldef1} imply 
\begin{equation} \label{Sbd1}
R_1 \ll MN+MN^2 p^{-\delta_1}+\sum_{\ell \in \mathcal{L}(c,\boldsymbol{\delta})} |A_{\ell,c}| |\mathbb{S}_{\ell}|.
\end{equation}
As explained in Section \ref{highlevel}, the distribution of the $\alpha_{\ell} n^2$ (which governs the size of $\mathbb{S}_{\ell}$) is sensitive to the convergents of the continued fraction expansion of $\alpha_{\ell}$. Thus we will consider a partition of $\mathcal{L}(c,\boldsymbol{\delta})$,
\begin{equation} \label{partition}
\mathcal{L}(c,\boldsymbol{\delta})=\mathcal{H}_1(c,\boldsymbol{\delta}) \cup \mathcal{H}_2(c,\boldsymbol{\delta}) \cup \mathcal{H}_3(c,\boldsymbol{\delta}),
\end{equation}
defined below.  By convention, all convergents in the following definitions and arguments are denoted by irreducible fractions. The sets in \eqref{partition} are given by
\begin{multline} \label{H1def}
\mathcal{H}_1(c,\boldsymbol{\delta}):= \Big \{ \ell \in \mathcal{L}(c,\boldsymbol{\delta}) : \quad  \text{for all} \quad 1 \leq h \leq p^{\delta_5}, \\
 h \alpha_{\ell} \text{ has a convergent } \frac{a_{\ell,h}}{b_{\ell,h}} \text{ such that } b_{\ell,h}  \in [p^{\delta_2},p^{\delta_3}]    \Big \};
\end{multline}

\begin{multline} \label{H2def}
\mathcal{H}_2(c,\boldsymbol{\delta}):= \Big \{ \ell \in \mathcal{L}(c, \boldsymbol{\delta}) : 
\quad \text{there exists} \quad 1 \leq h_{\ell} \leq p^{\delta_5} \\
\text{ for which }  h_{\ell} \alpha_{\ell} \text{ has no convergent } \frac{a}{b} \text{ with } b  \in [p^{\delta_2},p^{\delta_4} ]    \Big \}; 
\end{multline}
and
\begin{equation} \label{H3def}
\mathcal{H}_3 (c,\boldsymbol{\delta}):=\mathcal{L}(c,\boldsymbol{\delta}) \setminus (\mathcal{H}_1( c,\boldsymbol{\delta})  \cup \mathcal{H}_2(c,\boldsymbol{\delta}) ). 
\end{equation}

Our argument will require
$\boldsymbol{\delta} \in (0,1)^6$ to satisfy some constraints. We record them here for convenience:
\begin{align} 
\delta_2 <\delta_3&<\delta_4; \label{delta234}  \\
 \delta_6 & < \frac{1}{5}; \label{delta6} \\
16N p^{\delta_2+\delta_5}& <\frac{p}{2} \label{delta25}.
\end{align}

\begin{remark}
The constraint \eqref{delta234} implies that the sets \eqref{H1def} and \eqref{H2def} are well-defined.
The constraints \eqref{delta234}--\eqref{delta25} will be used in the treatment of $\mathcal{H}_3(c,\boldsymbol{\delta})$ in Section \ref{H3}.
\end{remark}

Also note that the elements in $\mathcal{L}(c,\boldsymbol{\delta})$ depend on $c$ by definition. However, the criteria for an element $\ell \in \mathcal{L}(c,\boldsymbol{\delta})$ to belong to a $\mathcal{H}_{j}(c,\boldsymbol{\delta})$ is independent of $c$.

 First we bound $|\mathcal{L}(c,\boldsymbol{\delta})|$ via a second moment estimate of the $A_{\ell,c}$. This will be useful in some of the following arguments. We have 
\begin{equation} \label{momA}
\sum_{\ell \hspace{-0.2cm} \pmod{p}} |A_{\ell,c}|^2= \sum_{M \leq m,m^{\prime} \leq 2 M} \sum_{\substack{t^2 \equiv cm \hspace{-0.2cm} \pmod{p} \\
{t^{\prime}}^2 \equiv cm^{\prime} \hspace{-0.2cm} \pmod{p} }}  \sum_{\ell \hspace{-0.2cm} \pmod{p}} e \Big( \frac{2 \ell(t-t^{\prime})}{p} \Big) \ll pM.
\end{equation}
Using \eqref{momA} we obtain 
\begin{equation} \label{mombd}
|\mathcal{L}(c,\boldsymbol{\delta})| \leq \sum_{\ell \hspace{-0.2cm} \pmod{p}} \Big( \frac{|A_{\ell,c}| p^{\delta_1}}{M} \Big)^2 \ll \frac{p^{1+2 \delta_1}}{M},
\end{equation} 
uniformly in $c$.

\subsection{Treatment of $\mathcal{H}_1(c,\boldsymbol{\delta})$} \label{H1}
We remark that for $\ell \in \mathcal{H}_1(c,\boldsymbol{\delta})$, the sequence 
\begin{equation} \label{seq}
\mathcal{N}_{\ell}:=\{\alpha_{\ell}n^2\}_{1 \leq n \leq N}
\end{equation}
has small discrepancy. Thus to bound their contribution to \eqref{Sbd1}, we obtain an upper bound for $|\mathbb{S}_{\ell}|$. Let $\mathcal{D}(\mathcal{N}_{\ell})$ denote the discrepancy of $\mathcal{N}_{\ell}$. The number of $n \in [1,N]$ such that 
\begin{equation} \label{dio}
\| \alpha_{\ell} n^2+\beta_{\ell} \| \leq \frac{8 N}{p}
\end{equation}
is 
\begin{equation} \label{orig}
\ll \frac{N^2}{p}+N \mathcal{D}(\mathcal{N}_{\ell}).
\end{equation}
 In order to bound $\mathcal{D}(\mathcal{N}_{\ell})$, we consider for each $1 \leq h \leq p^{\delta_5}$,
\begin{equation*}
E_{\ell,h}:=\sum_{1 \leq n \leq N}  e (h \alpha_{\ell} n^2 ).
\end{equation*}
By definition, the continued fraction expansion of $h \alpha_{\ell}$ has a convergent 
\begin{equation*}
\frac{a_{\ell,h}}{b_{\ell,h}} \quad \text{with} \quad b_{\ell,h} \in [p^{\delta_2},p^{\delta_3}].
\end{equation*}
Moreover,
\begin{equation*}
\Big | h \alpha_{\ell}-\frac{a_{\ell,h}}{b_{\ell,h}} \Big | \leq \frac{1}{b_{\ell,h}^2}.
\end{equation*}
Applying Weyl's inequality \cite[Lemma~2.4]{Va} we obtain
\begin{equation} \label{weylapp}
E_{\ell,h} \ll N^{1+\varepsilon}  \Big(p^{-\delta_2}+N^{-1}+\frac{p^{\delta_3}}{N^2} \Big)^{\frac{1}{2}} \ll  N^{\varepsilon}(N p^{-\frac{\delta_2}{2}}+N^{\frac{1}{2}}+p^{\frac{\delta_3}{2}}),
\end{equation}
which is uniform in both $h, \ell$ and $c$. Next, applying the Erd\"{o}s--Turan inequality \cite[(2.42) on pg.~114]{KN} and \eqref{weylapp} we obtain uniformly in $\ell$ and $c$,
\begin{align} \label{discrep}
N \mathcal{D}(\mathcal{N}_{\ell}) &\ll N p^{-\delta_5}+\sum_{1 \leq h \leq p^{\delta_5}} \frac{|E_{\ell,h}|}{h} \nonumber \\
& \ll (Np)^{\varepsilon}( N p^{-\delta_5}+ N p^{-\frac{\delta_2}{2}}+N^{\frac{1}{2}}+p^{\frac{\delta_3}{2}}).
\end{align}
Thus the right side \eqref{orig} is 
\begin{equation} \label{Slbd}
\ll (Np)^{\varepsilon} \Big( \frac{N^2}{p}+N p^{-\delta_5} +N  p^{-\frac{\delta_2}{2}}+N^{\frac{1}{2}} +p^{\frac{\delta_3}{2}} \Big),
\end{equation}
uniformly in $\ell$ and $c$. Observe that for each $\ell \in \mathcal{L}(c,\boldsymbol{\delta})$ (and in particular $ \ell \in \mathcal{H}_1(c,\boldsymbol{\delta})$) and $n \in [1,N]$ satisfying \eqref{dio}, there is at most one element $(u,v) \in \mathbb{S}_{\ell}$ such that (cf. \eqref{diophantine})
\begin{equation*}
u^2-v^2 \equiv \pm n \pmod{p}.
\end{equation*}
The same statement holds when $n=0$. Observing that $|A_{\ell,c}| \ll M$ and using \eqref{mombd} and \eqref{Slbd} we obtain
\begin{equation} \label{H1bd}
\sum_{\ell \in \mathcal{H}_1(c,\boldsymbol{\delta})} |A_{\ell,c}| |\mathbb{S}_{\ell}| \ll (Np)^{\varepsilon} ( p^{2 \delta_1} N^2+N p^{1+2 \delta_1-\delta_5}+N p^{1+2 \delta_1-\frac{\delta_2}{2}}+N^{\frac{1}{2}} p^{1+2 \delta_1}+p^{1+2 \delta_1+\frac{\delta_3}{2}}).
\end{equation}

\subsection{Treatment of $\mathcal{H}_2(c,\boldsymbol{\delta})$} \label{H2}
We draw on the intuition that membership of $\mathcal{H}_2(c,\boldsymbol{\delta})$ is a rare event. Thus we give an upper bound for $|\mathcal{H}_2(c,\boldsymbol{\delta})|$
that is stronger than that implied by \eqref{mombd}.

For each $\ell \in \mathcal{H}_2(c,\boldsymbol{\delta})$, fix $1 \leq h_{\ell} \leq p^{\delta_5}$ such that $h_{\ell} \alpha_{\ell}$ has no convergent 
\begin{equation*}
\frac{a}{b} \quad \text{with}  \quad b \in [p^{\delta_2},p^{\delta_4}].
\end{equation*}
Let $a_{\ell}/b_{\ell}$ be the convergent to $h_{\ell} \alpha_{\ell}$ with $b_{\ell} \in [1,p^{\delta_2})$ maximal and let $a^{*}_{\ell}/b^{*}_{\ell}$ denote the next convergent. Both such convergents exist. Then we must have $b^*_{\ell}>p^{\delta_4}$ and we know that 
\begin{equation*}
\Big | h_{\ell} \alpha_{\ell}-\frac{a_{\ell}}{b_{\ell}} \Big | \leq \frac{1}{b_{\ell} b^{*}_{\ell}}<\frac{1}{b_{\ell} p^{\delta_4}}.
\end{equation*}
Therefore 
\begin{equation} \label{mulineq}
|b_{\ell} h_{\ell} \overline{\ell}^2- p a_{\ell} |<p^{1-\delta_4}.
\end{equation}
Let $\mu_{\ell} \in \Z \cap (-p/2,p/2]$ be such that
\begin{equation} \label{mucong1}
\mu_{\ell} \equiv b_{\ell} h_{\ell} \overline{\ell}^2 \pmod{p}.
\end{equation}
Thus \eqref{mulineq} guarantees
\begin{equation*}
|\mu_{\ell}|<p^{1-\delta_4}.
\end{equation*}

Conversely, consider the congruence 
\begin{equation} \label{convcong}
\mu \equiv b h \overline{\ell}^2 \pmod{p}.
\end{equation}
Given any
\begin{equation*}
\mu \in (-p^{1-\delta_4}, p^{1-\delta_4}), \quad  b \in [1,p^{\delta_2}) \quad  \text{and} \quad h \in [1,p^{\delta_5}), 
\end{equation*}
then these determine $\ell$ in \eqref{convcong} up to sign. Thus
\begin{equation*}
|\mathcal{H}_2(c,\boldsymbol{\delta})| \ll p^{1-\delta_4+\delta_2+\delta_5},
\end{equation*}
uniformly in $c$. Since $|A_{\ell,c}| \ll M$ and $|\mathbb{S}_{\ell}| \ll N$ we obtain 
\begin{equation} \label{S2bd}
\sum_{\ell \in \mathcal{H}_2(c,\boldsymbol{\delta})} |A_{\ell,c}| |\mathbb{S}_{\ell}| \ll MNp^{1-\delta_4+\delta_2+\delta_5}.
\end{equation}

\subsection{Treatment of $\mathcal{H}_3 (c,\boldsymbol{\delta})$} \label{H3}
Recall \eqref{delta234}.
We unpack the definition of $\mathcal{H}_3(c,\boldsymbol{\delta})$. Let $\ell \in \mathcal{H}_3(c,\boldsymbol{\delta})$. Since $\ell \not \in \mathcal{H}_1(c,\boldsymbol{\delta})$, there exists $1 \leq h_{\ell} \leq p^{\delta_5}$ such that $h_{\ell} \alpha_{\ell}$ does not have a convergent 
\begin{equation*} 
\frac{a}{b} \quad \text{with} \quad b \in [p^{\delta_2},p^{\delta_3}].
\end{equation*}
For each $\ell$, fix such a choice $h_{\ell}$. Furthermore, since $\ell \not \in \mathcal{H}_2(c,\boldsymbol{\delta})$, $h_{\ell} \alpha_{\ell}$ is guaranteed to have a convergent 
\begin{equation*} 
\frac{a^{*}_{\ell}}{b^{*}_{\ell}} \quad \text{such that} \quad  b^{*}_{\ell} \in (p^{\delta_3},p^{\delta_4}].
\end{equation*}
Take such a convergent with $b^{*}_{\ell}$ minimal. 

For each $\ell$, denote 
\begin{equation*}
\mathbb{V}_{\ell}:=\Big \{0 \leq n \leq N: \| \alpha_{\ell} n^2+\beta_{\ell}  \| \leq \frac{8 N}{p}   \Big \}.
\end{equation*}
For each $p^{\delta_3} \leq U \leq p^{\delta_4}$ and $0 \leq V \leq N$ we define
\begin{equation*}
\mathcal{E}_c(U,V,\boldsymbol{\delta}):=\{\ell \in \mathcal{H}_3(c,\boldsymbol{\delta}): b^{*}_{\ell} \in [U,2U] \quad \text{and} \quad |\mathbb{V}_{\ell}| \in [V,2V]  \}.
\end{equation*}
Uniformly in $U$, $V$, $c$ and $\boldsymbol{\delta}$ (satisfying \eqref{delta234}--\eqref{delta25})
we have by \eqref{mombd} that
\begin{equation} \label{EUVbd}
|\mathcal{E}_c(U,V,\boldsymbol{\delta})| \ll \frac{p^{1+2 \delta_1}}{M}.
\end{equation}

We prove that the contribution to \eqref{Sbd1} from all 
\begin{equation} \label{speciall}
\ell \in \bigcup_{0 \leq V \leq Np^{-\delta_6}}  \mathcal{E}_c(U,V,\boldsymbol{\delta})=:\mathcal{B}_{\delta_6}
\end{equation}
is small. Applying \eqref{EUVbd}, $|A_{\ell,c}| \ll M$ and the remark following \eqref{Slbd} we see that the contribution to \eqref{Sbd1} from $\ell \in \mathcal{B}_{\delta_6}$ is 
\begin{equation} \label{smallbd}
\ll N p^{1+2 \delta_1-\delta_6}.
\end{equation}

We now consider the case when $V$ is large. Observe that $\mathcal{H}_3(c,\boldsymbol{\delta}) \setminus  \mathcal{B}_{\delta_6}$ can be covered by $O(\log^2 p)$ sets $\mathcal{E}(U,V,\boldsymbol{\delta})$ with
\begin{align} 
 p^{\delta_3} \leq & U \leq p^{\delta_4}; \nonumber \\
 Np^{-\delta_6} \leq & V \leq N.  \label{Vbound}
\end{align}
 From \eqref{smallbd} and the remark following \eqref{Slbd} we obtain 
\begin{equation} \label{refined}
\sum_{\ell \in \mathcal{H}_3(c,\boldsymbol{\delta})} |A_{\ell,c}| |\mathbb{S}_{\ell} | \ll p^{1+2 \delta_1-\delta_6} N+ M \log^2 p \max_{p^{\delta_3} \leq U \leq p^{\delta_4}} \max_{N p^{-\delta_6} \leq V \leq N} V \cdot |\mathcal{E}_c(U,V,\boldsymbol{\delta})|.
\end{equation}  
Thus we need to bound $V \cdot |\mathcal{E}_c(U,V,\boldsymbol{\delta})|$.  For each $\ell \in \mathcal{E}_c(U,V,\boldsymbol{\delta})$, we now construct an algebraic set $\mathfrak{C}_{\ell} \subseteq \mathbb{F}^3_p$ with restricted variables. Arrange the numbers $n_{\ell,j} \in \mathbb{V}_{\ell}$ with order
 \begin{equation} \label{string}
0 \leq n_{\ell,1}<n_{\ell,2}<\cdots<n_{\ell,|\mathbb{V}_{\ell}|} \leq N.
 \end{equation}
 The average consecutive gap between these numbers is 
 \begin{equation*}
 \frac{N}{|\mathbb{V}_{\ell}|} \asymp \frac{N}{V} \ll p^{\delta_6}.
 \end{equation*}
 More than $|\mathbb{V}_{\ell}|/2$ consecutive gaps are less than or equal to $2N/|\mathbb{V}_{\ell}|$. By the pigeonhole principle there exists an integer $1 \leq d_{\ell} \leq 2N/|\mathbb{V}_{\ell}|$ that is repeated as a consecutive gap at least $|\mathbb{V}_{\ell}|^2/4N \gg 1$ times (note that \eqref{rangerecall},
 \eqref{delta6}, and \eqref{Vbound} guarantee that $|\mathbb{V}_{\ell}|^2/4N \gg 1$).  Thus we define
\begin{multline} \label{syscong1}
 \mathfrak{C}_{\ell}:= \big \{ (n,A,B) \in [1,N] \times [-8N,8N]^2 : \overline{\ell}^2 n^2+\ell^2 \equiv A \pmod{p} \\
    \quad \text{and} \quad \overline{\ell}^2 (n+d_{\ell})^2+\ell^2  \equiv B \pmod{p} \big \}.
\end{multline}
We form
\begin{equation*}
\mathfrak{U}_{c}(U,V,\boldsymbol{\delta}):= \bigcup_{\ell \in \mathcal{E}_c(U,V,\boldsymbol{\delta})} \{\ell \} \times \mathfrak{C}_{\ell} \subseteq \mathbb{F}^4_p,
\end{equation*}
and study this object now.

The above discussion implies the pointwise bound $|\mathfrak{C}_{\ell}| \gg V^2/N$, so 
 \begin{equation} \label{intuit}
| \mathfrak{U}_{c}(U,V,\boldsymbol{\delta}) | \gg \frac{V^2 |\mathcal{E}_c(U,V,\boldsymbol{\delta})|}{N}.
 \end{equation}
Thus it suffices to establish an upper bound for $|\mathfrak{U}_c(U,V,\boldsymbol{\delta})|$. We count the number of $Q:=(\ell;n,A,B) \in  \mathfrak{U}_{c}(U,V,\boldsymbol{\delta})$ with $A \equiv B \pmod{p}$ and $A \not \equiv B \pmod{p}$ separately.

Given $\ell \in \mathcal{E}_c(U,V,\boldsymbol{\delta})$ and $A \equiv B \pmod{p}$, an algebraic manipulation determines at most one possible $Q$. Thus \eqref{EUVbd} implies that there are 
\begin{equation} \label{AequivB}
\ll  \frac{p^{1+2 \delta_1}}{M}
\end{equation}
such $Q$. 

The rest of the argument treats the case $A \not \equiv B \pmod{p}$. Recall the constraint 
\eqref{delta25}. Let 
 \begin{equation*}
\mathfrak{T}_c(U,V,\boldsymbol{\delta}):=\Big \{g+pr \in \mathbb{Z}:  |r| \leq \frac{36 N^2}{UV}+1, \quad  |g| \leq 16  p^{\delta_2+\delta_5} N \quad \text{and} \quad g \neq 0   \Big \}.
 \end{equation*}
 be a set containing a union of short arithmetic progressions. We will construct a map
 \begin{equation*}
 t_{\bullet}: Q \in \mathfrak{U}_{c}(U,V,\boldsymbol{\delta}) \quad  (A \not \equiv B \pmod{p}) \rightarrow t_Q \in \mathfrak{T}_c(U,V,\boldsymbol{\delta}),
 \end{equation*}
 whose fibers have size $O(p^{\varepsilon})$ for any fixed $\varepsilon>0$.
These facts will imply
\begin{equation} \label{countineq}
| \mathfrak{U}_c(U,V,\boldsymbol{\delta})| \ll  p^{\varepsilon} |\mathfrak{T}_c(U,V,\boldsymbol{\delta})|+\frac{p^{1+2 \delta_1}}{M}.
\end{equation}

Starting with $Q \in \mathfrak{U}_{c}(U,V,\boldsymbol{\delta})$, subtracting the congruences in \eqref{syscong1} yields 
\begin{equation} \label{subtract}
\overline{\ell}^2 (2n d_{\ell}+d_{\ell}^2) \equiv B-A \not \equiv 0 \pmod{p}.
\end{equation}
Recall that for each $\ell \in \mathcal{H}_{3}(c,\boldsymbol{\delta})$, we fixed a choice $h_{\ell} \in [1,p^{\delta_5}]$ such that $h_{\ell} \alpha_{\ell}=h_{\ell} \overline{\ell}^2/p$ has no convergent 
\begin{equation*} 
\frac{a}{b} \quad \text{with} \quad b \in [p^{\delta_2},p^{\delta_3}],
\end{equation*}
and has a convergent
\begin{equation*}
\frac{a^{*}_{\ell}}{b^{*}_{\ell}} \quad \text{such that} \quad  b^{*}_{\ell} \in (p^{\delta_3},p^{\delta_4}],
\end{equation*}
with $b^*_{\ell}$ minimal. Moreover $\ell \in \mathcal{E}_{c}(U,V,\boldsymbol{\delta})$ restricts $b^*_{\ell} \in [U,2U]$. Let $a_{\ell}/b_{\ell}$ denote the convergent to $h_{\ell} \alpha_{\ell}$ with $b_{\ell} \in [1,p^{\delta_2})$ maximal.  Thus $a_{\ell}/b_{\ell}$ and $a^*_{\ell}/b^*_{\ell}$ are consecutive convergents. 
Let $\mu_{\ell} \in \Z \cap (-p/2,p/2]$ be such that
\begin{equation*}
\mu_{\ell} \equiv b_{\ell} h_{\ell} \overline{\ell}^2 \pmod{p}.
\end{equation*}
Note that $\mu_{\ell} \not \equiv 0 \pmod{p}$. By a similar argument to the one contained in Section \ref{H2} we have 
\begin{equation*}
|\mu_{\ell}|  \leq \frac{p}{U}.
\end{equation*}
Multiplying \eqref{subtract} by $b_{\ell} h_{\ell}$ we obtain 
\begin{equation} \label{multcong}
\mu_{\ell} (2nd_{\ell}+d_{\ell}^2) \equiv b_{\ell} h_{\ell} (B-A) \pmod{p}.
\end{equation}
Writing \eqref{multcong} as an equation of integers we have 
\begin{equation*}
\mu_{\ell} (2nd_{\ell}+d_{\ell}^2)=pr + b_{\ell} h_{\ell} (B-A) \quad \text{for some} \quad r \in \mathbb{Z}.
\end{equation*}
Observe that 
\begin{equation*}
0<|b_{\ell} h_{\ell} (B-A)| \leq 16 p^{\delta_2+\delta_5} N \quad \text{and} \quad \big |\mu_{\ell} (2nd_{\ell}+d_{\ell}^2) \big| \leq \frac{36N^2 p}{UV}+1
\end{equation*}
Thus
\begin{equation*}
t_Q:=\mu_{\ell} (2nd_{\ell}+d_{\ell}^2) \in \mathfrak{T}_c(U,V,\boldsymbol{\delta}).
\end{equation*}

Suppose we are given $t=pr+g \in  \mathfrak{T}_c(U,V,\boldsymbol{\delta})$. Since $t \neq 0$ (by \eqref{delta25}), $g \neq 0$ and $A \not \equiv B \pmod{p}$, the number of tuples
\begin{equation} \label{bigtuple}
(\mu,n,d,b,h,B-A) \in \mathbb{Z} \times [0,N] \times  \Big [1, \frac{2N}{V} \Big] \times [1,p^{\delta_2}] \times [1,p^{\delta_5}] \times [-16N,16N]
\end{equation}
satisfying 
\begin{equation*}
t=\mu d (2n+d) \quad \text{and} \quad g=bh(B-A)
\end{equation*}
is at most $O(p^{\varepsilon})$ by divisor considerations. A tuple in \eqref{bigtuple} then determines two values of $\ell$ mod $p$ using
\begin{equation*}
\mu \equiv b h \overline{\ell}^2 \pmod{p}.
\end{equation*}
Thus there are at most $O(p^{\varepsilon})$ valid $(n,d,\ell)$ for a given $t$. Each $3$-tuple together with the equations defining $\mathfrak{C}_{\ell}$ in \eqref{syscong1} determine at most one pair $(A,B) \in [-8N,8N]^2$. Thus there are at most $O(p^{\varepsilon})$ quadruples $(\ell;n,A,B) \in \mathfrak{U}_c(U,V,\boldsymbol{\delta})$ such that $t_Q=t$ and \eqref{countineq} holds.

Combining \eqref{intuit} and \eqref{countineq} gives 
\begin{equation*}
 \max_{p^{\delta_3} \leq U \leq p^{\delta_4}} \max_{N p^{-\delta_6} \leq V \leq N} V \cdot |\mathcal{E}_c(U,V,\boldsymbol{\delta})|  \ll (Np)^{\varepsilon} \Big( \frac{p^{1+2 \delta_1+\delta_6}}{M}+N^2 p^{\delta_2+\delta_5+2 \delta_6-\delta_3}+Np^{\delta_2+\delta_5+\delta_6} \Big).
\end{equation*}
Inserting this into \eqref{refined} we obtain 
\begin{equation} \label{S3bd}
\sum_{\ell \in \mathcal{H}_3(c,\boldsymbol{\delta})} |A_{\ell,c}| |\mathbb{S}_{\ell} | \ll (Np)^{\varepsilon}(p^{1+2 \delta_1-\delta_6} N+p^{1+2 \delta_1+\delta_6}+N^2 M p^{\delta_2+\delta_5+2 \delta_6-\delta_3}+MN p^{\delta_2+\delta_5+\delta_6}). 
\end{equation}

Inserting \eqref{H1bd}, \eqref{S2bd} and \eqref{S3bd} into \eqref{Sbd1} we obtain
\begin{align} 
R_1 & \ll  
(Np)^{\varepsilon} \big(MN+MN^2 p^{-\delta_1}+  p^{2 \delta_1} N^2+N p^{1+2 \delta_1-\delta_5}+N p^{1+2 \delta_1-\frac{\delta_2}{2}} \nonumber \\
&+N^{\frac{1}{2}} p^{1+2 \delta_1} +p^{1+2 \delta_1+\frac{\delta_3}{2}}  + MNp^{1-\delta_4+\delta_2+\delta_5}+Np^{1+2 \delta_1-\delta_6} \nonumber  \\ 
&+p^{1+2 \delta_1+\delta_6}+MN^2  p^{\delta_2+\delta_5+2 \delta_6-\delta_3}+MN p^{\delta_2+\delta_5+\delta_6} \big).  \label{combined}
\end{align}

The same argument above can be applied to bound $R_{-1}$ by the right hand side of \eqref{combined}.  One fixes a non-zero non-quadratic residue $j$ modulo $p$ and sees that \eqref{exampcase} is equivalent to 
\begin{equation*}
\Big(\frac{\overline{j} n_1}{p} \Big)=\Big(\frac{\overline{j} n_2}{p} \Big)=\Big(\frac{jcm}{p} \Big)=1.
\end{equation*}
Thus the analogue of \eqref{Rtri} is
\begin{align*}
R_{-1}&= \frac{1}{2} \sum_{N \leq n_1, n_2 \leq 2N}   \Big | \sum_{M \leq m \leq 2M} \hspace{0.15cm} \sum_{\substack{t \hspace{-0.2cm} \pmod{p} \\ t^2 \equiv c j m}} \hspace{0.15cm} \sum_{\substack{u,v \hspace{-0.2cm} \pmod{p} \\ u^2 \equiv \overline{j} n_1 \hspace{-0.2cm} \pmod{p} \\ v^2 \equiv  \overline{j} n_2 \hspace{-0.2cm} \pmod{p}}} e \Big( \frac{2 t(u+v)}{p}  \Big)  \Big | \\
& \leq  \frac{1}{2} \sum_{\ell \hspace{-0.2cm} \pmod{p}}  | A_{\ell,c j} | |\mathbb{S}_{\ell,j}|,
\end{align*}
where 
\begin{equation*}
\mathbb{S}_{\ell,j}:= \Big \{(u,v)  \in (\mathbb{F}_p^{\times})^2:  (j u^2, j v^2) \pmod{p} \in   [N,2N] \times [N,2N]  \quad \text{ and } \quad u + v \equiv \ell \hspace{-0.2cm} \pmod{p} \Big \}.
\end{equation*}
Repeating the algebraic manipulation with the linear congruence in the definition of $\mathbb{S}_{\ell,j}$ we obtain
\begin{equation*}
\| \alpha_{j,\ell} (u^2-v^2)^2 +\beta_{j,\ell} \| \leq \frac{8N}{p},
\end{equation*}
where 
\begin{equation*}
\alpha_{\ell,j}:=\frac{j \overline{\ell}^2}{p} \in \mathbb{Q} / \mathbb{Z} \quad \text{and} \quad \beta_{\ell,j}:=\frac{j \ell^2}{p} \in \mathbb{Q} / \mathbb{Z}.
\end{equation*}
Then \eqref{Sbd1} becomes 
\begin{equation} \label{Sbd12}
R_{-1} \ll MN+MN^2 p^{-\delta_1}+\sum_{\ell \in \mathcal{L}(c j,\boldsymbol{\delta})} |A_{\ell,cj}| |\mathbb{S}_{\ell,j}|
\end{equation}
and we consider the partition 
\begin{equation*}
\mathcal{L}(c j,\boldsymbol{\delta})=\mathcal{H}_1(c j,\boldsymbol{\delta}) \cup \mathcal{H}_2(c j,\boldsymbol{\delta}) \cup \mathcal{H}_3(c j,\boldsymbol{\delta}),
\end{equation*}
where one replaces $\alpha_{\ell}$ (resp. $\beta_{\ell}$) by $\alpha_{\ell,j}$ (resp. $\beta_{\ell,j}$) in the definitions of the $\mathcal{H}_i(c,\boldsymbol{\delta})$ occurring in \eqref{H1def}--\eqref{H3def}.  One can then repeat the arguments in Sections \ref{H1}--\ref{H3} making the necessary modifications.

For $p^{\frac{1}{2}-\frac{1}{10}} \leq N \leq p^{\frac{1}{2}+\frac{1}{10}}$, we see that 
\begin{equation*}
\boldsymbol{\delta}:=\Big(\frac{11}{288}, \frac{11}{48}, \frac{25}{36},\frac{407}{432},\frac{11}{96},\frac{11}{96} \Big)
\end{equation*}
satisfies \eqref{delta234}--\eqref{delta25}, and is sufficient to obtain Theorem \ref{critrangethm} (after multiplication by $p$, cf. \eqref{critset}). For aesthetic reasons we take a larger estimate (i.e. all denominators multiples of $27$).
\begin{remark} \label{mod4}
Observe that the above argument can be modified so that the estimate in Theorem \ref{critrangethm} holds when each $n_1,n_2$ and $m$ are restricted to fixed congruence classes modulo $4$.
\end{remark}

\end{proof}

\bibliographystyle{amsalpha}
\bibliography{twistedREV}

\end{document}